\documentclass[a4paper]{article}
\usepackage{amsmath, amsthm, amssymb, amsfonts}
\usepackage{cite}
\usepackage[normalem]{ulem}
\usepackage{hyperref}
\usepackage{mathtools}
\usepackage{verbatim} 
\usepackage{longtable} 
\usepackage{mathtools}

\usepackage{caption}
\usepackage{tikz-cd}
\usetikzlibrary{arrows}

\usepackage{titling}

\usepackage{multirow}
\usepackage{array}

\theoremstyle{plain}
\newtheorem{thm}{Theorem}
\newtheorem{cor}{Corollary}
\newtheorem{corstar}[cor]{Corollary*}

\newtheorem{lemma}{Lemma}

\newtheorem{lemmastar}[lemma]{Lemma*}

\newtheorem{prop}{Proposition}
\newtheorem{propstar}[prop]{Proposition*}
\newtheorem{conjecture}{Conjecture}  

\theoremstyle{definition}

\theoremstyle{remark}
\newtheorem{rmk}{Remark}

\newcommand{\BC}{{\mathbb{C}}}

\newcommand{\BH}{{\mathbb{H}}}

\newcommand{\BQ}{{\mathbb{Q}}}

\newcommand{\BZ}{{\mathbb{Z}}}

\newcommand{\CA}{{\mathcal A}}
\newcommand{\CB}{{\mathcal B}}
\newcommand{\CC}{{\mathcal C}}

\newcommand{\CE}{{\mathcal E}}

\newcommand{\CK}{{\mathcal K}}

\newcommand{\CO}{{\mathcal O}}

\newcommand{\CR}{{\mathcal R}}

\newcommand{\CX}{{\mathcal X}}

\newcommand{\Fm}{{\mathfrak{m}}}

\newcommand{\unit}{{\mathbf{1}}}

\newcommand{\blangle}{\big\langle}
\newcommand{\brangle}{\big\rangle}
\newcommand{\Blangle}{\Big\langle}
\newcommand{\Brangle}{\Big\rangle}

\newcommand{\pt}{{\mathsf{p}}}

\newcommand{\id}{{\mathrm{id}}}

\DeclareFontFamily{OT1}{rsfs}{}
\DeclareFontShape{OT1}{rsfs}{n}{it}{<-> rsfs10}{}
\DeclareMathAlphabet{\curly}{OT1}{rsfs}{n}{it}

\newcommand\Hom{\operatorname{Hom}}
\newcommand\End{\operatorname{End}}
\newcommand\Aut{\operatorname{Aut}}
\newcommand{\p}{\mathbb{P}}

\newcommand{\Mbar}{{\overline M}}
\newcommand{\ev}{{\mathrm{ev}}}
\newcommand{\QMod}{\mathrm{QMod}}
\newcommand{\QJ}{\mathrm{QJac}}
\newcommand{\QJac}{\mathrm{QJac}}
\newcommand{\Jac}{\mathrm{Jac}}

\newcommand{\AHM}{\mathrm{AHM}}
\newcommand{\AHJ}{\mathrm{AHJ}}
\newcommand{\im}{\mathrm{Im}}

\newcommand{\Pic}{\mathop{\rm Pic}\nolimits}

\newcommand{\HH}{\mathsf{H}}
\newcommand{\A}{\mathsf{A}}
\newcommand{\kk}{\mathsf{k}}
\newcommand{\1}{\mathsf{1}}
\newcommand{\T}{\mathsf{T}}
\newcommand{\F}{\mathsf{F}}

\newcommand{\vir}{\text{vir}}

\allowdisplaybreaks[1]

\begin{document}
\setlength{\droptitle}{-16pt}

\title{Gromov--Witten theory of elliptic fibrations: Jacobi forms and holomorphic anomaly equations}
\author{Georg Oberdieck and Aaron Pixton}
\date{\today}
\maketitle


\begin{abstract}
We conjecture that the relative Gromov--Witten potentials
of elliptic fibrations are (cycle-valued) lattice
quasi-Jacobi forms and satisfy a holomorphic anomaly equation.
We prove the conjecture for the rational elliptic surface in all genera and curve classes numerically. The generating series are quasi-Jacobi forms for the lattice $E_8$.
We also show the compatibility of the conjecture with the degeneration formula.
As Corollary we deduce that the Gromov--Witten potentials of the Schoen Calabi--Yau threefold
(relative to $\p^1$) are $E_8 \times E_8$ quasi-bi-Jacobi forms and satisfy
a holomorphic anomaly equation. This yields a partial verification of the BCOV holomorphic anomaly equation for Calabi--Yau threefolds.
For abelian surfaces the holomorphic anomaly equation
is proven numerically in primitive classes.
The theory of lattice quasi-Jacobi forms is reviewed.

In the Appendix the conjectural holomorphic anomaly equation
is expressed as a 
matrix action on the space of (generalized) cohomological field theories.
The compatibility of the matrix action with the Jacobi Lie algebra is proven.
Holomorphic anomaly equations for K3 fibrations are discussed
in an example.
\end{abstract}

\setcounter{tocdepth}{1} 
\tableofcontents
\setcounter{section}{-1}

\section{Introduction}
\subsection{Holomorphic anomaly equations}
Gromov--Witten invariants of a non-singular compact Calabi--Yau threefold $X$ are defined by the integrals
\[ \mathsf{N}_{g, \beta}
= \int_{[ \Mbar_{g}(X, \beta) ]^{\text{vir}}} 
1
\]
where $\Mbar_{g}(X, \beta)$ is the
moduli space of stable maps from connected genus~$g$ curves to $X$ of degree $\beta \in H_2(X,\BZ)$,
and $[ \, - \, ]^{\text{vir}}$ is its virtual fundamental class.
Mirror symmetry \cite{BCOV, Hos, ASYZ} makes the following predictions about the genus~$g$ potentials
\[ \F_g(q) = \sum_{\beta} \mathsf{N}_{g, \beta} q^{\beta}. \]

\begin{enumerate}
 \item[(i)] There exists a finitely generated subring of quasi-modular objects
 \[ \CR \subset \BQ[[ q^{\beta} ]] \]
 (depending on $X$) which contains all $\F_g(q)$.
 \item[(ii)] The series $\F_g(q)$ satisfy \emph{holomorphic anomaly equations},
 i.e. recursive formulas for the derivative of the modular completion of $\F_g$ with respect to the non-holomorphic variables.\footnote{
 In many cases $\CR$ can be described explicitly by generators and relations, and (ii) is equivalent to formulas for the formal derivative of $\F_g$ with respect to distinguished generators of the ring.}
\end{enumerate}
Here, the precise modular interpretation of $\F_g(q)$ is part of the problem and not well understood in general.
Mathematically, the predictions (i, ii) are not known yet for any (compact) Calabi--Yau threefold.\footnote{
The (non-compact) local $\p^2$ case was recently established in \cite{LP}.}

\subsection{The Schoen Calabi--Yau threefold} \label{intro:schoen_calabiyau}
A \emph{rational elliptic surface} $R \to \p^1$ is the successive blowup of $\p^2$ along
the base points of a pencil of cubics containing a smooth member.
Its second cohomology group admits the splitting
\[ H^2(R,\BZ) = \mathrm{Span}_{\BZ}(B,F) \overset{\perp}{\oplus} E_8(-1) \]
where $B,F$ are the classes of a fixed section and a fiber respectively. Let also
\[ W = B + \frac{1}{2} F . \]

Let $R_1, R_2$ be rational elliptic surfaces with disjoint sets of basepoints of singular fibers.
The \emph{Schoen Calabi--Yau threefold} \cite{Schoen} is the fiber product
\[ X = R_1 \times_{\p^1} R_2 \,. \]
We have the commutative diagram of fibrations
\begin{equation} \label{dfsdfsaaad}
\begin{tikzcd}
& X \ar[swap]{dl}{\pi_2} \ar{dr}{\pi_1} \ar{dd}{\pi} & \\
R_1 \ar[swap]{dr}{p_1} & & R_2 \ar{dl}{p_2} \\
& \p^1 &
\end{tikzcd}
\end{equation}
where $\pi_i$ are the elliptic fibrations induced by $p_i : R_i \to \p^1$. Let
\[ W_i, F_i \in H^2(R_i, \BQ), \quad E^{(i)}_{8}(-1) \subset H^2(R_i, \BZ) \]
denote the classes $W,F$ and the $E_8$-lattice on $R_i$ respectively. We have
\[
H^2(X,\BQ) = \langle D \rangle \oplus
\Big( \langle \pi_1^{\ast} W_2 \rangle \oplus \pi_1^{\ast} E_8^{(2)}(-1)_{\BQ} \Big) \oplus \Big( \langle \pi_2^{\ast} W_1 \rangle \oplus \pi_2^{\ast} E_8^{(1)}(-1)_{\BQ} \Big)
\]
where we let $\langle \cdot \rangle$ denote the $\BQ$-linear span and $D$ is the class of a fiber of $\pi$. 

For all $(g,k) \notin \{ (0,0), (1,0) \}$ define\footnote{
The cases $(g,k) \in \{ (0,0), (1,0) \}$ are excluded since
$\mathsf{N}_{g,0}$ is not defined for $g \in \{ 0, 1 \}$.}
the \emph{$\pi$-relative Gromov--Witten potential}
\begin{equation} \label{Defn_Fgk}
\F_{g,k}(\mathbf{z}_1, \mathbf{z}_2, q_1, q_2)
= \sum_{\pi_{\ast} \beta = k [\p^1]} \mathsf{N}_{g, \beta}
q_1^{W_1 \cdot \beta} q_2^{W_2 \cdot \beta} e( \mathbf{z}_1 \cdot \beta ) e( \mathbf{z}_2 \cdot \beta )
\end{equation}
where the sum is over all curve classes $\beta \in H_2(X,\BZ)$ of degree $k$ over $\p^1$,
we have suppressed pullbacks by $\pi_i$,
we write $e(x) = \exp(2 \pi i x)$ for all $x \in \BC$, and
\[ \mathbf{z}_i \in E_8^{(i)}(-1) \otimes \BC \]
is the (formal) coordinate on the $E_8$ lattice of $R_i$.

A (weak) $E_8$-Jacobi form is a holomorphic function of variables
\[ q = e^{2 \pi i \tau}, \tau \in \BH \quad \text{ and }
\quad \mathbf{z} \in E_8 \otimes \BC \]
which is semi-invariant under the action of the Jacobi group,
invariant under the Weyl group of $E_8$
and satisfies a growth condition at the cusp;
we refer to Section~\ref{Section_Lattice_Jacobi_forms} for an introduction to Jacobi forms.
The ring of weak $E_8$-Jacobi forms $\Jac_{E_8}$ carries a bigrading
by weight $\ell \in \BZ$ and index $m \in \BZ_{\geq 0}$,
\[ \Jac_{E_8} = \bigoplus_{\ell,m} \Jac_{E_8, \ell,m}. \]
Recall the second Eisenstein series
\[ C_2(q) = -\frac{1}{24} + \sum_{n \geq 1} \sum_{d|n} d q^n. \]
By assigning $C_2$ index $0$ and weight $2$ we have the bigraded extension
\begin{equation} \label{fdfsdfsd}
\widetilde{\Jac}_{E_8} = \Jac_{E_8}[C_2]
= \bigoplus_{\ell, m} \widetilde{\Jac}_{E_8, \ell, m}.
\end{equation}
The ring \eqref{fdfsdfsd} in the variables $q = q_i$ and $\mathbf{z}_i \in E_8^{(i)}$ is denoted by
$\widetilde{\Jac}^{(q_i, \mathbf{z}_i)}_{E_8}$.

Recall also the modular discriminant
\[ \Delta(q) = q \prod_{m \geq 1} (1-q^m)^{24}. \]

We prove the following basic quasi-modularity result.
\begin{thm} \label{Thm1} 
Every relative potential $\F_{g,k}$ is a $E_8 \times E_8$ bi-quasi-Jacobi form:
\[ \F_{g,k}(\mathbf{z}_1, \mathbf{z}_2, q_1, q_2)\ \in\ 
\frac{1}{\Delta(q_1)^{k/2}} \widetilde{\Jac}_{E_8, \ell,k}^{(q_1, \mathbf{z}_1)}
 \otimes \frac{1}{\Delta(q_2)^{k/2}} \widetilde{\Jac}_{E_8, \ell,k}^{(q_2, \mathbf{z}_2)}
\]
where $\ell = 2g-2+6k$.
\end{thm}
\vspace{7pt}

The appearance of $E_8 \times E_8$ bi-quasi-Jacobi forms
is in perfect agreement with
predictions made using mirror symmetry \cite{Sakai, HSS, HST}.

The elements in $\Jac_{E_8}$ are Jacobi forms and therefore modular objects.
The only source of non-modularity in $\widetilde{\Jac}_{E_8}$ and hence in $\F_{g,k}$ arises from the
strictly quasi-modular series $C_2(q)$.
We state a holomorphic anomaly equation which determines the dependence on $C_2$ explicitly.

Identify the lattice $E_8^{(i)}$ with the pair $(\BZ^8, Q_{E_8})$ where $Q_{E_8}$ is
the (positive definite) Cartan matrix of $E_8$, see Section~\ref{Section_E8_lattice}.
For $j \in \{1,2\}$ consider the differentiation operators with respect to $q_j$ and $\mathbf{z}_j = (z_{j,1}, \ldots, z_{j,8})$:
\[ D_{q_j} = \frac{1}{2 \pi i} \frac{d}{d \tau_j} = q_j \frac{d}{dq_j}, \quad D_{z_{j,\ell}} = \frac{1}{2 \pi i} \frac{d}{d z_{j,\ell}}. \]

\begin{thm} \label{Thm2} Every $\F_{g,k}$ satisfies the holomorphic anomaly equation
\begin{multline*}
\frac{d}{dC_2(q_2)} \F_{g,k}
=
\left( 2 k D_{q_1} - \sum_{i,j=1}^{8} \left( Q_{E_8}^{-1} \right)_{ij} D_{\mathbf{z}_{1,i}} D_{\mathbf{z}_{1,j}} + 24k C_2(q_1) \right) \F_{g-1,k} \\
+
\sum_{\substack{g=g_1+g_2 \\ k = k_1 + k_2}}
\left(
2 k_1 \F_{g_1,k_1} \cdot D_{q_1} \F_{g_2,k_2} - \sum_{i,j=1}^{8} \left( Q_{E_8}^{-1} \right)_{ij} D_{\mathbf{z}_{1,i}}(\F_{g_1,k_1}) \cdot  D_{\mathbf{z}_{1,j}}(\F_{g_2,k_2})
\right).
\end{multline*}
\end{thm}
\vspace{7pt}

Since $X$ is symmetric in $R_1, R_2$ up to a deformation,
the potentials $\F_{g,k}$ are symmetric under interchanging $(\mathbf{z}_i,q_i)$:
\[ \F_{g,k}(\mathbf{z}_1, \mathbf{z}_2, q_1, q_2) = \F_{g,k}(\mathbf{z}_2, \mathbf{z}_1, q_2, q_1). \]
Hence Theorem~\ref{Thm2} determines also the dependence of $\F_{g,k}$ on $C_2(q_1)$.

Theorems~\ref{Thm1} and~\ref{Thm2} show quasi-modularity and the holomorphic anomaly equation for the Gromov--Witten potentials of $X$ \emph{relative} to $\p^1$. This provides a partial verification of the absolute case (i,ii).
It also leads to modular properties when the Gromov--Witten potentials
are summed over the genus as follows.
Consider the topological string partition function
(i.e. the generating series of disconnected Gromov--Witten invariants)
of the Schoen geometry
\[ \mathsf{Z}(t, u, \mathbf{z}_1, \mathbf{z}_2, q_1, q_2)
=
\exp\left( \sum_{g \geq 0} \sum_{\beta > 0} \mathsf{N}_{g,\beta}
u^{2g-2} t^{D \cdot \beta}
q_1^{W_1 \cdot \beta} q_2^{W_2 \cdot \beta} e( \mathbf{z}_1 \cdot \beta ) e( \mathbf{z}_2 \cdot \beta ) \right).
\]
Under a variable change, $\mathsf{Z}$ is the generating series of Donaldson--Thomas / Pandharipande--Thomas invariants of the threefold $X$ \cite{PaPix2}.
For any curve class $\alpha \in H_2(R_1, \BZ)$
of some degree $k$ over the base $\p^1$
consider the coefficient
\[
\mathsf{Z}_{\alpha}(u, \mathbf{z}_2, q_2)
=
\Big[ \mathsf{Z}(t, u, \mathbf{z}_1, \mathbf{z}_2, q_1, q_2) \Big]_{t^k q^{W_1 \cdot \alpha} e(\mathbf{z}_1 \cdot \alpha)} .
\]
We write $(\mathbf{z}, q)$ for $(\mathbf{z}_2, q_2)$,
and work under the variable change $u = 2 \pi z$ and $q = e^{2 \pi i \tau}$.
We then have the following.

\begin{cor} \label{Cor_Schoen}
Under the variable change $u = 2 \pi z$ and $q = e^{2 \pi i \tau}$
the series $\mathsf{Z}_{\alpha}(z, \mathbf{z}, \tau)$ satisfies
the modular transformation law of Jacobi forms of weight $-6$ and index
$( \frac{1}{2} \langle \alpha - c_1(R_1), \alpha \rangle ) \oplus \frac{k}{2} Q_{E_8}$,
that is for all $\gamma = \binom{a\ b}{c\ d} \in \mathrm{SL}_2(\BZ)$
\begin{multline*}
\mathsf{Z}_{\alpha}\left( \frac{z}{c \tau + d},
\frac{\mathbf{z}}{c \tau + d},
\frac{a \tau + b}{c \tau + d} \right) \\
=
\xi(\gamma)^{k + 1} (c \tau + d)^{-6}
e\left( \frac{c}{2 (c \tau + d)} \Big[ k \mathbf{z}^T Q_{E_8} \mathbf{z} +
z^2 \langle \alpha - c_1(R_1), \alpha \rangle \Big] \right)
\mathsf{Z}_{\alpha}(z,\mathbf{z}, \tau)
\end{multline*}
where $\xi(\gamma) \in \{ \pm 1 \}$ 
is determined by
$\Delta^{\frac{1}{2}}(\gamma \tau) = \xi(\gamma) (c \tau + d)^{6} \Delta^{\frac{1}{2}}(\tau)$.
\end{cor}

By Theorem~\ref{Thm1} the series $\mathsf{Z}_{\alpha}$ also satisfies
the elliptic transformation law of Jacobi forms 
in the variable $\mathbf{z}$.
The elliptic transformation law in the genus variable $u$
is conjectured by Huang--Katz--Klemm \cite{HKK}
and corresponds to the expected symmetry of Donaldson--Thomas invariants
under the  Fourier--Mukai transforms by the Poincar\'e sheaf of
$\pi_2$, see \cite{OS1}. Hence conjecturally we find that
$\mathsf{Z}_{\alpha}$ is a meromorphic Jacobi form
(of weight and index as in Corollary~\ref{Cor_Schoen}).

We end our discussion with two concrete examples.
Expend the partition function $\mathsf{Z}$
by the degree over the base $\p^1$:
\[
\mathsf{Z}(t, u, \mathbf{z}_1, \mathbf{z}_2, q_1, q_2)
=
\sum_{k = 0}^{\infty} \mathsf{Z}_k(u, \mathbf{z}_1, \mathbf{z}_2, q_1, q_2) t^k.
\]
By a basic degeneration argument in degree $0$ we have 
\[ \mathsf{Z}_0 = \frac{1}{\Delta(q_1)^{\frac{1}{2}} \Delta(q_2)^{\frac{1}{2}}}. \]
In degree~$1$ the Igusa cusp form conjecture \cite[Thm.1]{HAE}
and an analysis of the sections of $\pi : X \to \p^1$ yields
\[ \mathsf{Z}_1 = \frac{\Theta_{E_8}(\mathbf{z}_1, q_1) \Theta_{E_8}(\mathbf{z}_2, q_2)}{\chi_{10}(e^{iu},q_1, q_2)} \]
where $\chi_{10}$ is the Igusa cusp form, a Siegel modular form, defined by
\[ \chi_{10}(p,q_1, q_2) = p q_1 q_2 \prod_{(k,d_1, d_2)>0} (1-p^k q_1^{d_1} q_2^{d_2})^{c(4 d_1 d_2-k^2)} \]
(with $c(n)$ being coefficients of a certain $\Gamma_0(4)$-modular form, see
\cite[Sec.0.2]{HAE}), and
\[ \Theta_{E_8}(\mathbf{z}, \tau) = 
\sum_{\gamma \in \BZ^{8}} q^{\frac{1}{2} \gamma^T Q_{E_8} \gamma} e\left(  \mathbf{z}^T Q_{E_8} \gamma \right),
\]
is the Riemann theta function of the $E_8$-lattice.
The general relationship of $\mathsf{Z}_k$
to Siegel modular forms for $k > 1$ is yet to be found.


\subsection{Beyond Calabi--Yau threefolds and the proof}
Recently it became clear
that we should expect properties (i, ii) not only for Calabi--Yau threefolds
but also for varieties $X$ (of arbitrary dimension) which are Calabi--Yau relative to a base $B$,
i.e. those which admit a fibration
\[ \pi : X \to B \]
whose generic fiber has trivial canonical class.
The potential $\F_g(q)$ is replaced here by a $\pi$-relative Gromov--Witten potential
which takes values in cycles on $\Mbar_{g,n}(B,\kk)$, the moduli space of stable maps to the base.
In this paper we conjecture and develop
such a theory for \emph{elliptic fibrations with section}.
Our main theoretical result is a conjectural link
between the Gromov--Witten theory of elliptic fibrations
and the theory of lattice quasi-Jacobi forms.
This framework allows us to conjecture a holomorphic anomaly equation.\footnote{See Section~\ref{Section_Elliptic_fibrations_and_conjectures} for details
on the conjectures.}

The elliptic curve (or more generally, trivial elliptic fibrations)
is the simplest case of our conjecture
and was proven in \cite{HAE}.
In this paper we prove the following new cases (see Section~\ref{Subsection_RES_Statement_of_results}):
\begin{enumerate}
 \item[(a)] 
 The $\p^1$-relative Gromov--Witten potentials of the rational elliptic surface are $E_8$-quasi-Jacobi forms
 numerically\footnote{i.e. after specialization to $\BQ$-valued Gromov--Witten invariants}.
  \item[(b)] The holomorphic anomaly equation holds for the rational elliptic surface numerically.
\end{enumerate}
In particular, (a) solves the complete descendent Gromov--Witten theory of
the rational elliptic surface in terms of $E_8$-quasi-Jacobi forms.
We also show:
\begin{enumerate}
 \item[(c)] The quasi-Jacobi form property and the holomorphic anomaly equation
 are compatible with the degeneration formula (Section~\ref{Subsection_Compa_with_deg_formula}).
\end{enumerate}

These results directly lead to
a proof of Theorem~\ref{Thm1} and~\ref{Thm2} as follows.
The Schoen Calabi--Yau $X$ admits a degeneration
\[ X \rightsquigarrow (R_1 \times E_2) \cup_{E_1 \times E_2} (E_1 \times R_2), \]
where $E_i \subset R_i$ are smooth elliptic fibers.
By the degeneration formula \cite{Junli2}
we are reduced to studying the case $R_i \times E_j$.
By the product formula \cite{LQ}
the claim then follows from the holomorphic anomaly equation for the
rational elliptic surface and the elliptic curve \cite{HAE}.

For completeness we also prove the following case:
\begin{enumerate}
\item[(d)] The holomorphic anomaly equation holds for the reduced
 Gromov--Witten theory of the
 abelian surface in primitive classes numerically. 
\end{enumerate}
An overview of the state of the art on holomorphic anomaly equations and the results of the paper is given in Table \ref{table:1}.

\begin{table}[h!]
\centering
\begin{tabular}{|p{0.4cm}| p{2.7cm} | p{3.0cm} | c | p{3.0cm} |}
 \hline
 dim & Geometry & Modularity & HAE & Comments\\ [0.5ex] 
 \hline\hline
 \multirow{2}{1cm}{$1$} & Elliptic curves & $\mathrm{SL}_2(\BZ)$-quasimod. & Yes & cycle-valued \cite{HAE} \\
 \cline{2-5}
 & Elliptic orbifold $\p^1$s & $\Gamma(n)$-quasimod. & Yes & cycle-valued \cite{MRS} (except case $(2^4)$)\\
 \hline
 \multirow{3}{1cm}{$2$} & K3 surfaces & $\mathrm{SL}_2(\BZ)$-quasimod. & Yes & numerically, primitive only \cite{MPT, HAE} \\ \cline{2-5}
 & Abelian surfaces & $\mathrm{SL}_2(\BZ)$-quasimod. & {\bf Yes} & numerically, primitive only \cite{BOPY} \\ \cline{2-5}
 & Rational elliptic surface & $\mathbf{E_8}${\bf -quasi-Jacobi forms} & {\bf Yes} & numerically, relative $\p^1$ \\
 \hline
 \multirow{3}{1cm}{$3$} & Local $\p^2$ & Explicit generators & Yes & cycle-valued \cite{LP} \\ \cline{2-5}
 & Formal Quintic & Explicit generators & Yes & cycle-valued \cite{LP} \\ \cline{2-5}
 & Schoen CY3 & $\mathbf{E_8 \times E_8}$ {\bf bi-quasi-Jacobi forms} & \textbf{Yes} & numerically, relative $\p^1$\\
 \hline
\end{tabular}
\caption{List of geometries for which modularity and holomorphic anomaly equations (HAE) are known. The {\bf bold} entries are proven in this paper.
Cycle-valued = as Gromov--Witten classes on $\Mbar_{g,n}$; numerically = as numerical Gromov--Witten invariants; primitive = for primitive curve classes only;
relative $B$ = relative to the base $B$ of a Calabi--Yau fibration.} 
\label{table:1}
\end{table}

\subsection{Overview of the paper}
In Section~\ref{Section_Lattice_Jacobi_forms} we review the theory of lattice quasi-Jacobi forms. We introduce the derivations induced by the non-holomorphic completions, prove some structure results, and discuss examples.
In Section~\ref{Section_Elliptic_fibrations_and_conjectures}
we present the main conjectures of the paper.
We conjecture that the $\pi$-relative Gromov--Witten theory of
an elliptic fibration is expressed by quasi-Jacobi forms
and satisfies a holomorphic anomaly equation with respect to the modular parameter.
In Section~\ref{Section_Consequences_of_the_Conjecture}
we discuss implications of the conjectures of Section~\ref{Section_Elliptic_fibrations_and_conjectures}.
In particular, we deduce the weight of the quasi-Jacobi form,
present a holomorphic anomaly equation with respect to the elliptic parameter,
and prove that under good conditions the Gromov--Witten potentials satisfy the elliptic transformation law of Jacobi forms.
The relationship to higher level quasi-modular forms is discussed.
In Section~\ref{Section_relative_geomtries} we extend the conjectures of Section~\ref{Section_Elliptic_fibrations_and_conjectures}
to the Gromov--Witten theory of $X$ relative to a divisor $D$, when both admit compatible elliptic fibrations.
We show that the conjectural holomorphic anomaly equation is compatible with the degeneration formula.
In Section~\ref{Section_RationalEllipticSurface} we study
the rational elliptic surface.
We show that the conjecture holds in all degrees and genera after specializing to numerical Gromov--Witten invariants;
in particular we show that the Gromov--Witten potentials are $E_8$ quasi-Jacobi forms (Section~\ref{Subsection_RES_Statement_of_results}).
The idea of the proof is to adapt a calculation scheme of Maulik--Pandharipande--Thomas \cite{MPT}
and show every step preserves the conjectured properties.
In Section~\ref{Section_Schoen_variety} we prove Theorems~\ref{Thm1} and~\ref{Thm2}
and Corollary~\ref{Cor_Schoen}.
%
In Section~\ref{Section_Abelian_surfaces} we numerically prove a holomorphic anomaly equation
for the reduced Gromov--Witten theory of abelian surfaces in primitive classes.

In Appendix~\ref{Section_CohFTs} we introduce weak $B$-valued field theories
and define a matrix action on the space of these theories.
This generalizes the Givental $R$-matrix action on cohomological field theories.
We express the conjectural holomorphic anomaly equation
as a matrix action
and discuss the compatibility with the Jacobi Lie algebra.
In Appendix~\ref{Section_K3fibrations} we discuss
relative holomorphic anomaly equations for K3 fibrations
in an example.

\subsection{Conventions} \label{Subsection_Conventions}
We always work with integral
cohomology modulo torsion,
in particular $H^{\ast}(X,\BZ)$ will stand for singular cohomology of $X$ modulo torsion.
On smooth connected projective varieties
we identify cohomology with homology classes via Poincar\'e duality.
A curve class is the homology class of a (possibly empty) algebraic curve.
Given $x \in \BC$ we write $e(x) = e^{2 \pi i x}$.
Results conditional on conjectures are denoted by Lemma*, Proposition*, etc.

\subsection{Acknowledgements}
We would like to thank Hyenho Lho and Rahul Pandharipande for conversations
on holomorphic anomaly equations, Jim Bryan for discussions on the Schoen Calabi-Yau,
Davesh Maulik for sharing his insights,
and Martin Raum for comments on Jacobi forms.
The results of the paper were first presented during a visit of the first author to
ETH Z\"urich in June 2017;
we thank the Forschungsinstitut f\"ur Mathematik for support.
The second author was supported by a fellowship from the Clay Mathematics Institute.

\section{Lattice Jacobi forms} \label{Section_Lattice_Jacobi_forms}
\subsection{Overview}
In Section~\ref{Subsection_Modular_forms} we briefly recall quasi-modular forms following Kaneko-Zagier \cite{KZ} and
Bloch-Okounkov \cite{BO}. 
Subsequently we give a modest introduction to lattice quasi-Jacobi forms.
Lattice Jacobi forms were defined in \cite{Zie}
and an introduction can be found in \cite{Sko}.
A definition of quasi-Jacobi forms of rank~$1$ appeared in \cite{Lib},
and for higher rank can be found in \cite{KM}.

\subsection{Modular forms} \label{Subsection_Modular_forms}
\subsubsection{Definition}
Let $\BH = \{ \tau \in \BC | \mathrm{Im}(\tau) > 0 \}$ be the upper half plane and set $q = e^{2 \pi i \tau}$.
A \emph{modular form} of weight $k$ is a holomorphic function $f(\tau)$ on $\BH$ satisfying
\begin{equation} f\left( \frac{a \tau + b}{c \tau + d} \right) = (c \tau + d)^k f(\tau) \label{TRPROP} \end{equation}
for all $\binom{a\ b}{c\ d} \in \mathrm{SL}_2(\BZ)$ and
admitting a Fourier expansion in $|q|<1$ of the form
\begin{equation} f(\tau) = \sum_{n = 0}^{\infty} a_n q^n, \quad \quad a_n \in \BC. \label{ererqe33} \end{equation}

An \emph{almost holomorphic function} is a function
\[ F(\tau) = \sum_{i=0}^{s} f_i(\tau) \frac{1}{y^{i}}, \quad \quad y = \mathrm{Im}(\tau) \]
on $\BH$ such that every $f_i$ has a Fourier expansion in $|q|<1$ of the form \eqref{ererqe33}.

An \emph{almost holomorphic modular form} of weight $k$ is an
almost holomorphic function
which satisfies the transformation law \eqref{TRPROP}.

A \emph{quasi-modular form} of weight $k$ is a function $f(\tau)$
for which there exists an almost holomorphic modular form $\sum_i f_i y^{-i}$ of weight $k$ with
$f_0 = f$.

We let $\AHM_{\ast}$ (resp. $\QMod_{\ast}$) be the ring of almost holomorphic
modular forms (resp. quasi-modular forms) graded by weight.
The 'constant term' map
\begin{equation} \AHM \to \QMod, \quad \sum_i f_i y^{-i} \mapsto f_0 \label{constanttermmap} \end{equation}
is well-defined and an isomorphism \cite{KZ, BO}.

\subsubsection{Differential operators}
The non-holomorphic variable
\[ \nu = \frac{1}{8 \pi y} \]
transforms under the action of $\binom{a\ b}{c\ d} \in \mathrm{SL}_2(\BZ)$ on $\BH$ as
\begin{equation} \nu\left( \frac{ a \tau + b}{c \tau + d} \right) = (c \tau + d)^2 \nu(\tau) + \frac{c (c \tau + d)}{4 \pi i}. \label{3432} \end{equation}

We consider $\tau$ and $\nu$ here as independent variables and define operators
\[ D_q = \frac{1}{2\pi i} \frac{d}{d \tau} = q \frac{d}{dq}, \quad D_{\nu} = \frac{d}{d \nu}. \]
Since $\tau$ and $\nu$ are independent we have
\[ D_q \nu = 0, \quad D_{\nu} \tau = 0. \]

A direct calculation using \eqref{3432} shows the ring $\AHM_{\ast}$ admits the derivations
\begin{align*} \widehat{D}_q = ( D_q - 2 k \nu + 2 \nu^2 D_{\nu} ) & \colon \AHM_k \to \AHM_{k+2} \\
D_{\nu} = \frac{d}{d\nu} & \colon \AHM_{k} \to \AHM_{k-2}.
\end{align*}
Since $\widehat{D}_q$ acts as $D_q$ on the constant term in $y$ we conclude that
$D_q$ preserves quasi-modular forms:
\[ D_q : \QMod_k \to \QMod_{k+2}. \]
Similarly, define the anomaly operator
\[ \T_q : \QMod_k \to \QMod_{k-2} \]
to be the map which acts by $D_{\nu}$ under the constant term isomorphism \eqref{constanttermmap}.
The following diagrams therefore commute:
\[
\begin{tikzcd}
\QMod_k \ar{d}{D_q} \ar{r}{\cong} & \AHM_k \ar{d}{\widehat{D}_q} & & \QMod_k \ar{d}{\T_q} \ar{r}{\cong} & \AHM_k \ar{d}{D_{\nu}} \\
\QMod_{k+2} \ar{r}{\cong} & \AHM_{k+2}, & & \QMod_{k-2} \ar{r}{\cong} & \AHM_{k-2}.
\end{tikzcd}
\]

The commutator relation
$\big[ D_{\nu}, \widehat{D}_q \big]|_{\AHM_k} = - 2 k \cdot \mathrm{id}_{\AHM_k}$
yields
\[ \left[ \T_q, D_q \right]\big|_{\QMod_k} = - 2 k \cdot \mathrm{id}_{\QMod_k}. \]

The operators $\T_q$ allows us to describe the modular transformation
of quasi-modular forms.
\begin{lemma} \label{Sdadgd}
For any $f(\tau) \in \QMod_k$ we have
\[
f\left( \frac{a \tau + b}{c \tau + d} \right)
= \sum_{\ell = 0}^{m} \frac{1}{\ell!} \left( - \frac{c}{4 \pi i} \right)^{\ell} (c \tau + d)^{k-\ell} \T_q^{\ell} f(\tau).
\]
\end{lemma}
\begin{proof}
Let $F(\tau) = \sum_{i=0}^{m} f_i(\tau) \nu^i$ be the almost holomorphic modular form
with associated quasi-modular form $f(\tau) = f_0(\tau)$. 
Let $A = \binom{a\ b}{c\ d}$, $j = c \tau + d$ and $\alpha = \frac{c}{4 \pi i}$.
We claim
\[ f_{r}(A \tau) = \sum_{\ell = r}^{m} (-\alpha)^{\ell -r} \binom{l}{r} j^{k-r-\ell} f_{\ell}(\tau) \]
for all $r$. The left-hand side is uniquely determined
from $F( A \tau ) = j^k F(\tau)$ by solving recursively from the highest $\nu$ coefficients on.
One checks the given equation is compatible with this constraint.
\end{proof}

\subsubsection{Eisenstein Series}
Let $B_k$ be the Bernoulli numbers. The Eisenstein series
\[
C_{k}(\tau) = -\frac{B_k}{k \cdot k!} + \frac{2}{k!}
\sum_{n \geq 1} \sum_{d|n} d^{k-1} q^n
\]
are modular forms of weight $k$ for every even $k > 2$.
In case $k=2$ we have
\[ C_2\left( \frac{a \tau + b}{c \tau + d} \right) = (c \tau + d)^2 C_2(\tau) - \frac{ c ( c \tau + d)}{4 \pi i} \]
for all $\binom{a\ b}{c\ d} \in \mathrm{SL}_2(\BZ)$.
Hence
\begin{equation} \widehat{C}_2(\tau) = \widehat{C}_2(\tau, \nu) = C_2(\tau) + \nu \label{C2completion} \end{equation}
is almost holomorphic and $C_2$ is quasi-modular (of weight $2$).

It is well-known that
\begin{equation} \QMod = \BQ[C_2, C_4, C_6], \quad \AHM = \BQ[ \widehat{C}_2, C_4, C_6 ] \label{fdfsdfdf} \end{equation}
and the inverse to the constant term map \eqref{constanttermmap} is
\[ \QMod \to \AHM, \ \ f(C_2, C_4, C_6) \mapsto \widehat{f} = f( \widehat{C}_2, C_4, C_6). \]
In particular, \[ \T_q = \frac{d}{dC_2}. \]

\begin{rmk}
Once the structure result \eqref{fdfsdfdf} is known we can immediately work with $\frac{d}{dC_2}$
and we do not need to talk about transformation laws. However, below
in the context of quasi-Jacobi forms we do not have such strong results at hands and we will use
an abstract definition of $\T_q$ instead (though see Section~\ref{Subsubsection_Rewriting_Tnu} for a version of $\frac{d}{dC_2}$).
\end{rmk}

\subsection{Jacobi forms} \label{Subsection_Jacobi_forms}
\subsubsection{Definition}
Consider variables
$z = (z_1, \ldots, z_n) \in \BC^n$, let $k \in \BZ$, and let
$L$ be a rational $n \times n$-matrix
such that $2L$ is integral and has even diagonals\footnote{This is the weakest condition on $L$
for which the second equation in \eqref{TRANSFORMATIONLAWJACOBI} can be nontrivially satisfied.
Indeed, if the condition is violated then $\lambda^T L \lambda$ is not integral in general and
hence the $q$-expansion of $\phi$ is fractional which contradicts \eqref{Jacobiexpansion}.}.

A \emph{weak Jacobi form} of weight $k$ and index $L$
is a holomorphic function $\phi(z, \tau)$ on $\BC^n \times \BH$
satisfying
\begin{equation} \label{TRANSFORMATIONLAWJACOBI}
\begin{aligned}
\phi\left( \frac{z}{c \tau + d}, \frac{a \tau + b}{c \tau + d} \right)
& = (c \tau + d)^k e\left( \frac{c z^t L z}{c \tau + d} \right) \phi(z,\tau) \\
\phi\left( z + \lambda \tau + \mu, \tau \right)
& = e\left( - \lambda^t L \lambda \tau - 2 \lambda^t L z \right) \phi(z,\tau)
\end{aligned}
\end{equation}
for all $\binom{a\ b}{c\ d} \in \mathrm{SL}_2(\BZ)$ and $\lambda, \mu \in \BZ^n$
and admitting a Fourier expansion of the form
\begin{equation} \label{Jacobiexpansion}
\phi(z,\tau) = \sum_{n \geq 0} \sum_{r \in \BZ^n} c(n,r) q^n \zeta^r
\end{equation}
in $|q|<1$; here we used the notation
\[ \zeta^r = e(z \cdot r) = e\left( \sum_i z_i r_i \right) = \prod_i \zeta_i^{r_i} \]
with $\zeta_i = e(z_i)$.

We will call the first equation in \eqref{TRANSFORMATIONLAWJACOBI} the modular,
and the second equation in \eqref{TRANSFORMATIONLAWJACOBI} the elliptic transformation law of Jacobi forms.

By definition weak Jacobi forms are allowed to have poles at cusps.
If the index $L$ is positive definite then a \emph{(holomorphic) Jacobi form} is a weak Jacobi form
which is holomorphic at cusps, or equivalently, satisfies $c(n,r) = 0$
unless $r^{t} L^{-1} r \leq 4n$.
We will not use this stronger notion and all the Jacobi forms
are considered here to be weak.


\subsubsection{Quasi-Jacobi forms}
For every $i$ consider the real analytic function
\[ \alpha_i(z,\tau) = \frac{z_i - \overline{z_i}}{\tau - \overline{\tau}} = \frac{\im(z_i)}{\im(\tau)} \]
and define
\[ \alpha = (\alpha_1, \ldots, \alpha_n). \]
We have the transformations
\begin{align*}
\alpha\left( \frac{z}{c \tau + d}, \frac{a \tau + b}{c \tau + d} \right)
 & = (c \tau + d) \alpha(z,\tau) - c z \\
 \alpha\left( z + \lambda \tau + \mu, \tau\right)
 & = \alpha(z,\tau) + \lambda
\end{align*}
for all $\binom{a\ b}{c\ d} \in \mathrm{SL}_2(\BZ)$ and $\lambda, \mu \in \BZ^n$.

An \emph{almost holomorphic function} on $\BC^n \times \BH$ is a function
\[ \Phi(z, \tau) = \sum_{i \geq 0} \sum_{j = (j_1, \ldots, j_n) \in (\BZ_{\geq 0})^n}
\phi_{i, j}(z,\tau) \nu^{i} \alpha^j, \quad \quad \alpha^j = \alpha_1^{j_1} \cdots \alpha_n^{j_n} \]
such that each
of the finitely many non-zero
$\phi_{i, j}(z,\tau)$ is holomorphic and admits a Fourier expansion
of the form \eqref{Jacobiexpansion} in the region $|q|<1$.

An \emph{almost holomorphic weak Jacobi form} of weight $k$ and index $L$
is an almost holomorphic function $\Phi(z,\tau)$
which satisfies the transformation law \eqref{TRANSFORMATIONLAWJACOBI} of weak Jacobi forms of weight $k$ and index $L$.

A \emph{quasi-Jacobi form} of weight $k$ and index $L$ is a function $\phi(z,\tau)$ on $\BC^n \times \BH$
such that there exists an almost holomorphic weak Jacobi form
$\sum_{i,j} \phi_{i,j} \nu^i \alpha^j$ of weight $k$ and index $L$
with $\phi_{0,0} = \phi$.

We let $\AHJ_{k,L}$ (resp. $\QJ_{k,L}$) be the vector space of almost holomorphic weak (resp. quasi-) Jacobi forms
of weight $k$ and index $L$.
The vector space of index~$L$ quasi-Jacobi forms is denoted by
\[ \QJac_{L} = \bigoplus_{k \in \BZ} \QJ_{k,L}. \]
Multiplication of functions endows the direct sum
\[ \QJac = \bigoplus_{L} \QJac_{L}, \]
where $L$ runs over all 
rational $n \times n$-matrices
such that $2L$ is integral and has even diagonals,
with a commutative ring structure. We call
$\QJac$ the \emph{algebra of quasi-Jacobi forms}
on $n$ variables.


\begin{lemma} The constant term map
\[ \AHJ_{k,L} \to \QJ_{k,L}, \quad \sum_{i,j} \phi_{i, j} \nu^{i} \alpha^j \mapsto \phi_{0,0} \]
is well-defined and an isomorphism.
\end{lemma}
\begin{proof}
Parallel to the rank $1$ case in \cite{Lib}.
\end{proof}

\subsubsection{Differential operators} \label{Subsubsection_quasiJacobi_Differentialoperators}
Consider $\tau, \nu, z_i, \alpha_i$ as independent variables and recall
the Fourier variables $q = e^{2 \pi i \tau}$ and $\zeta_i = e^{2 \pi i z_i}$.
Define the differential operators
\begin{alignat*}{2}
D_q & = \frac{1}{2 \pi i} \frac{d}{d\tau} = q \frac{d}{dq} & \quad \quad \quad D_\nu & = \frac{d}{d \nu} \\
D_{\zeta_i} & = \frac{1}{2 \pi i} \frac{d}{d z_i} = \zeta_i \frac{d}{d \zeta_i} & D_{\alpha_i} & = \frac{d}{d \alpha_i}.
\end{alignat*}

A direct check using the transformation laws \eqref{TRANSFORMATIONLAWJACOBI} shows
\[ D_{\nu} : \AHJ_{k,L} \to \AHJ_{k-2, L}, \quad D_{\alpha_i} : \AHJ_{k,L} \to \AHJ_{k-1,L}. \]
Define anomaly operators $\T_q$ and $\T_{\alpha_i}$ by the commutative diagrams
\[
\begin{tikzcd}
\QJ_{k,L} \ar{d}{\T_q} & \ar{l}{\cong} \AHJ_k \ar{d}{D_{\nu}} & & \QJ_{k,L} \ar{d}{\T_{\alpha_i}} & \ar{l}{\cong} \AHJ_k \ar{d}{D_{\alpha_i}} \\
\QJ_{k-2,L} & \ar{l}{\cong} \AHJ_{k-2,L} & & \QJ_{k-1,L} & \ar{l}{\cong} \AHJ_{k-1,L}
\end{tikzcd}
\]
where the horizontal maps are the 'constant term' maps.

Similarly, we have operators\footnote{See \cite[Sec.2]{CR} for a
Lie algebra presentation of these operators.}
\begin{align*}
\widehat{D}_q = \left( D_q - 2 k \nu + 2 \nu^2 D_{\nu} + \sum_{i=1}^{n} \alpha_i D_{\zeta_i} + \alpha^T L \alpha \right) \colon \AHJ_{k,L} \to \AHJ_{k+2,L} \\
\widehat{D}_{\zeta_i} = \left( D_{\zeta_i} + 2 \alpha^T L e_i - 2 \nu D_{\alpha_i} \right) \colon \AHJ_{k,L} \to \AHJ_{k+1,L}
\end{align*}
where $e_i = (\delta_{ij})_{j}$ is the $i$-th standard basis vector in $\BC^n$.
Since $\widehat{D}_q, \widehat{D}_{\zeta_i}$ act as $D_q$ and $D_{\zeta_i}$ on the constant term,
we find that $D_q, D_{\zeta_i}$ act on quasi-Jacobi forms:
\[ D_q : \QJ_{k,L} \to \QJ_{k+2,L}, \quad D_{\zeta_i} : \QJ_{k,L} \to \QJ_{k+1,L}. \]

For $\lambda = (\lambda_1, \ldots, \lambda_n) \in \BZ^n$ we will write
\[ D_{\lambda} = \sum_{i=1}^{n} \lambda_i D_{\zeta_i}, \quad \quad
\T_{\lambda} = \sum_{i=1}^{n} \lambda_i \T_{\alpha_i}.
 \]

The commutation relations of the above operators read\footnote{The operators $\T_q, \T_{\lambda}, D_q, D_{\lambda}$ as well as the weight and index grading operators define an action
of the Lie algebra of the semi-direct product of $\mathrm{SL}_2(\BC)$ with a Heisenberg group on the space $\QJ_{\mathbf{L}}$,
see \cite[Sec.1]{Zie}, \cite[Sec.2]{CR} and also \cite[Thm.1.4]{EZ}.}
\begin{equation} \label{COMMUTATIONRELATIONS}
\begin{alignedat}{2}
\left[ \T_q, D_q \right]\big|_{\QJac_{k,L}} & = -2 k \cdot \id_{\QJac_{k,L}} &
\left[ \T_{\lambda}, D_q \right] & = D_{\lambda} \\
\left[ \T_{\lambda}, D_{\mu} \right]\big|_{\QJac_{k,L}} & = 2 ( \lambda^T L \mu )\cdot \id_{\QJac_{k,L}}
& \quad \quad \left[ \T_q, D_{\lambda} \right] & = -2 \T_{\lambda}
\end{alignedat}
\end{equation}
and
\[ [ D_q, D_{\lambda} ] = [ D_{\lambda}, D_{\mu} ] = [ \T_q, \T_{\lambda} ] = [ \T_{\lambda} , T_{\mu} ] = 0 \]
for all $\lambda, \mu \in \BZ^n$.

\begin{lemma}
\label{Lemma_ell_trans_law_for_quasi_Jac}
Let $\phi \in \QJac_{L}$. Then
\begin{align*} \phi(z + \lambda \tau + \mu, \tau)
& = e\left( - \lambda^t L \lambda \tau - 2 \lambda^t L z \right)  \sum_{\ell \geq 0} \frac{(-1)^i}{i!} \T_{\lambda}^i \phi(z,\tau) \\
& = e\left( - \lambda^t L \lambda \tau - 2 \lambda^t L z \right) \exp\left( - \T_{\lambda} \right) \phi(z,\tau) 
\end{align*}
\end{lemma}
\begin{proof}
Since the claimed formula is compatible
with addition on $\BZ^n$, we may assume $\lambda = e_i$.
Let $\Phi$ be the non-holomorphic completion of $\phi$. We expand
\[ \Phi = \sum_{j \geq 0} \phi_j \alpha_i^j \] 
where $\phi_j$ depends on all variables except $\alpha_i$
(these variables are invariant under $z \mapsto z + e_i \tau$).
Then a direct check shows that the claimed formula is determined by, and compatible with the relation
\[ \Phi(z + e_i \tau) = e\left( - e_i^t L e_i \tau - 2 e_i L z \right) \Phi(z). \qedhere \]
\end{proof}

\begin{lemma} \label{Lemma_QJAc_ModTrLaw}
Let $\phi \in \QJac_{k,L}$ such that $\T_{\lambda} \phi = 0$ for all $\lambda \in \BZ^n$. Then
\[
\phi\left( \frac{a \tau + b}{c \tau + d} \right)
= e\left( \frac{c z^T L z}{c \tau + d} \right) \sum_{\ell \geq 0} \frac{1}{\ell!} \left( - \frac{c}{4 \pi i} \right)^{\ell} (c \tau + d)^{k-\ell} \T_q^{\ell} \phi(\tau).
\]
\end{lemma}
\begin{proof}
Since $\T_{\lambda} \phi = 0$ for all $\lambda$, the non-holomorphic
completion of $\phi$ is of the form
$\Phi(z,\tau) = \sum_{i \geq 0} \phi_i(z,\tau) \nu^i$
where $\phi_i$ are \emph{holomorphic} and in
$\cap_{\lambda} \mathrm{Ker}( T_{\lambda} )$. 
The same proof as Lemma~\ref{Sdadgd} applies now.
\end{proof}

\subsubsection{Rewriting $\T_q$ as $\frac{d}{dC_2}$} \label{Subsubsection_Rewriting_Tnu}
Define the vector space of quasi-Jacobi forms which are annihilated by $\T_q$:
\[ \QJac'_{L} = \mathrm{Ker}\left( \T_q : \QJac_{L} \to \QJac_L \right). \]
We have the following structure result
whose proof is essentially identical to \cite[Prop.3.5]{BO} and which we therefore omit.

\begin{lemma} $\QJac_L = \QJac'_L \otimes_{\BC} \BC[C_2]$. \end{lemma}

By the Lemma
every quasi-Jacobi form can be uniquely written as a polynomial in $C_2$.
In particular,
the formal derivative $\frac{d}{dC_2}$ is well-defined. Comparing with \eqref{C2completion} we conclude that
\[ \T_q = \frac{d}{dC_2} : \QJac_L \to \QJac_L. \]

\subsubsection{Specialization to quasi-modular forms} \label{Subsubsection_Specialization_to_quasimodular_forms}
By setting $z=0$ the quasi-Jacobi forms of weight $k$ and index $L$
specialize to weight $k$ quasi-modular forms:
\begin{align*}
\AHJ_{k,L} \to \AHM_{k},& \quad F(z,\tau) \mapsto F(0,\tau) \\
\QJac_{k,L} \to \QMod_k,& \quad f(z,\tau) \mapsto f(0,\tau).
\end{align*}
The specialization maps commute with the operators $\T_q$.

\subsection{Theta decomposition and periods} \label{Subsection_theta_decomposition}
We discuss theta decompositions of quasi-Jacobi forms
if the index $L$ is positive definite.
For this we will need to work with several more general notions of modular forms
than what we have defined above (e.g. for congruence subgroups, of half-integral weight, or vector-valued).
Since we do not need the results of this section for the main arguments of the paper
we will not introduce these notions here and instead refer to \cite{Sko, Scheit}.\footnote{
The results of Section~\ref{Subsection_theta_decomposition} are essential only for Section~\ref{Subsubsection_elliptic_fibrations_quasimodularforms},
which is not used later on.
Proposition~\ref{QJac_Prop2} also appears in Section~\ref{Subsection_ab_explain},
but in this case the lattice $2L$ is unimodular
and hence we can use Proposition~\ref{QJac_Prop3}
to re-prove Proposition~\ref{QJac_Prop2} without additional theory.
}

Assume $L$ is positive definite, and
for every $x \in \BZ^n/ 2 L \BZ^n$ define the index $L$ theta function
\[ \vartheta_{L,x}(z,\tau) = \sum_{\substack{r \in \BZ^n \\ r \equiv x \text{ mod } 2L \BZ^n}} e\left( \tau \frac{1}{4} r^{T} L^{-1} r + r^T z \right). \]
Let $\mathrm{Mp}_2(\BZ)$ be the metaplectic double cover of $\mathrm{SL}_2(\BZ)$ and
consider the ring
\[ \widetilde{\Jac}_{k,L} = \bigcap_{\lambda \in \BZ^n} \mathrm{Ker}\left( \T_{\lambda} : \QJac_{k,L} \to \QJac_{k+1, L} \right). \]

\begin{prop} \label{QJac_Prop1} Assume $L$ is positive definite and let $f \in \QJac_{k,m}$.
\begin{enumerate}
\item[(i)] There exist (finitely many) unique quasi-Jacobi forms $f_d \in \widetilde{\Jac}_{k-\sum_i d_i, L}$
where $d=(d_1, \ldots, d_n) \in \BZ_{\geq 0}^n$
such that
\[
f(z,\tau) = \sum_{d} D_{\zeta_1}^{d_1} \cdots D_{\zeta_n}^{d_n} f_d(z,\tau).
\]
\item[(ii)] Every $f_d(z,\tau)$ above can be expanded as
\[ f_d(z,\tau) = \sum_{x \in \BZ^n/2 L \BZ^n} h_{k,x}(\tau) \vartheta_{L,x}(z,\tau) \]
where $( h_{k,x} )_{x}$ is a vector-valued
weakly-holomorphic quasi-modular form for the dual of the Weil representation
of $\mathrm{Mp}_2(\BZ)$ on $\BZ^n / 2 L \BZ^n$.
\end{enumerate}
\end{prop}

The quasi-modular forms $(h_{k,x})_x$ of (ii) are weakly holomorphic
(i.e. have poles at cusps) 
since we define our quasi-Jacobi forms as almost-holomorphic versions of weak-Jacobi forms. The quasi-Jacobi forms for which $(h_{k,x})_x$ are
holomorphic correspond to
\emph{holomorphic} Jacobi forms
(which require a stronger vanishing condition on their Fourier coefficients).

\begin{proof}[Proof of Proposition~\ref{QJac_Prop1}]
(i) Let $F$ be the completion of $f$ and consider the expansion
\[ F = \sum_{j = (j_1, \ldots, j_n)} f_j(z,\tau,\nu) \alpha^j. \]
Let $j$ be a maximal index, i.e. $f_{j+e_i} = 0$ for every $i$ where $e_i$ is the standard basis.
Then $\T_{\lambda} f_j = 0$ for every $\lambda$ and hence $f_j \in \widetilde{\Jac}_{k-|j|,L}$.
Replacing $f$ by
\[ f - \left( D_{\frac{1}{2} L^{-1} e_1} \right)^{j_1} \cdots \left( D_{\frac{1}{2} L^{-1} e_n} \right)^{j_n} f_j \]
the claim follows by induction.

(ii) The existence of $h_{k,x}(\tau)$ follows from the elliptic transformation law.
For the modularity see \cite[Sec.4]{Sko}.
\end{proof}

The \emph{level of $L$} is the smallest positive integer $\ell$ such that $\frac{1}{2} \ell L^{-1}$ has integral entries and even diagonal entries.
Let 
\[ \Gamma(\ell)^{\ast} \subset \mathrm{Mp}_2(\BZ) \]
be the lift of the congruence subgroup $\Gamma(\ell)$ to $\mathrm{Mp}_2(\BZ)$ defined in \cite[Sec.2]{Sko}.

Given a function $f = \sum_{r \in \BZ^n} c_r(\tau) \zeta^r$ with $\zeta^r = e(z^t r)$, let
\[ [ f ]_{\zeta^{\lambda}} = c_{\lambda}(\tau) \]
denote the coefficient of $\zeta^\lambda$.

\begin{prop} \label{QJac_Prop2} Assume $L$ is positive definite of level $\ell$ and let $f \in \QJac_{k,L}$. 
\begin{enumerate}
 \item[(i)] For every $\lambda \in \BZ^n$, the coefficient
 \[ [ \, f \, ]_{\lambda} := q^{-\frac{1}{4} \lambda^T L^{-1} \lambda} [ \, f \, ]_{\zeta^{\lambda}} \]
 is a weakly-holomorphic quasi-modular form for $\Gamma(\ell)^{\ast}$ of weight $\leq k-\frac{n}{2}$.
 If $\lambda = 0$ then $[ \, f \, ]_{\lambda}$ is homogeneous of weight $k-\frac{n}{2}$.
 \item[(ii)] We have
 \[
  \T_q [\,  f \, ]_{\lambda} = \left[\, \T_q f\, \right]_{\lambda} + \frac{1}{2} \sum_{a,b=1}^{n} \left( L^{-1} \right)_{ab}
  \left[\, \T_{\zeta_a} \T_{\zeta_b} f \, \right]_{\lambda}.
 \]
\end{enumerate}
\end{prop}

In (ii) of Proposition~\ref{QJac_Prop2} we used an extension of the operator $\T_q$ to quasi-modular forms for congruence subgroups.
The existence of this operator follows parallel to Section~\ref{Subsection_Modular_forms}
from the arguments of \cite{KZ}.

The $\zeta^{\lambda}$-coefficients of Jacobi forms are sometimes referred to as periods.
A quasi-modularity result for the periods of certain multivariable elliptic functions
(certain meromorphic Jacobi forms of index $L=0$)
has been obtained in \cite[App.A]{HAE}. The formula in \cite[Thm.7]{HAE} is similar
to the above but requires a much more delicate argument.

\begin{proof}[Proof of Proposition~\ref{QJac_Prop2}]
(i) The Weil representation restricts to the trivial representation on $\Gamma(\ell)$,
see \cite[Prop.4.3]{Scheit}.
Hence the $h_{k,x}$ are $\Gamma(\ell)^{\ast}$-quasi-modular by Proposition~\ref{QJac_Prop1}(ii).

(ii) For the second part we consider the expansion
\begin{equation} f(z,\tau) =  \sum_{x \in \BZ^n/2 L \BZ^n} \sum_{d}
h_{k,x}(\tau) D_{\zeta_1}^{d_1} \cdots D_{\zeta_n}^{d_n} \vartheta_{L,x}(z,\tau)
\label{dfsderwe} \end{equation}
which follows from combining both parts of Proposition~\ref{QJac_Prop1}.
Let
\[ \T_{\Delta} = \frac{1}{2} \sum_{a,b=1}^{n} \left( L^{-1} \right)_{ab} \T_{e_a} \T_{e_b}. \]
By \eqref{COMMUTATIONRELATIONS} we have
$\left[ \T_q, D_{\lambda} \right] = -\left[ \T_{\Delta} , D_{\lambda} \right]$
for every $\lambda$.
Since $\vartheta_{L,x}$ is a Jacobi form (for a congruence subgroup\footnote{We extend
the operators $\T_q, \T_{\lambda}$ here to quasi-Jacobi forms for congruence subgroups. The commutation relations are identical.})
we also have $\T_q \vartheta_{L,x} = \T_{\Delta} \vartheta_{L,x} = 0$.
This implies
\[ \T_q D_{\zeta_1}^{d_1} \cdots D_{\zeta_n}^{d_n} \vartheta_{L,x}(z,\tau)
 = -\T_{\Delta} D_{\zeta_1}^{d_1} \cdots D_{\zeta_n}^{d_n} \vartheta_{L,x}(z,\tau).
\]
Hence applying $\T_q$ to \eqref{dfsderwe} yields
\begin{align*}
\T_q f & = \sum_{x,d} \left( \T_q( h_{k,x} ) D_{\zeta_1}^{d_1} \cdots D_{\zeta_n}^{d_n} \vartheta_{L,x}
 - h_{k,x} \T_{\Delta} D_{\zeta_1}^{d_1} \cdots D_{\zeta_n}^{d_n} \vartheta_{L,x} \right) \\
 & = \left( \sum_{x,d} \T_q( h_{k,x} ) D_{\zeta_1}^{d_1} \cdots D_{\zeta_n}^{d_n} \vartheta_{L,x} \right) - \T_{\Delta} f.
\end{align*}
The claim follows by taking the coefficient of $\zeta^{\lambda}$.
\end{proof}

\begin{cor} $\QJac_{k,L}$ is finite-dimensional for every weight $k$
and positive definite index $L$.
\end{cor}
\begin{proof}
By Proposition~\ref{QJac_Prop1} the space $\widetilde{\Jac}_{k,L}$ is isomorphic to
a space of meromorphic vector-valued quasi-modular forms of some fixed weight $k$
for which the order of poles at the cusps is bounded by a constant depending only on $L$.
In particular, it is finite dimensional and vanishes for $k \ll 0$.
The claim now follows from the first part of Proposition~\ref{QJac_Prop1}.
\end{proof}

\subsection{Examples}
\subsubsection{Rank $0$}
If the lattice $\Lambda$ has rank zero, a quasi-Jacobi form of weight $k$ is a quasi-modular form of the same weight.

\subsubsection{Rank $1$ lattice}
The ring of quasi-Jacobi forms in the rank $1$ case has been
determined and studied by Libgober in \cite{Lib}.

\subsubsection{Half-unimodular index} \label{Subsubsection_Halfunimodular_index}
Let $Q$ be positive definite and unimodular of rank $n$.
We describe the ring of quasi-Jacobi forms 
of index $L = \frac{1}{2}Q$.
The main example is the Riemann theta function
\begin{equation} \label{THETAFUNCTION} \Theta_{Q}(z, \tau)
=
\sum_{\gamma \in \BZ^{n}} q^{\frac{1}{2} \gamma^T Q \gamma} e\left(  z^T Q \gamma \right),
\end{equation}
which is a Jacobi form\footnote{ 
Since $Q$ is unimodular the theta function satisfies the transformation laws for the full modular group and 
not just a subgroup \cite[Sec.3]{Zie}.}
of weight $n/2$ and index $Q/2$,
\[ \Theta_{Q}(z,\tau) \in \Jac_{\frac{n}{2}, \frac{1}{2}Q}. \]

The following structure result shows that this is essentially the only Jacobi form that we need to consider in this index.
\begin{prop} \label{QJac_Prop3}
Let $Q$ be positive definite and unimodular. Then every $f \in \QJac_{k, \frac{Q}{2}}$
can be uniquely written as a finite sum
\[ f = \sum_{d = (d_1, \ldots, d_n)} f_d(\tau) D_{\zeta_1}^{d_1} \cdots D_{\zeta_n}^{d_n} \Theta_{Q}(z,\tau) \]
where $f_d \in \QMod_{k-\sum_i d_i}$ for every $d$. In particular, for every $\lambda \in \BZ^n$ we have
\[ q^{-\frac{1}{4} \lambda^T Q^{-1} \lambda} [ \, f \, ]_{\zeta^{\lambda}} \in \QMod_{\leq k}. \]
\end{prop}
\begin{proof}
Parallel to the proof of Proposition~\ref{QJac_Prop1}.
\end{proof}

\subsubsection{The $E_8$ lattice and $E_8$-Jacobi forms} \label{Section_E8_lattice}
Consider the Cartan matrix of the $E_8$ lattice
\[
Q_{E_8} =
\begin{pmatrix}
 2 & -1 &  0 &  0 &  0 &  0 &  0 & 0 \\
-1 &  2 & -1&  0 &  0 &  0 &  0 & 0 \\
 0 & -1 &  2 & -1 &  0 &  0 &  0 & 0 \\
 0 &  0 & -1 &  2 & -1 &  0 &  0 & 0 \\
 0 &  0 &  0 & -1 &  2 & -1 &  0 & -1 \\
 0 &  0 &  0 &  0 & -1 &  2 & -1 & 0 \\
 0 &  0 &  0 &  0 &  0 & -1 &  2 & 0 \\
 0 &  0 & 0 &  0 &  -1 &  0 &  0 & 2
\end{pmatrix}.
\]
We define a natural subspace of the space of Jacobi forms of index $\frac{m}{2} Q_{E_8}$.

A \emph{weak $E_8$-Jacobi form} of weight $k$ and index $m$ is a
weak Jacobi form $\phi$ of weight $k$ and index $L = \frac{m}{2} Q_{E_8}$
which satisfies
\[ \phi( w(z), \tau ) = \phi( z, \tau ) \]
for all $w \in W(E_8)$, where $W(E_8)$ is the Weyl group of $E_8$.
We let 
\[ \Jac_{E_8, k,m} \subset \Jac_{k, \frac{m}{2} Q_{E_8}} \]
be the ring of weak $E_8$-Jacobi forms.

Practically the subspace of $E_8$-Jacobi forms is much smaller than the large space of
Jacobi forms of index $\frac{m}{2} Q_{E_8}$. 
The first example of an $E_8$-Jacobi form is the theta function $\Theta_{E_8}$ defined in \eqref{THETAFUNCTION}.
Further examples and a conjectural structure result for the ring of weak $E_8$-Jacobi forms
can be found in \cite{Sakai2}.

\section{Elliptic fibrations and conjectures}
\label{Section_Elliptic_fibrations_and_conjectures}
\subsection{Elliptic fibrations}
\subsubsection{Definition} \label{Subsubsection_ellip_fibr_definition}
Let $X$ and $B$ be non-singular projective varieties and let
\[ \pi : X \to B \]
be an elliptic fibration, i.e. a flat morphism
with fibers connected curves of arithmetic genus $1$.
We always assume $\pi$ satisfies the following properties\footnote{After Conjecture~\ref{Conj_HAE} we discuss how these assumptions can be removed.}:
\begin{enumerate}
\item[(i)] All fibers of $\pi$ are integral.
\item[(ii)] There exists a section $\iota : B \to X$.
\item[(iii)] $H^{2,0}(X, \BC) = H^0(X, \Omega_X^2) = 0$.
\end{enumerate}

\subsubsection{Cohomology}
Let $B_0 \in H^2(X)$
be the class of the section $\iota$,
and let $N_{\iota}$ be the normal bundle of $\iota$.
We define the divisor class
\[ W = B_0 - \frac{1}{2} \pi^{\ast} c_1(N_{\iota}). \]
Consider the endomorphisms of $H^{\ast}(X)$ defined by
\[
T_+(\alpha) = ( \pi^{\ast} \pi_{\ast} \alpha ) \cup W, \quad
T_-(\alpha) = \pi^{\ast} \pi_{\ast} ( \alpha \cup W ),
\]
for all $\alpha \in H^{\ast}(X)$. The maps $T_{\pm}$ satisfy the relations
\[ T_+^2 = T_+, \quad T_-^2 = T_-, \quad T_+ T_- = T_- T_+ = 0. \]
The cohomology of $X$ therefore splits as\footnote{
The subspaces $H_{+}, H_{-}, H_{\perp}$ are the $+1, -1, 0$-eigenspaces respectively
of the endomorphism of $H^{\ast}(X)$ defined by
\[ \alpha \mapsto [ W\cup \ , \pi^{\ast} \pi_{\ast}] \alpha = W \cup \pi^{\ast} \pi_{\ast}(\alpha) -
\pi^{\ast} \pi_{\ast}( W \cup \alpha). \]
}
\begin{equation} H^{\ast}(X) = H^{\ast}_{+} \oplus H^{\ast}_- \oplus H^{\ast}_{\perp} \label{SPLITTING} \end{equation}
where $H^{\ast}_{\pm} = \mathrm{Im}(T_\pm)$
and $H^{\ast}_{\perp} = \mathrm{Ker}( T_+ ) \cap \mathrm{Ker}(T_-)$.

We have the relation
\[
\langle T_{+}(\alpha), \alpha' \rangle = \langle \alpha, T_{-}(\alpha') \rangle,
\quad \alpha, \alpha' \in H^{\ast}(X)
\]
where $\langle \ , \ \rangle$ is the intersection pairing on $H^{\ast}(X)$. Therefore
\[ \blangle H^{\ast}_{+}, H^{\ast}_{+} \brangle = \blangle H^{\ast}_{-}, H^{\ast}_{-} \brangle =
\blangle H^{\ast}_{\pm}, H^{\ast}_{\perp} \brangle = 0. \]

Consider the isomorphisms
\begin{align*}
H^{\ast}(B) \to H^{\ast}_{-},\ \ &  \alpha \mapsto \pi^{\ast}(\alpha) \\
H^{\ast}(B) \to H^{\ast}_{+},\ \ &  \alpha \mapsto \pi^{\ast}(\alpha) \cup W.
\end{align*}
The pairing between $H^{\ast}_{+}$ and $H^{\ast}_{-}$ is determined by the compatibility
\[ \int_{B} \alpha \cdot \alpha' = \int_X \pi^{\ast}(\alpha) \cdot \left( \pi^{\ast}(\alpha') \cdot W \right)
\quad \text{ for all } \alpha, \alpha' \in H^{\ast}(B).
\]

\subsubsection{The lattice $\Lambda$} \label{Subsubsection_LatticeLambda}
Let $F \in H_2(X,\BZ)$ be the class of a fiber of $\pi$ and consider the $\BZ$-lattice
\[ \BZ F \oplus \iota_{\ast} H_2(B,\BZ) \subset H_2(X,\BZ). \]
Its orthogonal complement in the dual space $H^2(X,\BZ)$ is
the $\BZ$-lattice
\begin{equation} \Lambda = \Big( \BQ F \oplus \iota_{\ast} H_2(B,\BZ) \Big)^{\perp} \subset H^2(X,\BZ). \label{fsdfsdgsdf} \end{equation}
Since
$\BQ F \oplus \iota_{\ast} H_2(B,\BZ)$ generates $H_{2,+} \oplus H_{2,-}$ over $\BQ$,
we have
\[ \Lambda \subset H^2_{\perp}, \quad \Lambda \otimes \BQ = H^2_{\perp}. \]

Let $\kk_1, \ldots, \kk_r$ be an integral basis\footnote{Recall we always work modulo torsion as per Section~\ref{Subsection_Conventions}.}
of $H_2(B,\BZ)$
and let $\kk_i^{\ast} \in H^2(B,\BZ)$ be a dual basis.
The projection
\[ p_{\perp} : H^2(X,\BQ) \to H^2_{\perp} \]
with respect to the splitting \eqref{SPLITTING}
acts on $\alpha \in H^{2}(X)$ by
\[
p_{\perp}(\alpha) = \alpha - (\alpha \cdot F) B_0 - \sum_{i=1}^{r} \Big( \big( \alpha - (\alpha \cdot F) B_0  \big) \cdot \iota_{\ast} \kk_i \Big) \pi^{\ast} \kk_i^{\ast}.
\]
and is therefore defined over $\BZ$.
Hence the inclusion \eqref{fsdfsdgsdf} splits.

\subsubsection{Variables}
Consider a fixed integral basis of the free abelian group $\Lambda$,
\[ b_1, \ldots, b_n \in \Lambda. \]
We will identify an element $z = (z_1, \ldots, z_n) \in \BC^n$ with the
element $\sum_{i=1}^{n} z_i b_i$. Hence we obtain the identification
\[ \BC^n \cong \Lambda \otimes \BC = H^2_{\perp}(X,\BC). \]
Given a class $\beta \in H_2(X,\BZ)$, we write
\begin{equation}\label{eq:zetabeta}
\zeta^{\beta} = \exp( z \cdot \beta ) = \prod_{i=1}^{n} \zeta_i^{b_i \cdot \beta}
\end{equation}
where $\zeta_i = e(z_i)$ and $\cdot$ is the intersection pairing.

\subsubsection{Pairings and intersection matrices} \label{subsubsection_pairing_and_intersection_matrix}
Every element $\kk \in H_{2}(B,\BZ)$ defines a
symmetric (possibly degenerate) bilinear form on $H^2_{\perp}$ by
\[ (\alpha, \alpha')_\kk = \int_B \pi_{\ast}(\alpha \cup \alpha') \cdot \kk. \]
The restriction of $(\cdot, \cdot )_\kk$ to $\Lambda$ takes integral values. 

\begin{lemma} For every curve class $\kk \in H_2(B,\BZ)$
the quadratic form $( \cdot , \cdot )_\kk$ is even on $\Lambda$,
that is $(\alpha, \alpha)_\kk \in 2 \BZ$ for every $\alpha \in \Lambda$,
\end{lemma}
\begin{proof}
Since the pairing is linear in $\kk$ it suffices to prove
$( \cdot , \cdot )_{\kk + \ell}$ and $( \cdot , \cdot )_{\ell}$ are even for a suitable class
$\ell \in H_2(B,\BZ)$.
Let $C \subset B$ be a curve in class $\kk$.
We can assume $C$ is reduced and irreducible (otherwise prove the claim for each
reduced irreducible component).
By embedding $B$ into a projective space and choosing suitable hyperplane sections
we can find\footnote{
Assume $C \subset B \subset \p^n$, and let $K_d$ be the
the kernel of $H^0(\CO_{\p^n}(d)) \to H^0( \CO_{\p^n}(d)|_C)$
for $d \gg 0$.
For generic sections $f_1, \ldots, f_{m} \in K_d$, $m=\dim B - 1$
the intersection $\Sigma = B \cap_i V(f_i)$ is a curve which contains $C$.
The key step is to show $\Sigma = C + D$ for a smooth curve $D$ which does not contain $C$;
all other conditions follow from a usual Bertini argument.
To show that $\Sigma$ is of multiplicity $1$ at $C$, let $p \in C$ be a point at which $C$ is smooth
and consider the projectivized normal bundle $P$ of $C$ inside $B$ at $p$.
The set of $f_1, \ldots, f_m$ which vanish at
some $v \in P$ simultaneously is a closed co-dimension $m$ subset.
Since $\dim(P) = m-1$, by choosing $f_i$ generic we can guarantee the
tangent spaces to $\Sigma(f)$ and $C$ are the same at $p$;
hence the multiplicity of $C$ in $\Sigma$ is $1$.}
a curve $D \subset B$ not containing $C$ and
a deformation of $C \cup D$ to a curve $D'$,
such that $D, D'$ are smooth and $X_D$ and $X_{D'}$ are smooth elliptic surfaces over $D, D'$ respectively;
here we let $X_{\Sigma} = \pi^{-1}(\Sigma)$ for $\Sigma \subset B$.
Hence it suffices to show $( \cdot , \cdot )_\kk$ is even if $\kk$ is represented
by a smooth curve $C$ such that $X_C$ is smooth. Let $\alpha \in \Lambda$.
Since $\alpha|_{X_C}$ is of type $(1,1)$ and orthogonal to the section and fiber class
the claim now follows from the adjunction formula, see e.g. \cite[Thm.7.4]{Shioda}.
\end{proof}

The matrix of $-( \cdot , \cdot )_{\kk}$
with respect to the basis $\{ b_i \}$ is denoted by
\[ Q_{\kk} \in M_{n \times n}(\BZ). \]
Hence for all $v = (v_1, \ldots, v_n)$ and $v' = (v'_1, \ldots, v'_n)$ in $\BQ^n$ we have
\[ \Big( \sum_i v_i b_i\, , \sum_{i} v'_i b_i \Big)_{\kk} = - v^T Q_{\kk} v'. \]
If $\kk$ is a curve class, the matrix $Q_{\kk}$ has even diagonal entries.

\subsection{Gromov--Witten classes and conjectures}
\subsubsection{Definition}
Let $\beta \in H_2(X,\BZ)$ be a curve class,
let $\kk = \pi_{\ast} \beta \in H_2(B,\BZ)$ and let
\[ \Mbar_{g,n}(X,\beta) \]
be the moduli space of genus $g$ stable maps to $X$ in class $\beta$ with $n$ markings.

For all $g,n,\kk$ such that\footnote{
If $\kk=0$ and $2g-2+n \leq 0$
the moduli space $\Mbar_{g,n}(B, \kk)$ is empty,
but $\Mbar_{g,n}(X, \beta)$ for some $\beta>0$ with $\pi_{\ast} \beta = \kk$ may be non-empty. In this case no induced morphism exists.}
\[ \kk > 0 \quad \text{ or } \quad 
2g-2+n > 0
\]
the elliptic fibration $\pi$ induces a morphism
\[ \pi : \Mbar_{g,n}(X, \beta) \to \Mbar_{g,n}(B, \kk). \]
Consider cohomology classes 
\[ \gamma_1, \ldots, \gamma_n \in H^{\ast}(X). \]
We define the \emph{$\pi$-relative Gromov--Witten class}
\[ \CC^{\pi}_{g,\beta}(\gamma_1, \ldots, \gamma_n)
 = \pi_{\ast}\left(  \left[ \Mbar_{g,n}(X,\beta) \right]^{\text{vir}} \prod_{i=1}^{n} \ev_i^{\ast}(\gamma_i) \right) \in H_{\ast}( \Mbar_{g,n}(B, \kk) ).
\]

\subsubsection{Quasi-Jacobi forms}
Let $\kk \in H_2(B,\BZ)$ be a fixed class. Consider the generating series
\[
\CC^{\pi}_{g, \kk}(\gamma_1, \ldots, \gamma_n) = \sum_{\pi_{\ast} \beta = \kk} \CC^{\pi}_{g,\beta}(\gamma_1, \ldots, \gamma_n) q^{W \cdot \beta} \zeta^{\beta}
\]
where the sum is over all curve classes $\beta \in H_2(X,\BZ)$ with $\pi_{\ast} \beta = \kk$. 
By definition,
\[ \CC^{\pi}_{g, \kk}(\gamma_1, \ldots, \gamma_n) \in H_{\ast}( \Mbar_{g,n}(B, \kk) ) \otimes \BQ[[q^{\frac{1}{2}}, \zeta^{\pm 1}]]. \]

Recall the space $\QJ_{L}$ of quasi-Jacobi forms of index $L$, and let 
\[ \Delta(q) = q \prod_{m \geq 1} (1-q^m)^{24} \]
be the modular discriminant.
The following is a refinement of \cite[Conj.A]{HAE}. 

\begin{conjecture} \label{Conj_Quasimodularity}
The series $\CC^{\pi}_{g, \kk}(\gamma_1, \ldots, \gamma_n)$ is a cycle-valued quasi-Jacobi form
of index $Q_{\kk}/2$:
\[ 
\CC^{\pi}_{g, \kk}(\gamma_1, \ldots, \gamma_n)
\, \in \,
H_{\ast}(\Mbar_{g,n}(B, \kk)) \otimes
\frac{1}{\Delta(q)^m}
\QJ_{Q_{\kk}/2}
\]
where $m = -\frac{1}{2} c_1(N_{\iota}) \cdot \mathsf{k}$.
\end{conjecture}

\subsubsection{Holomorphic anomaly equation}
Recall the differential operator on $\QJac_L$ induced by the non-holomorphic variable $\nu$,
\[ \T_q = \frac{d}{dC_2} : \QJac_L \to \QJac_L. \]
Since $\Delta(q)$ is a modular form, we have
\[ \T_q \Delta(q) = 0. \]
We conjecture a holomorphic anomaly equation for the classes $\CC^{\pi}_{g,\kk}$. 
The equation is exactly the same as in \cite[Conj.B]{HAE}.

Consider the diagram
\[
\begin{tikzcd}
\Mbar_{g,n}(B,\mathsf{k}) & M_{\Delta} \ar[swap]{l}{\iota} \ar{d} \ar{r} & \Mbar_{g-1, n+2}(B, \mathsf{k}) \ar{d}{\ev_{n+1} \times \ev_{n+2}} \\
& B \ar{r}{\Delta} & B \times B
\end{tikzcd}
\]
where $\Delta$ is the diagonal, $M_{\Delta}$ is the fiber product
and $\iota$ is the gluing map along the last two points.
Similarly, for every
splitting $g = g_1 + g_2$, $\{ 1, \ldots, n \} = S_1 \sqcup S_2$
and $\mathsf{k} = \mathsf{k_1} + \mathsf{k_2}$
consider
\[
\begin{tikzcd}
\Mbar_{g,n}(B,\mathsf{k}) & M_{\Delta, \kk_1, \kk_2} \ar{d} \ar[swap]{l}{j} \ar{r} & 
\Mbar_{g_1, S_1 \sqcup \{ \bullet \}}(B, \mathsf{k_1})
\times \Mbar_{g_2, S_2 \sqcup \{ \bullet \}}(B, \mathsf{k_2}) \ar{d}{\ev_{\bullet} \times \ev_{\bullet}} \\
& B \ar{r}{\Delta} & B \times B
\end{tikzcd}
\]
where $M_{\Delta, \kk_1, \kk_2}$ is the fiber product and
$j$ is the gluing map along the marked points labeled by $\bullet$.

\begin{conjecture} \label{Conj_HAE} On $\Mbar_{g,n}(B, \mathsf{k})$,
\begin{align*}
\T_q \CC^{\pi}_{g,\kk}( \gamma_1, \ldots, \gamma_n )
\ =\  & 
\iota_{\ast} \Delta^{!}
\CC^{\pi}_{g-1,\kk}( \gamma_1, \ldots, \gamma_n, \1, \1 ) \\
&
+ \sum_{\substack{ g= g_1 + g_2 \\
\{1,\ldots, n\} = S_1 \sqcup S_2 \\
\mathsf{k} = \mathsf{k}_1 + \mathsf{k}_2}}
j_{\ast} \Delta^{!} \left(
\CC^{\pi}_{g_1, \kk_1}( \gamma_{S_1}, \1) \boxtimes
\CC^{\pi}_{g_2, \mathsf{k}_2}( \gamma_{S_2}, \1 ) \right) \\
&
- 2 \sum_{i=1}^{n} 
\CC^{\pi}_{g,\kk}( \gamma_1, \ldots, \gamma_{i-1},
 \pi^{\ast} \pi_{\ast} \gamma_i , \gamma_{i+1}, \ldots, \gamma_n ) \cdot \psi_i,
\end{align*}
where $\psi_i \in H^2(\Mbar_{g,n}(B,\kk))$
is the cotangent line class at the $i$-th marking.
\end{conjecture}

Since the moduli space of stable maps in negative genus is empty,
the corresponding terms in Conjecture~\ref{Conj_HAE} vanish.
Further, the sum in the second term on the right runs over all splittings for which the
moduli spaces $\Mbar_{g_i, |S_i|+1}(B, \kk_i)$ are stable, or equivalently, for which the classes
$\CC_{g_i,\kk_i}^{\pi}(\gamma_{S_i}, \1)$ are defined.
In particular, if $g_i = 0$ and $\kk_i = 0$ we require $|S_i| \geq 2$. 

By Section~\ref{Subsubsection_Specialization_to_quasimodular_forms} quasi-Jacobi forms specialize to quasi-modular forms
under $\zeta = 1$, and the specialization map commutes with $\T_q$.
Hence Conjectures~\ref{Conj_Quasimodularity} and \ref{Conj_HAE}
generalize and are compatible with \cite[Conj.A and B]{HAE}. 

We have always assumed here that the elliptic fibration
has integral fibers, a section, and $h^{2,0}(X) = 0$,
see \mbox{(i-iii)} in Section~\ref{Subsubsection_ellip_fibr_definition}.
We expect Conjectures~\ref{Conj_Quasimodularity} and~\ref{Conj_HAE}
hold without these assumption if some modifications are made:
It is plausible (i) can be removed without any modifications. 
If we remove (ii) we need to work with a multi-section of the fibration,
which leads to quasi-Jacobi forms which are modular
with respect to $\Gamma(N)$ where
$N$ is the degree of a multisection over the base. 
If (iii) is violated then the Gromov--Witten theory of $X$ mostly vanishes
by a Noether--Lefschetz argument.
Using instead a nontrivial reduced Gromov--Witten theory
(such as \cite{KT} for algebraic surfaces satisfying $h^{2,0} > 0$)
forces then some basic modifications to the holomorphic anomaly equation,
see e.g. Section~\ref{Section_Abelian_surfaces} for the case of the abelian surface.

\section{Consequences of the conjectures} \label{Section_Consequences_of_the_Conjecture}
\subsection{A weight refinement} \label{Subsection_weight_refinement}
Define a modified degree function $\underline{\deg}(\gamma)$
on $H^{\ast}(X)$ by the assignment
\begin{equation*} \label{modified_degree_function}
\underline{\deg}(\gamma) =
\begin{cases}
2 & \text{if } \gamma \in \mathrm{Im}(T_+) \\
1 & \text{if } \gamma \in \mathrm{Ker}( T_+ ) \cap \mathrm{Ker}(T_-) \\
0 & \text{if } \gamma \in \mathrm{Im}(T_-) \\
\end{cases}
\end{equation*}

The following is parallel to \cite[Appendix B]{HAE}.

\begin{lemmastar} \label{Lemma_Weight_statement}
Assume Conjectures~\ref{Conj_Quasimodularity} and \ref{Conj_HAE} hold.
Then for any $\underline{\deg}$-homogeneous classes
$\gamma_1, \ldots, \gamma_n \in H^{\ast}(X)$
and $\kk \in H_2(B,\BZ)$ we have
\[
\CC^{\pi}_{g, \kk}(\gamma_1, \ldots, \gamma_n)
\, \in \,
H_{\ast}(\Mbar_{g,n}(B, \kk)) \otimes \frac{1}{\Delta(q)^m} \QJ_{\ell, Q_{\kk}}
\]
where $m = -\frac{1}{2} c_1(N_{\iota}) \cdot \mathsf{k}$
and $\ell = 2g - 2 + 12m + \sum_i \underline{\deg}(\gamma_i)$.
\end{lemmastar}

\subsection{Disconnected Gromov--Witten classes} \label{Subsection_Disconnected_GW_classes}
We reformulate the holomorphic anomaly equation of Conjecture~\ref{Conj_HAE} for disconnected Gromov--Witten classes.
Let
\[ \Mbar_{g,n}^{\bullet}(B,\kk) \]
be the moduli space of stable maps $f : C \to B$ from possibly \emph{disconnected}
curves of genus $g$ in class $\kk$, with the requirement
that for every connected component $C' \subset C$ 
at least one of the following holds:
\begin{enumerate}
\item[(i)] $f|_{C'}$ is non-constant, 
or
\item[(ii)] $C'$ has genus $g'$ and carries $n'$ markings with $2g'-2+n' > 0$.
\end{enumerate}
Let $\Mbar_{g,n}'(X,\beta)$
be the moduli space of stable maps $f : C \to X$ from possibly disconnected
curves of genus $g$ in class $\beta$, with the requirement
that for every connected component $C' \subset C$ 
at least one of the following holds:
\begin{enumerate}
    \item[(i)] $\pi \circ f|_{C'}$ is non-constant, or 
    \item[(ii)] $C'$ has genus $g'$ and carries $n'$ markings with $2g'-2+n' > 0$.
\end{enumerate}
For all\footnote{Here $\Mbar_{g,n}^{\bullet}(B,\kk)$ is empty if and only if $\Mbar_{g,n}'(X,\beta)$ is empty,
so we do not need to exclude any values of $(g,\kk)$.}
$g \in \BZ$ and curve classes $\kk$ the fibration $\pi$ induces a map
\[ \pi : \Mbar_{g,n}'(X,\beta) \to \Mbar_{g,n}^{\bullet}(B,\kk). \]
Define the disconnected Gromov--Witten classes by
\[
\CC_{g,\kk}^{\pi, \bullet}(\gamma_1, \ldots, \gamma_n)
=
\sum_{\pi_{\ast} \beta = \kk} \zeta^{\beta} q^{W \cdot \beta}
\pi_{\ast} \left( \left[ \Mbar'_{g,n}(X,\beta) \right]^{\text{vir}} \prod_i \ev_i^{\ast}(\gamma_i) \right).
\]
The right hand side is a series with coefficients in the homology of $\Mbar_{g,n}^{\bullet}(B, \kk)$.

Since the disconnected classes $\CC^{\bullet}_{g,k}$ can be expressed in terms of connected classes $\CC^{\bullet}_{g,k}$ and vice versa, Conjecture~\ref{Conj_Quasimodularity} is equivalent to
the quasi-Jacobi property of the disconnected theory:
\[
\CC_{g,\kk}^{\pi, \bullet}(\gamma_1, \ldots, \gamma_n)
\in H_{\ast}( \Mbar_{g,n}^{\bullet}(B, \kk) ) \otimes
\frac{1}{\Delta(q)^m} \QJ_{Q_{\kk}/2}
\]
where $m = -\frac{1}{2} c_1(N_{\iota}) \cdot \mathsf{k}$.
Similarly, Conjecture~\ref{Conj_HAE} is equivalent to the
following disconnected version of the holomorphic anomaly equation:
\begin{lemmastar} \label{Lemma_discHAE}Conjecture~\ref{Conj_HAE} is equivalent to
\begin{align*}
\T_q \CC^{\pi, \bullet}_{g,\kk}( \gamma_1, \ldots, \gamma_n )
\ =\  & 
\iota_{\ast} \Delta^{!}
\CC^{\pi, \bullet}_{g-1,\kk}( \gamma_1, \ldots, \gamma_n, \1, \1 ) \\
&
- 2 \sum_{i=1}^{n} 
\psi_i \cdot \CC^{\pi, \bullet}_{g,\kk}( \gamma_1, \ldots, \gamma_{i-1},
 \pi^{\ast} \pi_{\ast} \gamma_i , \gamma_{i+1}, \ldots, \gamma_n ).
\end{align*}
\end{lemmastar}

\subsection{Elliptic holomorphic anomaly equation} \label{Subsection_EHAE}
Recall the anomaly operator with respect to the elliptic parameter:
\[ \T_{\lambda} : \QJac_{k,L} \to \QJac_{k-1, L}, \quad \lambda \in \Lambda \]
(recall we identify $\Lambda$ with $\BZ^n$ here).
The anomaly equation of $\CC_g(\ldots)$ with respect to the operator $\T_{\lambda}$ reads as follows.

\begin{lemmastar} \label{Lemma_EllHAE}Assume Conjectures~\ref{Conj_Quasimodularity} and \ref{Conj_HAE} hold.
Then
\[
\T_{\lambda} \CC^{\pi}_{g,k}(\gamma_1, \ldots, \gamma_n) =
\sum_{i=1}^{n} \CC^{\pi}_{g,k}(\gamma_1, \ldots, \gamma_{i-1}, A(\lambda) \gamma_i, \gamma_{i+1}, \ldots, \gamma_n),
\]
for any $\lambda \in \Lambda$, where $A(\lambda) : H^{\ast}(X) \to H^{\ast}(X)$ is defined by
\[ A(\lambda) \gamma = \lambda \cup \pi^{\ast} \pi_{\ast}( \gamma) - \pi^{\ast} \pi_{\ast}( \lambda \cup \gamma ), \quad \gamma \in H^{\ast}(X). \]
\end{lemmastar}
\begin{proof}
Let $\lambda \in \Lambda$ and recall from Section~\ref{Subsubsection_quasiJacobi_Differentialoperators} the commutation relation
\[ \left[ \T_q , D_{\lambda} \right] = -2 \T_{\lambda}. \]

Let $p : \Mbar_{g,n+1}(B, \kk) \to \Mbar_{g,n}(B,\kk)$ be the map that forgets the last marked point.
We have
\[ D_{\lambda} \CC_{g,\kk}^{\pi}(\gamma_1, \ldots, \gamma_n) = p_{\ast} \CC_{g,\kk}^{\pi}(\gamma_1, \ldots, \gamma_n, \lambda). \]
Hence we obtain
\[
-2 \T_{\lambda}  \CC^{\pi}_{g,k}(\gamma_1, \ldots, \gamma_n)
=
p_{\ast} \T_q \CC_{g,\kk}(\gamma_1, \ldots, \gamma_n, \lambda)
- D_{\lambda} \T_q \CC_{g,\kk}(\gamma_1, \ldots, \gamma_n).
\]
Only two terms contribute in this difference. The first arises from the second term in the holomorphic anomaly equation
on $\Mbar_{g,n+1}(B, \kk)$. The summand with $g_i = 0$ and $n+1 \in S_i$ with $|S_i| = 2$ contributes
\[ 2 \sum_{i=1}^{n} \CC_{g,\kk}^{\pi}(\gamma_1, \ldots, \pi^{\ast} \pi_{\ast}( \gamma_i \cup \lambda ), \ldots, \gamma_n ). \]
The second contribution arises from the third term of the holomorphic anomaly equation when comparing the classes $\psi_{i}$ under pullback by $p$. It is
\[ -2 \sum_{i=1}^{n} \CC_{g,\kk}^{\pi}(\gamma_1, \ldots, \lambda \cup \pi^{\ast} \pi_{\ast}( \gamma_i ), \ldots, \gamma_n ). \]
Adding up yields the claim.
\end{proof}

Consider the exponential $\exp(A(\lambda))$ which acts on $\gamma \in H^{\ast}(X)$ by
\[
(\exp A(\lambda) )\gamma = \gamma + \lambda \cup \pi^{\ast} \pi_{\ast}( \gamma) - \pi^{\ast} \pi_{\ast}( \lambda \cup \gamma )
- \frac{1}{2} \pi^{\ast}\left( \pi_{\ast}(\lambda^2) \cdot \pi_{\ast}(\gamma) \right).
\]
Lemma*~\ref{Lemma_EllHAE} then yields
\begin{align*}
\exp( \T_{\lambda} ) \CC^{\pi}_{g,k}(\gamma_1, \ldots, \gamma_n)
& =
\CC_{g,k}( \exp(A(\lambda)) \gamma_1, \ldots, \exp(A(\lambda)) \gamma_n).
\end{align*}
We will see in Section~\ref{Section_The_elliptic_transformation_law} how in good situations this is related to the automorphism defined
by adding the section corresponding to the class $\lambda$.

\subsection{The elliptic transformation law}
\label{Section_The_elliptic_transformation_law}
Recall the projection $p_{\perp}$ to the lattice $\Lambda$ from Section~\ref{Subsubsection_LatticeLambda}.
Throughout Section~\ref{Section_The_elliptic_transformation_law} we assume that the fibration
$\pi : X \to B$ satisfies the following condition,
which holds for example for the rational elliptic surface:

\vspace{4pt}
\noindent
\textbf{Assumption ($\star$).} For every $\lambda \in \Lambda$ there is a unique section $B_{\lambda} \subset X$
such that $p_{\perp}( [ B_{\lambda} ] ) = \lambda$.
\vspace{4pt}

Let $\lambda \in \Lambda$ and consider the morphism
\[ t_{\lambda} : X \to X,\ x \mapsto (x + B_{\lambda}(\pi(x))) \]
of fiberwise addition with $B_{\lambda}$.
Since $\pi \circ t_{\lambda} = \pi$ this implies
\[
\CC^{\pi}_{g, t_{\lambda \ast} \beta}(t_{\lambda \ast} \gamma_1, \ldots, t_{\lambda \ast} \gamma_n)
= \CC^{\pi}_{g, \beta}(\gamma_1, \ldots, \gamma_n).
\]
Let us write $\CC_{g,\kk}^{\pi}(\ldots)(z)$ to denote the dependence
of $\CC_{g,\kk}^{\pi}(\ldots)$ on the variable $z \in \Lambda \otimes \BC$. 
From the last equation we obtain
\begin{align*}
& \CC^{\pi}_{g, \kk}(\gamma_1, \ldots, \gamma_n)(z) \\
& = \sum_{\pi_{\ast} \beta = \kk}
\CC^{\pi}_{g,\beta}(t_{\lambda \ast} \gamma_1, \ldots,
t_{\lambda \ast} \gamma_n) q^{(t_{\lambda \ast} W) \cdot \beta} e\big( (t_{\lambda \ast} z) \cdot \beta \big) \\
& = e\left( -\frac{ \tau}{2} \pi_{\ast}(\lambda^2) \cdot \kk
- \pi_{\ast}( z \cdot \lambda) \cdot \kk \right)
\CC^{\pi}_{g,\kk}(t_{\lambda \ast} \gamma_1, \ldots, t_{\lambda \ast} \gamma_n)
(z + \lambda \tau) \\
& = e\left( \frac{ \tau}{2} \lambda^T Q_{\kk} \lambda + \lambda^T Q_{\kk} z \right)
\CC^{\pi}_{g,\kk}(t_{\lambda \ast} \gamma_1, \ldots, t_{\lambda \ast} \gamma_n)
(z + \lambda \tau).
\end{align*}
Rearranging the terms slightly yields
\begin{multline} \label{ffsdghhjjjhhh}
\CC^{\pi}_{g,\kk}(\gamma_1, \ldots, \gamma_n)
(z + \lambda \tau)
 \\ =
e\left( -\frac{1}{2} \lambda^T Q_{\kk} \lambda - \lambda^T Q_{\kk} z \right)
\CC^{\pi}_{g,\kk}(t_{-\lambda \ast} \gamma_1, \ldots, t_{-\lambda \ast} \gamma_n)(z).
\end{multline}
We obtain the following.

\begin{lemma} \label{Lemma_elliptic_transf_law}
Assume $\pi : X \to B$ satisfies Assumption ($\star$). If every $\gamma_i$ is translation invariant,
i.e. $t_{\lambda \ast} \gamma_i = \gamma_i$ for all $\lambda \in \Lambda$,
then $\CC^{\pi}_{g,\kk}(\gamma_1, \ldots, \gamma_n)$ satisfies the elliptic transformation law of Jacobi forms:
\[
\CC^{\pi}_{g,\kk}(\gamma_1, \ldots, \gamma_n)
(z + \lambda \tau)
=
e\left( -\frac{1}{2} \lambda^T Q_{\kk} \lambda - \lambda^T Q_{\kk} z \right)
\CC_{g,\kk}^{\pi}(\gamma_1, \ldots, \gamma_n)(z)
\]
for all $\lambda \in \Lambda$.
\end{lemma}

Even if the $\gamma_i$ are not translation invariant we
have the following relationship to the transformation law of quasi-Jacobi forms.
Recall the endomorphism $A(\lambda)$ from Section~\ref{Subsection_EHAE}.
For the rational elliptic surface we have\footnote{It would be interesting
to know for which elliptic fibrations
\eqref{dfgfdgfgfd} holds.
}
\begin{equation} t_{\lambda \ast} = \exp A(\lambda) \label{dfgfdgfgfd} \end{equation}
for all $\lambda \in \Lambda$.
Assuming 
Conjectures~\ref{Conj_Quasimodularity} and \ref{Conj_HAE} we can rewrite \eqref{ffsdghhjjjhhh} as
\begin{align*}
& \CC^{\pi}_{g,\kk}(\gamma_1, \ldots, \gamma_n)
(z + \lambda \tau) \\
=\, &
e\left( -\frac{1}{2} \lambda^T Q_{\kk} \lambda - \lambda^T Q_{\kk} z \right)
\CC^{\pi}_{g,\kk}(\exp(A(-\lambda)) \gamma_1, \ldots, \exp(A(-\lambda)) \gamma_n) \\
=\, &
e\left( -\frac{1}{2} \lambda^T Q_{\kk} \lambda - \lambda^T Q_{\kk} z \right)
\exp( - \T_{\lambda}) \CC_{g,\kk}^{\pi}(\gamma_1, \ldots, \gamma_n),
\end{align*}
which is the elliptic transformation law of quasi-Jacobi forms stated in Lemma~\ref{Lemma_ell_trans_law_for_quasi_Jac}.

\subsection{Quasi-modular forms} \label{Subsubsection_elliptic_fibrations_quasimodularforms}
The elliptic periods (i.e. $\zeta^{\alpha}$-coefficients)
of a quasi-Jacobi form are quasimodular forms,
see Proposition~\ref{QJac_Prop2}.
Together with Conjecture~\ref{Conj_Quasimodularity} this leads to a basic quasi-modularity statement for elliptic fibrations as follows.
Let $\kk \in H_2(B,\BZ)$ be a curve class, and consider the pairing on $H^2(X,\BZ)$ defined by
\begin{equation}
\quad ( \alpha , \alpha' )_{\kk} = \int_{\kk} \pi_{\ast}(\alpha \cdot \alpha') \quad \text{ for all } 
\alpha, \alpha' \in H^2(X,\BZ).
\label{pairingg}
\end{equation}
Throughout Section~\ref{Subsubsection_elliptic_fibrations_quasimodularforms} we make the
following assumption which is equivalent to the positive definiteness of $Q_{\kk}$
and holds in many cases of interest\footnote{On
an elliptic surface satisfying $h^{2,0} = 0$ the assumption holds by the Hodge index theorem whenever $k \neq 0$.
}.

\vspace{4pt}
\noindent \textbf{Assumption ($\dag$).} The restriction of $( \cdot, \cdot )_{\kk}$ to $\Lambda$ is negative-definite.

\vspace{7pt}
Consider the cohomology classes on $B$ orthogonal to $\kk$,
\[ \kk^{\perp} = \left\{ \gamma \in H^2(B,\BZ)\, \middle|\, \langle \gamma, \kk \rangle = 0 \right\} \]
where $\langle \cdot , \cdot \rangle$ is the pairing between cohomology and homology on $B$.
Consider also the null space of $( \cdot , \cdot )_\kk$,
\[ N_{\kk} = \left\{ v \in H^2(X,\BZ)\, \middle| \, (v, H^2(X,\BZ))_\kk = 0 \right\}. \]
We have $\pi^{\ast} \kk^{\perp} \subset N_{\kk}$. By assumption ($\dag$) this inclusion is an equality,
\[ N_{\kk} = \pi^{\ast} \kk^{\perp}, \]
and the induced pairing on $H^2(X,\BZ) / N_{\kk}$ is of signature $(1, n+1)$.\footnote{The combination of both statements is equivalent to Assumption ($\dag$).}

The dual of $H^2(X,\BZ) / N_{\kk}$ is naturally identified with the lattice
\[ L_{\kk} = \left\{ \beta \in H_2(X,\BZ)\, \middle| \, \pi_{\ast} \beta = c \cdot \kk \text{ for some } c \in \BQ \right\}. \]
The non-degenerate pairing on $H^2(X,\BZ) / N_{\kk}$ induces a non-degenerate pairing on $L_{\kk}$
which we denote by $( \cdot , \cdot )_{\kk}$ as well.

For any $\alpha \in H_2(X,\BZ) / \BQ F$ with $\pi_{\ast} \alpha = \kk$ consider the theta series
\[
\CC^\pi_{g,\alpha}(\gamma_1, \ldots, \gamma_n)
=
\sum_{[ \beta ] = \alpha }
\CC^\pi_{g,\beta}(\gamma_1, \ldots, \gamma_n) q^{-\frac{1}{2} \langle \beta, \beta \rangle_{\kk}} 
\]
where the sum is over all curve classes $\beta$ with residue class $\alpha$ in $H_2(X,\BZ)/ \BQ F$.

\begin{lemmastar} \label{Lemma_Quasiodularformselli} Assume Conjecture~\ref{Conj_Quasimodularity} and \ref{Conj_HAE}, and Assumption \textup{($\dag$)}.
Let $\ell$ be the smallest positive integer such that $\ell Q_{\kk}^{-1}$ has integral entries and even diagonal.
Then every $\CC^\pi_{g,\alpha}(\gamma_1, \ldots, \gamma_n)$
is a cycle-valued weakly-holomorphic quasi-modular form of level $\ell$.
\end{lemmastar}

The Lemma shows that although the elliptic fibration $\pi : X \to B$ has a section,
we should expect the generating series of Gromov--Witten invariants in the fiber direction
to be quasi-modular of higher level (with pole at cusps).
It is remarkable that these higher-index quasi-modular forms
when arranged together appropriately should form $\mathrm{SL}_2(\BZ)$-quasi-Jacobi forms.

If $Q_{\kk}$ is unimodular then we obtain level $1$, hence $\mathrm{SL}_2(\BZ)$-quasi-modular forms in Lemma*~\ref{Lemma_Quasiodularformselli}.
For the rational elliptic surface the level of the quasi-modular form
is exactly the degree over the base curve.
This compares well with the conjectural quasi-modularity of the Gromov--Witten invariants of K3 surfaces in inprimitive classes, see \cite[Sec.7.5]{MPT}.

Using Proposition~\ref{QJac_Prop2} (ii)
the holomorphic anomaly equation for the quasi-Jacobi classes $\CC^\pi_{g,\kk}(\ldots )$
yields a holomorphic anomaly equation for the theta-series $\CC^\pi_{g, \alpha}( \ldots )$.
However, in the non-unimodular case the result is rather complicated and difficult to handle.\footnote{The unimodular
case is further discussed in Section~\ref{Section_Abelian_surfaces}.}
The holomorphic anomaly equation takes its simplest form for quasi-Jacobi forms.

\begin{proof}[Proof of Lemma~\ref{Lemma_Quasiodularformselli}]
Let $\lambda$ be the image of $\alpha$ in $H_{2, \perp}$.
A computation yields
\[
\CC^\pi_{g,\alpha}(\gamma_1, \ldots, \gamma_n)
=
q^{ - \frac{1}{4} \lambda^T L^{-1} \lambda }\left[ \CC^\pi_{g, \kk} (\gamma_1, \ldots, \gamma_n) \right]_{\zeta^{\lambda}}
\]
which implies the Lemma by Proposition~\ref{QJac_Prop2}.
\end{proof}

\subsection{Calabi--Yau threefolds}
Let $\pi : X \to B$ be an elliptically fibered Calabi--Yau threefold with section $\iota : B \to X$
and $h^{2,0}(X) = 0$.
The moduli space of stable maps is of virtual dimension $0$.
For all $(g,\kk) \notin \{ (0,0), (1,0) \}$ define the Gromov--Witten potential
\[ \F_{g,\kk}(q, \zeta)
= \int_{\Mbar_{g}(B,\kk)} \CC^{\pi}_{g,\kk}()
= \sum_{\pi_{\ast} \beta = \kk} q^{W \cdot \beta} \zeta^{\beta} \int_{[ \Mbar_{g}(X,\beta) ]^{\text{vir}}} 1.
\]
By the Calabi--Yau condition we have $N_{\iota} \cong \omega_B$. Hence Conjecture~\ref{Conj_Quasimodularity} implies
\[ \F_{g,\kk}(q) \in \frac{1}{\Delta(q)^{-\frac{1}{2} K_B \cdot \kk}} \QJac. \]
We have the following holomorphic anomaly equation (see also \cite[0.5]{HAE}).
\begin{propstar} \label{Prop_HAE_for_CY3}
Assume Conjectures \ref{Conj_Quasimodularity} and~\ref{Conj_HAE}. Then we have
\[
\T_q \F_{g,\kk}
=
\langle \kk + K_B, \kk \rangle \F_{g-1, \kk}
+ \sum_{\substack{g=g_1+g_2 \\ \kk = \kk_1 + \kk_2}}
\langle \kk_1, \kk_2 \rangle \F_{g_1, \kk_1} \F_{g_2, \kk_2}
+ \frac{\delta_{g2} \delta_{k0}}{4} \langle K_B, K_B \rangle.
\]
where we let $\langle - , - \rangle$ denote the intersection pairing on $B$,
the first term on the right is defined to vanish if $(g,\kk) = (2,0)$,
and the sum is over all values $(g_i, \kk_i)$ for which $\F_{g_i, \kk_i}$ is defined.
\end{propstar}
\begin{proof}
If $\kk > 0$ or $g > 2$ Conjecture~\ref{Conj_HAE} implies
\begin{align*}
\T_q \F_{g,\kk}
& = \int \CC_{g-1,\kk}(\pi^{\ast} \Delta_B)
+ \sum_{\substack{g=g_1+g_2 \\ \kk = \kk_1 + \kk_2}} \sum_j
\int \CC_{g_1, \kk_1}( \pi^{\ast} \Delta_{B,j} )
\cdot \int \CC_{g_2, \kk_2}(\pi^{\ast} \Delta_{B,j}^{\vee}) \\
& =
\langle \kk, \kk \rangle \F_{g-1,\kk}
+
\sum_{\substack{g=g_1+g_2 \\ \kk = \kk_1 + \kk_2 \\ \kk_1, \kk_2 > 0}}
\langle \kk_1, \kk_2 \rangle \F_{g_1, \kk_1} \F_{g_2, \kk_2}\\
& 
+ 2 \sum_j \int \CC_{g-1, \kk}(\pi^{\ast} \Delta_{B, j})
\cdot \int_{[\Mbar_{1,1}(X,0)]^{\text{vir}}} \ev_1^{\ast}( \pi^{\ast} \Delta_{B,j}^{\vee})
\end{align*}
where we have written 
\[ \Delta_B = \sum_j \Delta_{B,j} \boxtimes \Delta_{B,j}^{\vee} \in H^{\ast}(B^2) \]
for the K\"unneth decomposition of the diagonal of $B$.
By \cite{GP} we have
\[ [ \Mbar_{1,1}(X,0) ]^{\text{vir}} = (c_3(X) - c_2(X) \lambda_1) \cap [ \Mbar_{1,1} \times X ] \]
and by \cite[Sec.4]{AHR} we have
\[ c_2(X) = \pi^{\ast} (c_2(B) + c_1(B)^2) + 12 \iota_{\ast} c_1(B). \]
Hence we find
\[ \int_{[\Mbar_{1,1}(X,0)]^{\text{vir}}} \ev_1^{\ast}( \pi^{\ast} \Delta_{B,j}^{\vee}) = -\frac{1}{2} \langle \Delta_{B,j} , c_1(B) \rangle \]
from which we obtain
\[ \T_q \F_{g,\kk}
=
\langle \kk + K_B, \kk \rangle \F_{g-1, \kk}
+ \sum_{\substack{g=g_1+g_2 \\ \kk = \kk_1 + \kk_2}}
\langle \kk_1, \kk_2 \rangle \F_{g_1, \kk_1} \F_{g_2, \kk_2}.
\]

If $(g,\kk) = (2,0)$ Conjecture B yields
\begin{align*} \T_q \F_{2,0}(q)
& =  \sum_j \int_{[\Mbar_{1,1}(X,0)]^{\text{vir}}} \ev_1^{\ast}( \pi^{\ast} \Delta_{B,j}) \cdot \int_{[\Mbar_{1,1}(X,0)]^{\text{vir}}} \ev_1^{\ast}( \pi^{\ast} \Delta_{B,j}^{\vee}) \\
& = \frac{1}{4} \int_B c_1(B)^2.  \qedhere
\end{align*}
\end{proof}

It will be useful later on to consider the disconnected case as well.
For any $g \in \BZ$ and $\kk \in H_2(B,\BZ)$ let
\[ 
\F_{g,\kk}^{\bullet} 
= \int_{\Mbar^{\bullet}_{g}(B, \kk)} \CC^{\bullet}_{g,\kk}()
= \sum_{\pi_{\ast} \beta = \kk} q^{W \cdot \beta} \zeta^{\beta} \int_{ [ \Mbar'_{g}(X,\beta) ]^{\text{vir}} } 1.
\]
The connected and disconnected potentials are related by
\begin{equation} \label{Dis_Con_rel}
\sum_{g, \kk} \F_{g,\kk}^{\bullet} u^{2g-2} t^{\kk}
=
\exp\left( \sum_{(g,\kk) \notin \{ (0,0), (1,0) \} } \F_{g,\kk} u^{2g-2} t^{\kk} \right).
\end{equation}
A direct calculation using \eqref{Dis_Con_rel} and Proposition*~\ref{Prop_HAE_for_CY3} implies the following disconnected holomorphic anomaly equation
\begin{equation} \label{352rw}
\T_q \F_{g,\kk}^{\bullet} =
\left\langle \kk + \frac{1}{2} K_B, \kk + \frac{1}{2} K_B \right\rangle \F_{g-1,\kk}^{\bullet}.
\end{equation}

\section{Relative geometries}
\label{Section_relative_geomtries}
\subsection{Relative divisor}
Let $\pi : X \to B$ be an elliptic fibration with section and integral fibers
such that $H^{2,0}(X) = 0$.
Let
\[ D \subset X. \]
be a non-singular divisor. We assume $\pi$ restricts to an elliptic fibration
\[ \pi_D : D \to A \]
for a non-singular divisor $A \subset B$.
The section of $\pi$ restricts to a section of $\pi_D$. Since $\pi$ has integral fibers, so does $\pi_D$.
We have the fibered diagram
\[
\begin{tikzcd}
D \ar{d}{\pi_D} \ar[hookrightarrow]{r} & X \ar{d}{\pi} \\
A \ar[hookrightarrow]{r} & B.
\end{tikzcd}
\]

\subsection{Relative classes}
Let $\eta = (\eta_i)_{i=1, \ldots, l(\eta)}$ be an ordered partition. Let
\[ \Mbar_{g,n}(X/D, \beta ; \eta) \]
be the moduli space parametrizing stable maps from connected genus $g$ curves to $X$ relative to $D$
with ordered ramification profile $\eta$ over the relative divisor $D$, see \cite{Junli1, Junli2}
for definitions and \cite[Sec.2]{GV} for an introduction to relative stable maps.
We have evaluation maps at the $n$ interior and the $l(\eta)$ relative marked points. The latter are denoted by
\[ \ev^{\mathrm{rel}}_{i} : \Mbar_{g,n}(X/D, \beta ; \eta) \to D, \ i=1, \ldots, l(\eta). \]
Since $D$ is non-singular, we have the induced morphism
\[
\pi : \Mbar_{g,n}(X/D, \beta; \eta) \to \Mbar_{g,n}(B/A, \kk ; \eta)
\]
where $\kk = \pi_{\ast} \beta$.

Let $\gamma_1, \ldots, \gamma_n \in H^{\ast}(X)$, let
$\kk \in H_2(B,\BZ)$ be a curve class and let
\[
\underline{\eta} = \big( (\eta_1, \delta_1), \ldots, (\eta_{l(\eta)}, \delta_{l(\eta)} ) \big), \quad \delta_i \in H^{\ast}(D),
\]
be an ordered cohomology weighted partition. Define the relative potential
\begin{multline*}
\CC_{g,\kk}^{\pi/D}( \gamma_1, \ldots, \gamma_n ; \underline{\eta} ) \\
=
\sum_{\pi_{\ast} \beta = \kk} \zeta^{\beta} q^{W \cdot \beta}
\pi_{\ast} \left( \left[  \Mbar_{g,n}(X/D, \beta; \eta) \right]^{\text{vir}} \prod_{i=1}^{n} \ev_i^{\ast}(\gamma_i) \prod_{i=1}^{l(\eta)} \ev_i^{\mathrm{rel} \ast}(\delta_i) \right)
\end{multline*}
where as before $W = [ \iota(B) ] - \frac{1}{2} \pi^{\ast} c_1(N_{\iota})$
and $\zeta^{\beta} = e(z \cdot \beta)$ with $z \in \Lambda \otimes \BC$.

In line with the rest of the paper we conjecture the following.
\begin{conjecture} \label{Conj_RelQuasimodularity} The series
$\CC_{g,\kk}^{\pi/D}( \gamma_1, \ldots, \gamma_n ; \underline{\eta} )$
is a cycle-valued quasi-Jacobi form of index $Q_{\kk}/2$:
\[
\CC_{g,\kk}^{\pi/D}( \gamma_1, \ldots, \gamma_n ; \underline{\eta} )
\in H_{\ast}(\Mbar_{g,n}(B/A, \kk ; \eta)) \otimes \frac{1}{\Delta(q)^m} \QJ_{Q_{\kk}/2}
\]
where $m = -\frac{1}{2} c_1(N_{\iota}) \cdot \mathsf{k}$.
\end{conjecture}

\subsection{Rubber classes}
Stating the holomorphic anomaly equation for relative classes requires rubber classes.
Let $N$ be the normal bundle of $D$ in $X$, and consider the projective bundle
\[ \p( N \oplus \CO_D ) \to D. \]
We let
\[ D_0, D_{\infty} \subset \p( N \oplus \CO_D ) \]
be the sections corresponding to the summands $\CO_D$ and $N$ respectively.

The group $\BC^{\ast}$ acts naturally on $\p( N \oplus \CO_D )$ by scaling in the fiber direction,
and induces an action on the moduli space of stable maps relative to both divisors denoted by
\[
 \Mbar_{g,n}( \p( N \oplus \CO_D ) / \{ D_0, D_{\infty} \}, \beta ; \lambda, \mu)
\]
where the ordered partitions $\lambda, \mu$ are the ramification profiles at $D_0$ and $D_{\infty}$ respectively.
We let
\[ \Mbar_{g,n}^{\sim}( \p( N \oplus \CO_D ) / \{ D_0, D_{\infty} \}, \beta ; \lambda, \mu) \]
denote the corresponding space of stable maps to the rubber target \cite{MP}.

Let $N'$ be the normal bundle to $A$ in $B$ and consider the relative geometry
\[ \p(N' \oplus \CO_{A}) / \{ A_0, A_{\infty} \}. \]
Since $D$ is non-singular the fibration $\pi$ induces a well-defined map
\[
\rho : \p(N \oplus \CO_{D}) \to \p(N' \oplus \CO_{A})
\]
which is an elliptic fibration with section and integral fibers. Let
\begin{multline*} \rho : \Mbar_{g,n}^{\sim}( \p( N \oplus \CO_D ) / \{ D_0, D_{\infty} \}, \beta; \lambda, \mu) \\
\to 
\Mbar_{g,n}^{\sim}( \p( N' \oplus \CO_A ) / \{ A_0, A_{\infty} \}, \kk ; \lambda, \mu)
\end{multline*}
be the induced map. We also let $\ev_i^{\text{rel }0}$ and $\ev_i^{\text{rel }\infty}$ denote the evaluation maps
at the relative marked points mapping to $D_0$ and $D_{\infty}$ respectively.
Because of the rubber target, the evaluation maps of the moduli space at the interior marked points take value in $D$.

For any $\gamma_1, \ldots, \gamma_n \in H^{\ast}(D)$ and any
ordered weighted partitions
\[ \underline{\lambda} =  \big( (\lambda_i, \delta_i) \big)_{i=1, \ldots, l(\lambda)}, \quad
 \underline{\mu} =  \big( (\mu_i, \epsilon_i) \big)_{i=1, \ldots, l(\mu)},
\quad \delta_i, \epsilon_i \in H^{\ast}(D) 
\]
we define
\begin{multline*}
\CC^{\rho, \mathrm{rubber}}_{g,\kk}(\gamma_1, \ldots, \gamma_n ; \underline{\lambda}, \underline{\mu} ) \\
=
 \sum_{\rho_{\ast} \beta = \kk} \zeta^{\beta} q^{W \cdot \beta}
\rho_{\ast} \bigg(
\left[ \Mbar_{g,n}^{\sim}( \p( N \oplus \CO_D ) / \{ D_0, D_{\infty} \}, \beta ; \lambda, \mu) \right]^{\text{vir}} \\
\cdot \prod_{i=1}^{n} \ev_i^{\ast}(\gamma_i) \prod_{i=1}^{l(\lambda)} \ev_{i}^{\text{rel }0 \ast}(\delta_i)
\prod_{i=1}^{l(\mu)} \ev_i^{\text{rel }\infty \ast}(\epsilon_i) \bigg).
\end{multline*}

\subsection{Disconnected classes}
To simplify the notation we will work with disconnected classes.
The disconnected versions
of moduli spaces and the classes $\CC$ will be denoted by a '$\bullet$' resp. a dash,
following the conventions of Section~\ref{Subsection_Disconnected_GW_classes}. Since connected and disconnected invariants
may be expressed in terms of each other, Conjecture \ref{Conj_RelQuasimodularity}
is equivalent to the quasi-Jacobi form property for the disconnected theory:
\[
\CC_{g,\kk}^{\pi/D, \bullet}( \gamma_1, \ldots, \gamma_n ; \underline{\eta} ) \in
H_{\ast}( \Mbar^{\bullet}_{g,n}(B/A, \kk ; \eta) ) \otimes \frac{1}{\Delta(q)^m} \QJ_{Q_{\kk}/2}
\]
where $m = -\frac{1}{2} c_1(N_{\iota}) \cdot \mathsf{k}$.
The holomorphic anomaly equation conjectured below for disconnected relative classes (Conjecture~\ref{Conj_RELHAE})
is equivalent to a corresponding version for connected classes.

\subsection{Holomorphic anomaly equation for relative classes}
Consider the diagram
\[
\begin{tikzcd}
\Mbar_{g,n}^{\bullet}(B/A, \kk, \eta) & M_{\Delta} \ar[swap]{l}{\xi} \ar{d} \ar{r} & \Mbar_{g-1,n+2}^{\bullet}(B/A, \kk, \eta) \ar{d}{\ev_{n+1} \times \ev_{n+2}} \\
& \CB \ar{r}{\Delta_{\CB}} &  \CB \times \CB
\end{tikzcd}
\]
where $\CB$ is the stack of target degenerations of $B$ relative to $A$, the map $\Delta_{\CB}$ is the diagonal,
$M_{\Delta}$ is the fiber product and $\xi$ is the gluing map along the final two marked points.
For simplicity, we will write
\[
\CC_{g-1,\kk}^{\pi/D, \bullet}( \gamma_1, \ldots, \gamma_n, \Delta_{B/A} ; \underline{\eta} ) \\
=
\Delta_{\CB}^{!} \CC_{g-1,\kk}^{\pi/D, \bullet}( \gamma_1, \ldots, \gamma_n, \1, \1 ; \underline{\eta} ).
\]

We state the relative holomorphic anomaly equation.

\begin{conjecture} \label{Conj_RELHAE} On $\Mbar_{g,n}^{\bullet}(B/A, \kk ; \eta)$ we have
\begin{align*} & \T_q \CC_{g,\kk}^{\pi/D, \bullet}( \gamma_1, \ldots, \gamma_n ; \underline{\eta} ) \\
& 
= \iota_{\ast}
 \CC_{g-1,\kk}^{\pi/D, \bullet}( \gamma_1, \ldots, \gamma_n, \Delta_{B/A} ; \underline{\eta} ) \\
%
%
& + 2 \sum_{\substack{ \{ 1, \ldots, n \} = S_1 \sqcup S_2  \\ m \geq 0 \\ g = g_1 + g_2 + m \\ \kk_1, \kk_2 }}
\sum_{\substack{ b ; b_1, \ldots, b_m \\ \ell ; \ell_1, \ldots, \ell_m}}
\frac{\prod_{i=1}^{m} b_i}{m!}
\xi_{\ast} 
\Bigg[
\CC_{g_1,\kk_1}^{\pi/D, \bullet}\Big( \gamma_{S_1} ; \big( (b, \Delta_{A, \ell}), (b_i, \Delta_{D, \ell_i})_{i=1}^{m}\big) \Big)\\
& \hspace{11.5em} \boxtimes 
\CC_{g_2, \kk_2}^{\rho, \bullet, \mathrm{rubber}}\Big( \gamma_{S_2} ; \big( (b, \Delta_{A, \ell}^{\vee}), (b_i, \Delta^{\vee}_{D, \ell_i})_{i=1}^{m} \big), \underline{\eta} \Big) \Bigg] \\
& - 2 \sum_{i=1}^{n}
 \psi_i \cdot \CC_{g,\kk}^{\pi/D, \bullet}( \gamma_1, \ldots,\gamma_{i-1},
  \pi^{\ast} \pi_{\ast} \gamma_i , \gamma_{i+1}, \ldots \gamma_n ; \underline{\eta} ) \\
 & - 2 \sum_{i=1}^{l(\eta)} 
 \psi_{i}^{\text{rel}} \cdot \CC_{g,\kk}^{\pi/D, \bullet}\Big( \gamma_1, \ldots, \gamma_n ;
 \big( (\eta_1, \delta_1), \ldots, \underbrace{(\eta_i, \pi_D^{\ast} \pi_{D\ast} \delta_i )}_{i\text{-th}}, \ldots, (\eta_n, \delta_n)\big) \Big)
\end{align*}
with the following notation.
We let
$\psi_i, \psi_i^{\text{rel}} \in H^2( \Mbar_{g,n}(B/A, \kk ; \eta) )$
denote the cotangent line classes at the $i$-th interior and relative marked points respectively.
The first sum is over 
all $\kk_1 \in H_2(B,\BZ)$ and $\kk_2 \in H_2( \p(N' \oplus \CO_A), \BZ)$
satisfying
\[ \kk_1 \cdot A = \kk_2 \cdot A \quad \text{ and } \quad
\kk_1 + r_{\ast} \kk_2 = \kk \]
where $r : \p(N' \oplus \CO_A) \to B$ is the composition of the projection to $A$ followed by the natural inclusion into $B$. 
The $b, b_1, \ldots, b_m$ run over all positive integers such that
$b+ \sum_i b_i = \kk_1 \cdot A = \kk_2 \cdot A$,
and the $\ell, \ell_i$ run over the splitting of the diagonals of $A$ and $D$ respectively:
\[ \Delta_A = \sum_{\ell} \Delta_{A,\ell} \otimes \Delta_{A, \ell}^{\vee}, \quad \quad
 \forall i \colon\  \Delta_{D} = \sum_{\ell_i} \Delta_{D, \ell_i} \otimes \Delta_{D, \ell_i}^{\vee}.
\]
The map $\xi$ is the gluing map to $\Mbar_{g,n}^{\bullet}(B/A, \kk ; \eta)$
along the common relative marking with ramification profile $(b, b_1, \ldots, b_m)$.
Since we cup with the diagonal classes of $A$ and $D$, the gluing map $\xi$ is well-defined.
\end{conjecture}

The relative product formula of \cite{LQ} together with \cite[Thms. 2 and 3]{HAE} yields the following.
\begin{prop}
Conjectures \ref{Conj_RelQuasimodularity} and \ref{Conj_RELHAE}
hold if $X = B \times E$ and $D = A \times E$, and $\pi : X \to B$ is the projection onto the first factor.
\end{prop}

\subsection{Compatibility with the degeneration formula}
\label{Subsection_Compa_with_deg_formula}
A degeneration of $X$ compatible with the elliptic fibration $\pi : X \to B$ is a flat family
\[ \epsilon : \CX \to \Delta \]
over a disk $\Delta \subset \BC$ satisfying:
\begin{enumerate}
 \item[(i)] $\epsilon$ is a flat projective morphism, smooth away from $0$.
 \item[(ii)] $\epsilon^{-1}(1) = X$.
 \item[(iii)] $\epsilon^{-1}(0) = X_1 \cup_D X_2$ is a normal crossing divisor.
 \item[(iv)] There exists a flat morphism $\tilde{\epsilon} : \CB \to \Delta$ satisfying (i-iii) with
 $\tilde{\epsilon}^{-1}(1) = B$ and $\tilde{\epsilon}^{-1}(0) = B_1 \cup_A B_2$.
 \item[(v)] There is an elliptic fibration $\CX \to \CB$ with section and integral fiber
that restricts to elliptic fibrations with integral fibers:
\[ \pi : X \to B, \quad  \pi_i : X_i \to B_i,\, i=1,2 \quad \rho : D \to A. \]
\end{enumerate}
We further assume that the canonical map
\begin{equation} H^{\ast}(X_1 \cup_D X_2) \to H^{\ast}(X) \label{canonicalmap} \end{equation}
determined by $\epsilon$ yields an inclusion 
$\Lambda_1 \oplus \Lambda_2 \subset \Lambda$
where $\Lambda_i = H^{2}_{\perp}(X_i, \BZ)$. Let
\[ \mathbf{z}_i \in \Lambda_i \otimes \BC \]
denote the coordinate on the $i$-th summand.

Consider cohomology classes
\[ \gamma_1, \ldots, \gamma_n \in H^{\ast}(X) \]
which lift to the total space of the degeneration or equivalently\footnote{We assume the disk is sufficiently small.} which lie in the image of \eqref{canonicalmap}.
Below let $p$ always denote the forgetful morphism from various moduli spaces of stable maps to the moduli space of stable curves, for example
\[ p : \Mbar_{g,n}^{\bullet}(B/A, \kk, \eta) \to \Mbar_{g,n}^{\bullet}. \]
The application of the degeneration formula \cite{Junli1, Junli2} to $\epsilon$ yields
\begin{multline} \label{323fewdfsdF}
p_{\ast} \CC^{\pi, \bullet}_{g,\kk}(\gamma_1, \ldots, \gamma_n) \Big|_{\mathbf{z} = (\mathbf{z}_1, \mathbf{z}_2)} \\
=
\sum_{\substack{ \{ 1, \ldots, n \} = S_1 \sqcup S_2 \\ \kk_1, \kk_2 \\ m \geq 0 \\ g = g_1 + g_2 + m - 1 }}
\sum_{\substack{ \eta_1, \ldots, \eta_m \\ \ell_1, \ldots, \ell_m}}
\frac{\prod_{i} \eta_i}{m!}
p_{\ast} \xi_{\ast} \bigg[ 
\CC^{\pi_1/D, \bullet}_{g_1, \kk_1}\left(\gamma_{S_1}; \underline{\eta}        \right) \boxtimes 
\CC^{\pi_2/D, \bullet}_{g_2, \kk_2}\left(\gamma_{S_2}; \underline{\eta}^{\vee} \right) \bigg]
\end{multline}
where $\kk_1, \kk_2$ run over all possible splittings of the curve class $\kk$,
the $\eta_1, \ldots, \eta_m$ run over all positive integers such that 
\[ \sum_i \eta_i = \kk_1 \cdot A = \kk_2 \cdot A, \]
the $\ell_i$ run over the splitting of the diagonals of $D$, and we have written
\[ \underline{\eta} = (\eta_i, \Delta_{D, \ell_i})_{i=1}^{m},
\quad
\underline{\eta}^{\vee} = (\eta_i, \Delta^{\vee}_{D, \ell_i})_{i=1}^{m}.
\]
Moreover, the map $\xi$ is the gluing map along the relative point (well-defined since we inserted the diagonal).

Assume Conjectures~\ref{Conj_Quasimodularity} and~\ref{Conj_RelQuasimodularity} hold, so that \eqref{323fewdfsdF}
is an equality of quasi-Jacobi forms.
Then Conjectures~\ref{Conj_HAE} and~\ref{Conj_RELHAE} each give a way to compute the class\footnote{We will omit the restriction of $\mathbf{z}$ to the pair $(\mathbf{z}_1, \mathbf{z}_2)$ in the notation from now on.}

\begin{equation*} \frac{d}{dC_2} p_{\ast} \CC^{\pi, \bullet}_{g,\kk}(\gamma_1, \ldots, \gamma_n) \label{ERWER} \end{equation*}
as follows:
\begin{enumerate}
 \item[(a)] Apply $\T_q$ to the left-hand side of \eqref{323fewdfsdF}, use Conjecture~\ref{Conj_HAE}, and apply the degeneration formula to each term of the result.
 \item[(b)] Apply $\T_q$ to the right-hand side of \eqref{323fewdfsdF} and use Conjecture~\ref{Conj_RELHAE}.
\end{enumerate}

We say Conjectures~\ref{Conj_HAE} and~\ref{Conj_RELHAE} are \emph{compatible with the degeneration formula}
if methods (a) and (b) yield the same result.

\begin{prop} \label{Dsdfsrgrsgdf} \label{Proposition_Compatibility_with_degeneration_formula}
Assume Conjectures~\ref{Conj_Quasimodularity} and~\ref{Conj_RelQuasimodularity}.
Conjectures~\ref{Conj_HAE} and~\ref{Conj_RELHAE} are compatible with the degeneration formula.
\end{prop}
\begin{proof}
After pushforward to the moduli space of stable curves,
we apply the degeneration formula to the right-hand side of Lemma*~\ref{Lemma_discHAE}.
The result is
\begin{align*}
& \T_q p_{\ast} \CC^{\pi, \bullet}_{g,\kk}(\gamma_1, \ldots, \gamma_n) \\
& =
\sum_{\substack{ \{ 1, \ldots, n \} = S_1 \sqcup S_2 \\ \kk_1, \kk_2;\, m \geq 0 \\
\eta_1, \ldots, \eta_m, \ell_1, \ldots, \ell_m \\
g-1 = g_1 + g_2 + m - 1 }} \frac{\prod_i \eta_i}{m!}
\bigg[ p_{\ast} \xi_{\ast} \left( \CC^{\pi_1/D, \bullet}_{g_1, \kk_1}(\gamma_{S_1}, \Delta_{B_1/A} ; \underline{\eta} ) \boxtimes \CC^{\pi_2/D, \bullet}_{g_2, \kk_2}( \gamma_{S_2} ; \underline{\eta}^{\vee}) \right) \\
& \hspace{9.1em} + \ p_{\ast} \xi_{\ast} \left( \CC^{\pi_1/D, \bullet}_{g_1, \kk_1}(\gamma_{S_1} ; \underline{\eta} ) \boxtimes \CC^{\pi_2/D, \bullet}_{g_2, \kk_2}( \gamma_{S_2}, \Delta_{B_2/A} ; \underline{\eta}^{\vee}) \right) \bigg] \\
& 
-2
\sum_{\substack{ \{ 1, \ldots, n \} = S_1 \sqcup S_2 \\ \kk_1, \kk_2;\, m \geq 0 \\
\eta_1, \ldots, \eta_m, \ell_1, \ldots, \ell_m \\
g-1 = g_1 + g_2 + m - 1 }} \frac{\prod_i \eta_i}{m!} \cdot \\
& \cdot \Bigg[ \sum_{i \in S_1}
p_{\ast} \xi_{\ast} \Big( \psi_i \CC^{\pi_1/D, \bullet}_{g_1, \kk_1}(\gamma_{S_1 \setminus \{ i \}}, \pi^{\ast}\pi_{\ast}(\gamma_i); \underline{\eta} )
\boxtimes
\CC^{\pi_2/D, \bullet}_{g_2, \kk_2}( \gamma_{S_2} ; \underline{\eta}^{\vee}) \Big) \\
& + \hspace{0.2em}
\sum_{i \in S_2}
p_{\ast} \xi_{\ast} \left( \CC^{\pi_1/D, \bullet}_{g_1, \kk_1}(\gamma_{S_1}; \underline{\eta} ) \boxtimes
\psi_i \CC^{\pi_2/D, \bullet}_{g_2, \kk_2}( \gamma_{S_2 \setminus \{ i \}}, \pi^{\ast} \pi_{\ast}(\gamma_i) ; \underline{\eta}^{\vee}) \right) \Bigg]
\end{align*}
where the sums are over the same data as in \eqref{323fewdfsdF}.

We need to compare this expression with the relative holomorphic anomaly equation applied to the right-hand side of \eqref{323fewdfsdF}.
In Conjecture~\ref{Conj_RELHAE} we have four terms on the right-hand side.
The first and third term of Conjecture~\ref{Conj_RELHAE} applied to \eqref{323fewdfsdF} yield exactly the four terms above.
Hence we are left to show that the second and fourth terms of Conjecture~\ref{Conj_RELHAE} applied to \eqref{323fewdfsdF} vanish.

We consider first the second term applied to the first factor in \eqref{323fewdfsdF} plus
the fourth term applied to the second factor in \eqref{323fewdfsdF}. The result is
\begin{equation} \begin{aligned}
\label{MEG}
2 \sum
\frac{\prod_{i} c_i}{r!} \frac{\prod_i \eta_i}{m!}
& p_{\ast} \xi_{\ast} 
\left[
\CC^{\pi_1/D, \bullet}_{g_1', \kk_1'}\big(\gamma_{S_1'}; \underline{\lambda} \big) \boxtimes
\CC_{g_1'', \kk_1''}^{\rho, \bullet, \text{rub}}\big( \gamma_{S_1^{\prime \prime}} ; \underline{\lambda}^{\vee}, \underline{\eta} \big) \boxtimes
\CC^{\pi_2/D, \bullet}_{g_2, \kk_2}\big(\gamma_{S_2}; \underline{\eta}^{\vee} \big) \right] \\
- 2 \sum \sum_{i=1}^{m}
\frac{\prod_j \eta_j}{m!}
& p_{\ast} \xi_{\ast} \left[ \CC^{\pi_1/D, \bullet}_{g_1, \kk_1}\big(\gamma_{S_1}; \underline{\eta} \big) \boxtimes
\psi_i^{\text{rel}} \CC^{\pi_2/D, \bullet}_{g_2, \kk_2}\big(\gamma_{S_2}; \underline{\eta}^{\vee}\big|_{\delta_i \mapsto \pi^{\ast} \pi_{\ast} \delta_i} \big) \right],
\end{aligned} \end{equation}
where the sum in the second line is over the same data as in \eqref{323fewdfsdF},
and the sums in the first line run additionally also over the following data: splittings
of $\kk_1$ into $\kk_1', \kk_1''$, decompositions $S_1 = S_1' \sqcup S_1^{\prime \prime}$,
positive integers $c ; c_1, \ldots, c_r$, $r \geq 0$
summing up to $\kk_1' \cdot A$,
splittings $g_1 = g_1' + g_1'' + r$, and diagonal splittings $\tilde{\ell} ; \tilde{\ell}_1, \ldots, \tilde{\ell}_r$
in the weighted partitions
\[
\underline{\lambda}        = \left( (c, \Delta_{A, \tilde{\ell}}), (c_i, \Delta_{D, \tilde{\ell_i}})_{i=1}^{r}\right),
\quad 
\underline{\lambda}^{\vee} = \left( (c, \Delta_{A, \tilde{\ell}}^{\vee}), (c_i, \Delta^{\vee}_{D, \tilde{\ell_i}})_{i=1}^{r} \right).
\]
Also, we write $\underline{\eta}\big|_{\delta_i \mapsto \alpha}$ if the $i$-th cohomology class in $\underline{\eta}$ is replaced by
some $\alpha$.

We use Lemma~\ref{lemma_psi_splitting} below to remove the relative $\psi$-class in the second line of \eqref{MEG}.
When doing that, the second term on the right in Lemma~\ref{lemma_psi_splitting} (the bubble term)
precisely cancels with the expression in the first line (switch $\eta \mapsto \lambda, \mu \mapsto \eta$ and trade the sum $\sum_{i=1}^m$ for a factor of $m$).
Hence we find that \eqref{MEG} is equal to
\begin{equation}
\label{dfsdg}
2 \sum \sum_{i=1}^{l(\eta)} \frac{\prod_{j \neq i} \eta_j}{m!}
p_{\ast} \xi_{\ast} \Bigg[ \CC^{\pi_1/D, \bullet}_{g_1, \kk_1}\left(\gamma_{S_1}; \underline{\eta} \right) \boxtimes
\CC^{\pi_2/D, \bullet}_{g_2, \kk_2}\left(\gamma_{S_2}; \underline{\eta}^{\vee}\big|_{\delta_i \mapsto
\pi^{\ast} \pi_{\ast} (\delta_i) c_1(N_{A/B_2}) } \right)  \Bigg],
\end{equation}
where the first sum is over the same data as in \eqref{323fewdfsdF}.

By a parallel discussion, the second term of Conjecture~\ref{Conj_RELHAE} applied to the second factor in \eqref{323fewdfsdF}
plus the fourth term applied to the first factor is
\begin{equation} \label{term22}
2 \sum \sum_{i=1}^{l(\eta)} \frac{\prod_{j \neq i} \eta_j}{m!}
p_{\ast} \xi_{\ast}
\Bigg[ \CC^{\pi_1/D, \bullet}_{g_1, \kk_1}\left(\gamma_{S_1}; \underline{\eta}\big|_{\delta_i \mapsto
\pi^{\ast} \pi_{\ast} (\delta_i) c_1(N_{A/B_1}) } \right) \boxtimes
\CC^{\pi_2/D, \bullet}_{g_2, \kk_2}\left(\gamma_{S_2}; \underline{\eta}^{\vee} \right)  \Bigg].
\end{equation}

The term \eqref{term22} agrees exactly with \eqref{dfsdg} except for the $i$-th relative insertion.
We consider the $i$-th relative insertion more closely. 
Using
\[ (\mathrm{id} \boxtimes \pi^{\ast} \pi_{\ast})\Delta_D 
= \Delta_A 
\]
and the balancing condition
\[ N_{A/B_1} \otimes N_{A/B_2} = \CO_A \]
the $i$-th relative insertion in \eqref{dfsdg} is
\begin{align*}
\big( 1 \boxtimes c_1(N_{A/B_2})\big) \cdot (\mathrm{id} \boxtimes \pi^{\ast} \pi_{\ast})\Delta_D
& = \big( 1 \boxtimes c_1(N_{A/B_2})\big) \cdot \Delta_A \\
& = \big( c_1(N_{A/B_2}) \boxtimes 1 \big) \cdot \Delta_A \\
& = - \big( c_1(N_{A/B_1}) \boxtimes 1 \big) \cdot \Delta_A.
\end{align*}
Since this is precisely the negative of the $i$-th relative insertion in \eqref{term22}, the sum of \eqref{dfsdg} and \eqref{term22} vanishes.
\end{proof}

\begin{lemma} Let $\underline{\eta} = \{ (\eta_i, \delta_i) \}$ be a cohomology weighted partition
and let $\gamma = (\gamma_1, \ldots, \gamma_n)$ with $\gamma_i \in H^{\ast}(X)$ be a list of cohomology classes. 
We have
\label{lemma_psi_splitting}
\begin{multline*}
\eta_i \cdot p_{\ast} \left( \psi_i^{\text{rel}} 
\CC_{g,\kk}^{\pi/D, \bullet}( \gamma ; \underline{\eta} ) \right) 
=
- p_{\ast} \left(
\CC_{g,\kk}^{\pi/D, \bullet}( \gamma ; \underline{\eta}\big|_{\delta_i \mapsto \delta_i c_1(N_{A/B}) } ) 
\right) \\
+ \sum_{\substack{ \{ 1, \ldots, n \} = S_1 \sqcup S_2 \\ \kk_1, \kk_2 \\ s \geq 0 \\ g = g_1 + g_2 + s - 1 }}
\sum_{\substack{ \mu_1, \ldots, \mu_s \\ \ell_1, \ldots, \ell_s}}
\frac{\prod_{i} \mu_i}{s!}
p_{\ast} \xi_{\ast} \bigg[ 
\CC^{\rho/D, \bullet, \textup{rub}}_{g_1, \kk_1} \left(\gamma_{S_1}; \underline{\eta}, \underline{\mu}        \right) \boxtimes 
\CC^{\pi_2/D, \bullet}_{g_2, \kk_2}              \left(\gamma_{S_2}; \underline{\mu}^{\vee} \right) 
\bigg]
\end{multline*}
where the sum is over the splittings of $\kk$ into $\kk_1 \in H_2( \p(N' \oplus \CO_A), \BZ)$
and $\kk_2 \in H_2(B,\BZ)$,
all positive integers $\mu_1, \ldots \mu_s$ summing up to $\kk_1 \cdot A$,
and over indices of diagonal splittings $\ell_1, \ldots, \ell_s$ for the cohomology weighted partitions
\[
\underline{\mu} = \{ (\mu_i, \Delta_{D, \ell_i})_{i=1}^{s} \},
\quad
\underline{\mu}^{\vee} = \{ (\mu_i, \Delta^{\vee}_{D, \ell_i})_{i=1}^{s} \}.
\]
As before we write $\underline{\eta}\big|_{\delta_i \mapsto \alpha}$ if the class $\delta_i$ is replaced by
some class $\alpha$.
\end{lemma}
\begin{proof}
We will remove the class $\psi_i^{\text{rel}}$ by an argument parallel to \cite[Sec.4.5, End of Case (ii-a)]{BOPY}.
Let $\CX$ be the stack of target degenerations of the pair $(X,D)$ and let
\[ f : C \to \CX \]
be a stable map parametrized by the moduli space 
$M = M^{\bullet}_{g,n}(X / D, \beta ; \underline{\eta} )$.

Let $c : \CX \overset{c}{\to} X$ be the canonical map contracting the bubbles.
Let $p_i^{\text{rel}} \in C$ be the $i$-th relative point and let
\[ q_i = c(f(p_i^{\text{rel}})) \in D \]
be its image in $X$.
If the irreducible component of $C$ containing $p_i^{\text{rel}}$ maps into a bubble of $\CX$, then
the composition $c \circ f$ vanishes to infinite order at $p$ in the direction normal to $D$.
If the component containing $p_i^{\text{rel}}$ maps into $X$, then
by the tangency condition the composition $c \circ f$ vanishes to order exactly $\eta_i$ in the normal direction.
In either case, the differential in the normal direction induces a map
\[ N_{D/X, q_i}^{\vee} \to \Fm^{\eta_i} / \Fm^{\eta_i+1}, \]
where $\Fm$ is the maximal ideal of the point $p_i^{\text{rel}} \in C$.
See also \cite[Proof of Prop. 1.1]{OP_GWH} for a similar argument.
Considering this map in family yields a map of line bundles on $M$:
\[ \ev_i^{\text{rel} \ast} N_{D/X}^{\vee} \to \left( L_i^{\text{rel}} \right)^{\otimes \eta_i}, \]
where $L_i^{\text{rel}}$ is the cotangent line bundle on $M$.
Dualizing we obtain a section
\[ \CO_M \to \left( L_i^{\text{rel}} \right)^{\eta_i} \otimes \ev_i^{\text{rel} \ast} N_{D/X}. \]
The vanishing locus of this section is the boundary divisor of the moduli space $M$
corresponding to the first bubble of $D$ (compare \cite{BOPY}).
Expressing the class
\[ c_1\left( (L_i^{\text{rel}} )^{\eta_i} \otimes \ev_i^{\text{rel} \ast} N_{D/X} \right) = \eta_i \psi_i^{\text{rel}} + \ev_i^{\text{rel} \ast} c_1(N_{D/X}) \]
through the vanishing locus of the section and using the splitting formula, as well as the relation
\[ N_{D/X} = \pi_D^{\ast} N_{A/B}, \]
then yields the claimed formula.
\end{proof}

\section{The rational elliptic surface} \label{Section_RationalEllipticSurface}
\subsection{Definition and cohomology}
Let $R$ be a rational elliptic surface defined by a pencil of cubics.
We assume the pencil is generic, so the induced elliptic fibration
\[ R \to \p^1 \]
has $12$ rational nodal fibers.
Let $H, E_1, \ldots, E_9$
be the class of a line in $\p^2$ and the exceptional classes
of blowup $R \to \p^2$ respectively. We let
$B = E_9$ be the zero-section of the elliptic fibration, and let $F$ be the class of a fiber:
\[ B = E_9, \quad F = 3 H - \sum_{i=1}^{9} E_i . \]
We measure the degree in the fiber direction against the class
\[ W = B + \frac{1}{2} F . \]

The orthogonal complement of $B,F$ in $H^2(R, \BZ)$ is a negative-definite unimodular lattice of rank $8$ and hence is isomorphic to $E_8(-1)$,
\[ H^2(R, \BZ) = \BZ B \oplus \BZ F \oplus E_8(-1) . \]
As in Section~\ref{Section_Elliptic_fibrations_and_conjectures},
we identify the lattice $E_8(-1)$ with $\BZ^8$ by picking a basis $b_1, \ldots, b_n$.
We may assume the basis is chosen such that
\[ Q_{E_8} = \left( - \int_R b_i \cup b_j \right)_{i,j=1,\ldots, 8} \]
is the (positive definite) Cartan matrix of $E_8$.
In the notation of Section~\ref{subsubsection_pairing_and_intersection_matrix}
the matrix $Q_k$ for $k \in H_2(\p^1, \BZ) \cong \BZ$ is then
\[ Q_{k} = k Q_{E_8}. \]


\subsection{The tautological ring and a convention} \label{Subsection_CONVENTION}
If $2g-2+n>0$, let $p : \Mbar_{g,n}(\p^1,k) \to \Mbar_{g,n}$
be the forgetful map to the moduli space of stable curves, and let
\[ R^{\ast}(\Mbar_{g,n}) \subset H^{\ast}(\Mbar_{g,n}) \]
be the tautological subring spanned by push-forwards of
products of $\psi$ and $\kappa$ classes on boundary strata \cite{FP13}.

We extend both definitions to the unstable case as follows.
If $g,n \geq 0$ but $2g-2+n \leq 0$, we define
$\Mbar_{g,n}$ to be a point, $p$ to be the canonical projection, and
$R^{\ast}(\Mbar_{g,n}) = H^{\ast}(\Mbar_{g,n}) = \BQ$.

\subsection{Statement of results} \label{Subsection_RES_Statement_of_results}
The following result shows that Conjecture~\ref{Conj_Quasimodularity} holds
for rational elliptic surfaces \emph{numerically}, i.e. after integration against any tautological class
pulled back from $\Mbar_{g,n}$
(with the convention of Section~\ref{Subsection_CONVENTION} in the unstable cases).

\begin{thm} \label{theorem_RES1}
Let $\pi : R \to \p^1$ be a rational elliptic surface.
For all $g,k \geq 0$ and $\gamma_1, \ldots, \gamma_n \in H^{\ast}(R)$
and for every tautological class $\alpha \in R^{\ast}(\Mbar_{g,n})$,
\[
\int_{\Mbar_{g,n}(\p^1, k)} p^{\ast}(\alpha) \cap \CC_{g,k}^{\pi}(\gamma_1, \ldots, \gamma_n) 
\in \frac{1}{\Delta(q)^{k/2}} \QJ_{\frac{k}{2} Q_{E_8}}.
\]
\end{thm}
\vspace{7pt}

By trading descendent insertions for tautological classes
Theorem~\ref{theorem_RES1} implies that the generating series of descendent invariants
of a rational elliptic surface (for base degree $k$ and genus $g$) are quasi-Jacobi forms of index $\frac{k}{2} Q_{E_8}$.

An inspection of the proof actually yields a slightly sharper result:
the ring of quasi-Jacobi forms $\oplus_k \QJ_{\frac{k}{2} Q_{E_8}}$
in Theorem~\ref{theorem_RES1}
may be replaced by the $\QMod$-algebra generated by the theta function
$\Theta_{E_8}$ and all its derivatives.

We show that the holomorphic anomaly equation 
holds for the rational elliptic surface numerically.
Consider the right-hand side of Conjecture~\ref{Conj_HAE}:
\begin{align*}
\mathsf{H}_{g,\kk}( \gamma_1, \ldots, \gamma_n )
\ =\  & 
\iota_{\ast} \Delta^{!}
\CC^{\pi}_{g-1,\kk}( \gamma_1, \ldots, \gamma_n, \1, \1 ) \\
&
+ \sum_{\substack{ g= g_1 + g_2 \\
\{1,\ldots, n\} = S_1 \sqcup S_2 \\
\mathsf{k} = \mathsf{k}_1 + \mathsf{k}_2}}
j_{\ast} \Delta^{!} \left(
\CC^{\pi}_{g_1, \kk_1}( \gamma_{S_1}, \1) \boxtimes
\CC^{\pi}_{g_2, \mathsf{k}_2}( \gamma_{S_2}, \1 ) \right) \\
&
- 2 \sum_{i=1}^{n} 
\CC^{\pi}_{g,\kk}( \gamma_1, \ldots, \gamma_{i-1},
 \pi^{\ast} \pi_{\ast} \gamma_i , \gamma_{i+1}, \ldots, \gamma_n ) \cdot \psi_i.
\end{align*}

\begin{thm} \label{theorem_RES2}
For every tautological class $\alpha \in R^{\ast}(\Mbar_{g,n})$,
\[
\frac{d}{dC_2} \int p^{\ast}(\alpha) \cap \CC_{g,k}^{\pi}(\gamma_1, \ldots, \gamma_n) 
\ = \ 
\int p^{\ast}(\alpha) \cap \mathsf{H}_{g,\kk}( \gamma_1, \ldots, \gamma_n ).
\]
\end{thm}
\vspace{7pt}

In the remainder of Section~\ref{Section_RationalEllipticSurface}
we present the proof of Theorems~\ref{theorem_RES1} and~\ref{theorem_RES2}.
In Section~\ref{Subsection_RES_Sections} we recall a few basic
results on the group of sections of a rational elliptic surface.
This leads to the genus $0$ case of Theorem~\ref{theorem_RES1} in Section~\ref{Subsection_Genus_Zero}.
In Section~\ref{Subsection_RES_relative_in_absolute} we discuss
the invariants of $R$ relative to a non-singular elliptic fiber of $\pi$.
In the last two sections we present the proofs of the general cases of Theorems~\ref{theorem_RES1} and~\ref{theorem_RES2}.

\subsection{Sections} \label{Subsection_RES_Sections}
Recall from \cite{Shioda}
the 1-to-1 correspondence between sections of $R \to \p^1$ and elements in the lattice $E_8(-1)$.
A section $s$ yields an element in $E_8(-1)$ by projecting its class $[s]$ onto the $E_8(-1)$ lattice.
Conversely, an element $\lambda \in E_8(-1) \subset H^2(R, \BZ)$
has a unique lift
$\hat{\lambda} \in H^2(R, \BZ)$ such that $\hat{\lambda}^2 = -1$, $\hat{\lambda} \cdot F = 1$ and $\hat{\lambda}$ pairs
positively with any ample class. By Grothendieck-Riemann-Roch $\hat{\lambda}$
is the cohomology class of a unique section $B_{\lambda}$. Explicitly,
\[ [B_{\lambda}] = W - \left( \frac{\langle \lambda, \lambda \rangle + 1}{2} \right) F + \lambda \]
where $\langle a, b \rangle = \int_R a \cup b$ for all $a,b \in H^{\ast}(R)$ is the intersection pairing.

By fiberwise addition and multiplication by $-1$
the set of sections of $R \to \p^1$ form a group, the Mordell-Weil group.
The correspondence between sections and classes in $E_8(-1)$ is a group homomorphism,
\[ B_{\lambda + \mu} = B_{\lambda} \oplus B_{\mu}, \quad B_{-\lambda} = \ominus B_{\lambda} \]
where we have written $\oplus, \ominus$ for the addition resp. subtraction on the elliptic fibers.
%
The translation by a section $\lambda \in E_8(-1)$,
\[ t_{\lambda} : R \to R, \ x \mapsto x + B_{\lambda}(\pi(x)), \]
acts on a cohomology class $\gamma \in H^{\ast}(X)$ by
\[ t_{\lambda \ast} \gamma = \gamma + \lambda \cup \pi^{\ast} \pi_{\ast}( \gamma) - \pi^{\ast} \pi_{\ast}( \lambda \cup \gamma )
- \frac{1}{2} \pi^{\ast}\left( \pi_{\ast}(\lambda^2) \cdot \pi_{\ast}(\gamma) \right). \]

%

\subsection{Genus zero} \label{Subsection_Genus_Zero}
\subsubsection{Overview}
Consider the genus $0$ stationary invariants
\begin{align*}
M_k(\zeta,q) & = \int \CC^\pi_{0,k}(\pt^{\times k-1}) \\
& = \sum_{\pi_{\ast} \beta = k} q^{W \cdot \beta} \zeta^{\beta}
\int_{ [ \Mbar_{0,k-1}(R, \beta) ]^{\text{vir}}} \prod_{i=1}^{k-1} \ev_i^{\ast}( \pt )
\end{align*}
for all $k \geq 1$, where $\pt \in H^4(R, \BZ)$ is the class Poincar\'e dual to a point.

\begin{prop} \label{Prop_Mkquasi} $M_k \in \frac{1}{\Delta(q)^{k/2}} \QJac_{8k-4, \frac{k}{2} Q_{E_8}}$ for all $k \geq 1$. \end{prop}

In the remainder of Section~\ref{Subsection_Genus_Zero} we prove Proposition~\ref{Prop_Mkquasi}. 

\subsubsection{The $E_8$ theta function} \label{Subsubsection_case_k=1}
All curve classes on $R$ of degree $1$ over $\p^1$ are of the form $B_{\lambda} + d F$ for some section $\lambda \in E_{8}(-1)$ and $d \geq 0$.
Using Section~\ref{Subsection_RES_Sections} and \cite[Sec.6]{BL} we find
\begin{align*}
M_1(q)
& = \sum_{\lambda \in E_{8}(-1)} \sum_{d \geq 0} q^{W \cdot (B_{\lambda} + dF)} \zeta^{\lambda} \int_{[ \Mbar_{0,0}(R, B_{\lambda} + dF) ]^{\text{vir}}} 1 \\
& = \sum_{\lambda \in E_{8}(-1)} \sum_{d \geq 0} q^{d - \frac{1}{2} - \frac{1}{2} \langle \lambda, \lambda \rangle} \zeta^{\lambda} \left[ \frac{1}{\Delta(q)^{1/2}} \right]_{q^{d-\frac{1}{2}}} \\
& = \frac{1}{\Delta(q)^{\frac{1}{2}}} \sum_{\lambda \in E_8(-1)} q^{-\frac{1}{2} \langle \lambda, \lambda \rangle} \zeta^{\lambda} \\
& = \frac{1}{\Delta(q)^{\frac{1}{2}}} \Theta_{E_8}(z, \tau).
\end{align*}
By Section~\ref{Section_E8_lattice}, $\Theta_{E_8}$ is a Jacobi form
of index $\frac{1}{2} Q_{E_8}$ and weight $4$.

\subsubsection{WDVV equation}
For any $\gamma_1, \ldots, \gamma_n \in H^{\ast}(R)$ define the
quantum bracket
\[
\blangle \gamma_1, \ldots, \gamma_n \brangle_{0,k} =
\sum_{\pi_{\ast} \beta = k} q^{W \cdot \beta} \zeta^{\beta} \int_{[ \Mbar_{0,n}(R, \beta) ]^{\text{vir}}} \prod_i \ev_i^{\ast}(\gamma_i).
\]

Recall the WDVV equation from \cite{FulPand}: For all $\gamma_1, \ldots, \gamma_n \in H^{\ast}(R)$ with
\[ \sum_{i=1}^{n} \deg(\gamma_i) = n+k-2 \]
we have
\begin{align*}
& \sum_{\substack{ k=k_1 + k_2 \\ \{ 1, \ldots, n-4 \} = S_1 \sqcup S_2 }}
\sum_{\ell}
\blangle \gamma_{S_1}, \gamma_{a}, \gamma_{b}, \Delta_\ell \brangle_{0,k_1} \blangle \gamma_{S_2}, \gamma_c, \gamma_d, \Delta_\ell^{\vee} \brangle_{0,k_2} \\
= & 
\sum_{\substack{ k=k_1 + k_2 \\ \{ 1, \ldots, n-4 \} = S_1 \sqcup S_2 }} \sum_{\ell}
\blangle \gamma_{S_1}, \gamma_{a}, \gamma_{c}, \Delta_\ell \brangle_{0,k_1} \blangle \gamma_{S_2}, \gamma_b, \gamma_d, \Delta_\ell^{\vee} \brangle_{0,k_2},
\end{align*}
where $\sum_{\ell} \Delta_{\ell} \otimes \Delta_{\ell}^{\vee}$ is the K\"unneth decomposition
of the diagonal class $\Delta \in H^{\ast}(R \times R)$.
Let also
\[ D = D_q, \quad D_i = D_{b_i} = D_{\zeta_i} = \frac{1}{2\pi i} \frac{d}{dz_i}. \]
We solve for the remaining series $M_k$ by applying the WDVV equation.

\subsubsection{Proof of Proposition~\ref{Prop_Mkquasi}}
The case $k=1$ holds by Section~\ref{Subsubsection_case_k=1}.
For $k = 2$ recall the basis $\{ b_i \}$ of $\Lambda$
and apply the WDVV equation for $(\gamma_i)_{\ell=1}^{4}= (F,F, b_i, b_j)$. The result is
\[
 4 \langle b_i, b_j \rangle M_2 = D_i \langle \Delta_1 \rangle_{0,1} \cdot D_j \langle \Delta_2 \rangle_{0,1}
 - \langle \Delta_1 \rangle_{0,1} \cdot D_i D_j \langle \Delta_2 \rangle_{0,1}
\]
where $\Delta_1, \Delta_2$ indicates that we sum over the diagonal splitting.
Choosing $i,j$ such that $\langle b_i, b_j \rangle \neq 0$
and applying the divisor equation on the right-hand side we find
$M_2$ expressed as a sum of products of derivatives of $M_1$. Checking the weight and index yields the claim for $M_2$.

Similarly, the WDVV equation for $(\gamma_i)_{i=1}^{4} = (\pt, F, F, W)$ yields
\[
3 M_3 = M_1 \cdot D^2 M_2 - 4 D^2 M_1 \cdot M_2 +
\sum_{i=1}^{8} \left( D_i D M_1 \cdot 2 D_i M_2 - D_i M_1 \cdot D_i D M_2 \right)
\]
which completes the case $k=3$.

If $k \geq 4$ we apply the WDVV equation for $(\gamma_1, \ldots, \gamma_k) = (\pt^{k-2}, \ell_1, \ell_2)$
for some $\ell_1, \ell_2 \in H^2(R)$. The result is
\begin{multline*}
( \ell_1 \cdot \ell_2 )
\blangle \pt^{k-1} \brangle_{0,k}
=
\sum_{a+b=k-4} \binom{k-4}{a} \Big( \blangle \pt^{a+1}, \ell_1, \Delta_1 \brangle_{0,a+2} \blangle \pt^{b+1}, \ell_2, \Delta_2 \brangle_{0,b+2} \\
- \blangle \pt^{a+2, \Delta_1} \brangle_{0,a+3} \blangle \pt^b, \ell_1, \ell_2, \Delta_2 \brangle_{0,b+1} \Big).
\end{multline*}
Taking $\ell_1 \cdot \ell_2 = 1$ and using an induction argument the proof is complete. \qed 

\subsection{Relative in terms of absolute} \label{Subsection_RES_relative_in_absolute}
Let $\leq$ be the lexicographic order on the set of pairs $(k,g)$, i.e.
\begin{equation} \label{lexigraph}
(k, g) \leq (k', g') \quad \Longleftrightarrow \quad k < k' \text{ or } \big( k=k' \text{ and } g \leq g' \big).
\end{equation}

Let $E \subset R$ be a non-singular fiber of $\pi : R \to \p^1$
over the point $0 \in \p^1$,
and recall from Section~\ref{Section_relative_geomtries} the $E$-relative
Gromov--Witten classes
\[ \CC_{g,k}^{\pi/E}(\gamma_1, \ldots, \gamma_n; \underline{\eta})
 \in H_{\ast}( \Mbar_{g,n}(\p^1/0,k ; \eta) ) \otimes \BQ[[ q^{\pm \frac{1}{2}}, \zeta^{\pm 1} ]]
\]
where $\underline{\eta}$ is the ordered cohomology weighted partition 
\begin{equation} \label{Eetetaerapart}
\underline{\eta} = \big( (\eta_1, \delta_1), \ldots, (\eta_{l(\eta)}, \delta_{l(\eta)} ) \big), \quad \delta_i \in H^{\ast}(E).
\end{equation}

We show the (numerical) quasi-Jacobi form property and holomorphic anomaly equation
in the absolute case imply the corresponding relative case.
For the statement and the proof we use the convention of Section~\ref{Subsection_CONVENTION}.

\begin{prop} \label{DEGENERATIONPROP} Let $K,G \geq 0$ be fixed. Assume
\[
\int_{\Mbar_{g,n}(\p^1, k)} p^{\ast}(\alpha) \cap \CC_{g,k}^{\pi}(\gamma_1, \ldots, \gamma_n) 
\in \frac{1}{\Delta(q)^{k/2}} \QJ_{\frac{k}{2} Q_{E_8}}
\]
for all $(k,g) \leq (K,G)$, $n \geq 0$, $\alpha \in R^{\ast}(\Mbar_{g,n})$
and $\gamma_1, \ldots, \gamma_n \in H^{\ast}(R)$.
Then
\[
\int_{\Mbar_{g,n}(\p^1/0, k ; \underline{\eta})} p^{\ast}(\alpha) \cap \CC_{g,k}^{\pi/E}(\gamma_1, \ldots, \gamma_n; \underline{\eta}) 
\in \frac{1}{\Delta(q)^{k/2}} \QJ_{\frac{k}{2} Q_{E_8}}
\]
for all $(k,g) \leq (K,G)$, $n \geq 0$, $\alpha \in R^{\ast}(\Mbar_{g,n})$,
$\gamma_1, \ldots, \gamma_n \in H^{\ast}(R)$ and cohomology weighted partitions $\underline{\eta}$.

Similarly, if the holomorphic anomaly equation holds numerically for all $\CC_{g,k}^{\pi}(\gamma_1, \ldots, \gamma_n)$
with $(k,g) \leq (K,G)$, then the relative holomorphic anomaly equation of Conjecture~\ref{Conj_RELHAE}
holds numerically for all $\CC_{g,k}^{\pi/E}(\gamma_1, \ldots, \gamma_n ; \underline{\eta})$ with $(k,g) \leq (K,G)$.
\end{prop}
\begin{proof}
The degeneration formula applied to the normal cone degeneration
\begin{equation} R \rightsquigarrow R \cup_E ( \p^1 \times E ) \label{3sdfsdfd} \end{equation}
expresses the absolute invariants of $R$ in terms of the relative invariants of $R/E$ and $(\p^1 \times E)/E_0$.
The quasi-modularity of the invariants of $(\p^1 \times E)/E_0$ relative to $\p^1$
follows from the product formula \cite{LQ} and \cite[Thm.2]{HAE}.
We may hence view the degeneration formula as a matrix between the absolute and relative (numerical) invariants of $R$
with coefficients that are quasi-modular forms.
By \cite[Thm.2]{MP} it is known that the matrix is non-singular: The absolute invariants determine the relative invariants of $R$.
We only need to check that the absolute terms with $(k,g) \leq (K,G)$ determine
the relative ones of the same constraint,
and that the quasi-Jacobi form property is preserved by this operation.
Since $\QJac_{\frac{k}{2} Q_{E_8}}$ is a module over $\QMod$,
the second statement is immediate from the induction argument used to prove the first.
The first follows from scrutinizing the algorithm in \cite[Sec.2]{MP} and we only sketch the argument here.

Given $(k,g) \leq (K,G)$, a cohomological weighted partition $\underline{\eta}$ as in \eqref{Eetetaerapart},
insertions $\gamma_1, \ldots, \gamma_n \in H^{\ast}(R)$, and a tautological class $\alpha \in R^{\ast}(\Mbar_{g,n})$,
consider the absolute invariant
\begin{multline} \label{DFSDFas}
\Blangle \alpha\, ;\, \prod_{i=1}^{n} \tau_0(\gamma_i) \prod_{i=1}^{l(\eta)} \tau_{\eta_i-1}(j_{\ast}\delta_i) \Brangle^{R}_{g,k} \\
= \sum_{\pi_{\ast} \beta = \kk} \zeta^{\beta} q^{W \cdot \beta}
\int_{[  \Mbar_{g,n+l(\eta)}(X, \beta) ]^{\text{vir}}} p^{\ast}(\alpha) \prod_{i=1}^{n} \ev_i^{\ast}(\gamma_i)
\prod_{i=1}^{l(\eta)} \psi_i^{\eta_i-1} \ev_i^{\ast}(j_{\ast} \delta_i)
\end{multline}
where we used the Gromov--Witten bracket notation of \cite{MP}, $j : E \to R$ is the inclusion,
and $\psi_i$ are the cotangent line classes on the moduli space of stable maps to $R$.
By trading the $\psi_i$ classes for tautological classes (modulo lower order terms)
and using the assumption on absolute invariants, we see that
the series \eqref{DFSDFas} is a quasi-Jacobi form of index $\frac{k}{2} Q_{E_8}$.
We apply the degeneration formula with respect to \eqref{3sdfsdfd} to the invariant \eqref{DFSDFas}.
The cohomology classes are lifted to the total space of the degeneration
as in \cite[Sec.2]{MP}, i.e. the $\gamma_i$ are lifted by pullback and the
$j_{\ast} \delta_i$ are lifted by inclusion of the proper transform of $E \times \BC$.
Using a bracket notation for relative invariants parallel to the above\footnote{The bracket notation is
explained in more detail in \cite{MP} with the difference that the ramification profiles
$\underline{\nu}$ are \emph{ordered} here. This yields slightly different factors
in the degeneration formula than in \cite{MP} but is otherwise not important.},
the degeneration formula yields
\begin{multline} \label{sdfsdf}
\Blangle \alpha ; \prod_i \tau_0(\gamma_i) \cdot \prod_{i=1}^{l(\eta)} \tau_{\eta_i-1}(j_{\ast}\delta_i) \Brangle^{R}_{g,k} = \\
\sum_{\substack{ m \geq 0 \\ \nu_1, \ldots, \nu_m, \ell_1, \ldots, \ell_m \\
g = g_1 + g_2 + m-1 \\
\{ 1, \ldots, n \} = S_1 \sqcup S_2 \\
\alpha_1, \alpha_2}}
\frac{\prod_{i} \nu_i}{m!}
\Blangle \alpha_1 ; \tau_0(\gamma_{S_1}) \Big| \underline{\nu} \Brangle^{R/E, \bullet}_{g_1,k}
\Blangle \alpha_2 ; \tau_0(\gamma_{S_2}) \prod_{i=1}^{l(\eta)} \tau_{\eta_i-1}(j_{\ast}\delta_i) \Big| \underline{\nu}^{\vee}
\Brangle^{(\p^1 \times E)/E, \bullet}_{g_2,k}
\end{multline}
where $\nu_1, \ldots, \nu_m$ run over all positive integers with sum $k$,
$\ell_1, \ldots, \ell_m$ run over all diagonal splittings in the cohomology weighted partitions
\[
\underline{\nu} =
( \nu_i, \Delta_{E,\ell_i} )_{i=1}^{m}, \quad 
\underline{\nu}^{\vee}
= ( \nu_i, \Delta^{\vee}_{E,\ell_i} )_{i=1}^{m},
\]
and $\alpha_1, \alpha_2$ run over all splittings of the tautological class $\alpha$.
The sum is taken only over those configurations of disconnected curves which yield a connected domain after gluing.

We argue now by an induction over the relative invariants of $R/E$ with respect to the lexicographic ordering on $(k,g,n)$.
If the invariants of $R/E$ in \eqref{sdfsdf} (the first factor on the right)
are disconnected, each connected component is of lower degree over $\p^1$,
and therefore these contributions are determined by lower order terms.
Hence we may assume that the invariants of $R/E$ are connected.
By induction over the genus we may further assume $g_1 = g$ in \eqref{sdfsdf},
or equivalently $g_2 = 1 - m$.
Consider a stable relative map in the corresponding moduli space and let 
\[ f : C_2 \to (\p^1 \times E)[a] \]
be the component which maps to an expanded pair of $(\p^1 \times E, E_0)$.
Since $g_2 = 1-m$ the curve $C_2$
has at least
$m$ connected components of genus $0$. Since each of these meets the relative divisor
and $l(\nu) = m$, the curve $C_2$ is a disjoint union of genus $0$ curves.
The rational curves in $\p^1 \times E$ are fibers of the projection to $E$.
Hence we find the right-hand side in \eqref{sdfsdf} is a fiber class integral (in the language of \cite{MP}).
Finally, by induction over $n$ we may assume $S_2 = \varnothing$.
As in \cite[Sec.2.3]{MP} we make a further induction over $\deg(\underline{\eta}) = \sum_i \deg(\delta_i)$
and a lexicographic ordering of the partition parts $\underline{\eta}$.
Arguing as in \cite[Sec.1, Relation 1]{MP}\footnote{
Using the dimension constraint the class $\alpha_2$ only increases the parts $\nu_k$, and hence by induction
we may assume $\alpha_2=1$.} we finally arrive at
\[
 \Blangle \alpha ; \prod_i \tau_0(\gamma_i) \cdot \prod_{i=1}^{l(\eta)} \tau_{\eta_i-1}(j_{\ast}\delta_i) \Brangle^{R}_{g,k} \\
=
c \cdot \Blangle \alpha ; \prod_i \tau_0(\gamma_i) | \underline{\nu} \Brangle^{R/E}_{g,k}
+ \ldots
\]
where $c \in \BQ$ is non-zero and '$\ldots$' is a sum
of a product of quasi-modular forms and relative invariants of $R/E$ of lower order.
By induction the lower order terms are quasi-Jacobi forms of index $\frac{1}{2} k Q_{E_8}$
which completes the proof of the quasi-Jacobi property of the invariants of $R/E$.

The relative holomorphic anomaly equation follows immediately from this algorithm and
the compatibility with the degeneration formula (Proposition~\ref{Proposition_Compatibility_with_degeneration_formula}).
\end{proof}

\subsection{Proof of Theorem~\ref{theorem_RES1}}
Assume that the classes $\gamma_1, \ldots, \gamma_n \in H^{\ast}(S)$ and $\alpha \in R^{\ast}(\Mbar_{g,n})$ are homogenenous.
We consider the dimension constraint
\begin{equation} \label{dimensionconstraint}
k + g - 1 + n = \deg(\alpha) + \sum_{i=1}^{n} \deg(\gamma_i)
\end{equation}
where $\deg( )$ denotes half the real cohomological degree.
The left-hand side in \eqref{dimensionconstraint} is the virtual dimension of $\Mbar_{g,n}(S, \beta)$ where $\pi_{\ast} \beta = k$.
If the dimension constraint is violated, the
left-hand side in Theorem~\ref{theorem_RES1} vanishes and the claim holds.
Hence we may assume \eqref{dimensionconstraint}.

We argue by induction on $(k,g,n)$ with respect to the lexicographic ordering
\begin{align*}
(k_1, g_1, n_1) < (k_2, g_2, n_2) \quad \Longleftrightarrow \quad
\quad & k_1 < k_2 \\
\text{ or } & \big( k_1=k_2 \text{ and } g_1 < g_2 \big) \\
\text{ or } & \big( k_1=k_2 \text{ and } g_1 = g_2 \text{ and } n_1 < n_2 \big)
\end{align*}

\vspace{5pt}
\noindent \textbf{Case (i):} $g=0$.

\begin{enumerate}
\item[(i-a)] If $k=0$ all invariants vanish, so we may assume $k>0$.

\item[(i-b)] If $\deg(\alpha) > 0$ then $\alpha$ is the pushforward of a cohomology class from the boundary $\iota : \partial \Mbar_{0,n} \to \Mbar_{0,n}$:
\[ \alpha = \iota_{\ast} \alpha'. \]
Using $\alpha'$ and the compatibility of the virtual class with boundary restrictions
we can replace the left-hand side of Theorem~\ref{theorem_RES1} by terms of lower order (see \cite[Sec.3]{HAE} for a parallel argument).

\item[(i-c)] If $\deg(\alpha) = 0$ but $\deg(\gamma_i) \leq 1$ for some $i$, then either
the series is zero (if $\deg(\gamma_i) = 0$) and the claim holds,
or we can apply the divisor equation to reduce to lower order terms.
Since derivatives of quasi-Jacobi forms are quasi-Jacobi forms of the same index the claim follows from the induction hypothesis.

\item[(i-d)] If $\deg(\alpha) = 0$ and $\gamma_i = \pt$ for all $i$ the claim follows by Proposition~\ref{Prop_Mkquasi}.
\end{enumerate}

\vspace{5pt}
\noindent \textbf{Case (ii):} $g>0$ and $\deg(\alpha) \geq g$.
\vspace{5pt}

By \cite[Prop.2]{FPrel} we have
\[ \alpha = \iota_{\ast} \alpha' \]
for some $\alpha'$ where $\iota : \partial \Mbar_{g,n} \to \Mbar_{g,n}$ is the inclusion of the boundary.
By restriction to the boundary we are reduced to lower order terms.

\vspace{5pt}
\noindent \textbf{Case (iii):} $g>0$ and $\deg(\alpha) < g$.
\vspace{5pt}

By the dimension constraint we have
\[ \sum_{i=1}^{n} \deg(\gamma_i) - n \geq k. \]
Hence after reordering we may assume $\gamma_1 = \ldots = \gamma_k = \pt$.
Consider the degeneration of $R$ to the normal cone of a non-singular fiber $E$,
\[ R \rightsquigarrow R \cup_E (\p^1 \times E). \]
We let $\rho : \p^1 \times E \to \p^1$
be the projection to the first factor and let $E_0$ denote the
fiber of $\rho$ over $0 \in \p^1$.
We apply the degeneration formula \cite{Junli1, Junli2} where
we specialize the insertions $\gamma_1, \ldots, \gamma_k$ to the component $\p^1 \times E$
and lift the other insertions by pullback.
In the notation of Section~\ref{Section_relative_geomtries} the result is
\begin{multline} \label{DEGGGG}
p_{\ast} \CC^{\pi}_{g,\kk}(\gamma_1, \ldots, \gamma_n) \\
=
\sum_{ \substack{ 
m \geq 0 \\
\eta_1, \ldots, \eta_m, \ell_1, \ldots, \ell_m \\
\{ k+1, \ldots, n \} = S_1 \sqcup S_2 \\ g = g_1 + g_2 + m - 1 }}
\frac{\prod_i \eta_i}{m!}
p_{\ast} \xi^{\mathrm{conn}}_{\ast} \left( \CC^{\pi/E, \bullet}_{g_1, \kk}(\gamma_{S_1}; \underline{\eta} ) \boxtimes \CC^{\rho/E_0, \bullet}_{g_2, \kk}( \pt^k, \gamma_{S_2} ; \underline{\eta}^{\vee}) \right)
\end{multline}
where $\eta_1, \ldots, \eta_m$ run over all positive integers summing up to $k$,
$\ell_1, \ldots, \ell_m$ run over all diagonal splittings in the partitions
\[
\underline{\eta} =
( \eta_i, \Delta_{E,\ell_i} )_{i=1}^{m}, \quad 
\underline{\eta}^{\vee}
= ( \eta_i, \Delta^{\vee}_{E,\ell_i} )_{i=1}^{m},
\]
the map $\xi$ is the gluing map along the relative markings,
and $\xi^{\mathrm{conn}}_{\ast}$ is pushforward by $\xi$ followed by
taking the summands with connected domain curve.

We will show that the right-hand side of \eqref{DEGGGG}, when integrated against any tautological class,
is a quasi-Jacobi form of index $\frac{k}{2} Q_{E_8}$.

By the product formula \cite{LQ} and \cite[Thm.2]{HAE}, each term 
\[ \CC^{\rho/E_0, \bullet}_{g_2, \kk}( \pt^k, \gamma_{S_2} ; \underline{\eta}^{\vee}) \]
is a cycle-valued quasi-modular form.
We consider the first factor
\begin{equation} \CC^{\pi/E, \bullet}_{g_1, \kk}(\gamma_{S_1}; \underline{\eta} ) \label{resdfsd} \end{equation}
after integration against any tautological class. We make two reduction steps:

\vspace{5pt} \noindent (1) We may assume \eqref{resdfsd} are connected Gromov--Witten invariants.

(Proof: The difference between connected and disconnected invariance is a sum of products of connected invariants of $R/E$
of degree lower than $k$ over the base. Hence
by Proposition~\ref{DEGENERATIONPROP} and the induction hypothesis they are quasi-Jacobi forms
after integration against tautological classes.)

\vspace{5pt} \noindent (2) We may assume $g_1 = g$.

(Proof: If $g_1<g$ 
the series \eqref{resdfsd} is a quasi-Jacobi form after integration
by Proposition~\ref{DEGENERATIONPROP} and induction.
)

By the above steps it remains to consider the terms of \eqref{resdfsd} which are connected and of genus $g$.
We will show that the term
\[ p_{\ast} \xi_{\ast}
\left( \CC^{\pi/E}_{g, \kk}(\gamma_{S_1}; \underline{\eta} )
\boxtimes \CC^{\rho/E_0, \bullet}_{g_2, \kk}( \pt^k, \gamma_{S_2} ; \underline{\eta}^{\vee}) \right) \]
is zero after integration against any tautological class.
Consider a stable relative map in the corresponding moduli space and let 
\[ f : C_2 \to (\p^1 \times E)[a] \]
be the component which maps to an expanded pair of $(\p^1 \times E, E_0)$.
Since $g=g_1 + g_2 + m - 1$ we have $g_2 = 1-m$,
hence $C_2$ has at least $m$ connected components of genus $0$. 
Since each such component meets the relative divisor $E$
and moreover $l(\eta) = m$, the domain curve of the stable map to $\p^1 \times E$ is a disjoint union of $m$ rational curves.
Since rational curves are of degree $0$ over the $E$-factor
and the stable map to $\p^1 \times E$ is incident to $k$ given point insertions,
the Gromov--Witten invariant is zero unless $m=k$ and $\underline{\eta} = (1, \omega)^k$ where $\omega \in H^2(E)$ is the point class.
Case (iii) then follows from Lemma~\ref{Lemma_vanishing} below.  \qed

\begin{lemma} \label{Lemma_vanishing} For all $k \geq 0$ and $\gamma_1, \ldots, \gamma_n \in H^{\ast}(R)$ we have
\[ \CC^{\pi/E}_{g, \kk}\left(\gamma_1, \ldots, \gamma_n; (1, \omega)^k \right) = 0 \]
where $\omega \in H^2(E)$ is the class of a point.
\end{lemma}
\begin{proof}
First we consider the case $k>0$. Let $\beta \in H_2(R,\BZ)$ be a curve class with $\pi_{\ast} \beta = k$.
Let $L \in \Pic(R)$ be the line bundle with $c_1(L) = \beta$.
Consider a relative stable map to an extended relative pair of $(R,E)$ in class $\beta$,
\[ f : C \to R[a]. \]
Since $R$ is rational, the universal family of curves on $R$ in a given class is a linear system.
Hence the intersection of $f(C)$ with the distinguished relative divisor $E \subset R[n]$ satisfies
\[ \CO_E( f(C) \cap E ) = L|_{E}. \]
Let $x_1, \ldots, x_k \in E$ be fixed points with $\CO_E(x_1 + \ldots + x_k) \neq L|_{E}$.
It follows that no stable relative map in class $\beta$ is incident to $(x_1, \ldots, x_k)$ at the relative divisor.
We conclude 
\[ \left[ \Mbar_{g,n}(R/E, \beta ; (1)^k) \right]^{\text{vir}} \prod_{i=1} \ev_i^{\text{rel}\ast}( [x_i] )\, = \, 0 \]
which implies the claim.

It remains to consider the case $k=0$. We have the equality of moduli spaces
\[
\Mbar_{g,n}(R/E, dF ; ()) = \Mbar_{g,n}(R, dF). 
\]
Under this identification the obstruction sheaf of stable maps to $R$ relative to $E$
for a fixed source curve $C$ is
\[ \mathrm{Ob}_{C,f} = H^1( C, f^{\ast} T_{R/E} ) \]
where $T_{R/E} = \Omega_R(\log E)^{\vee}$
is the log tangent bundle relative to $E$.
Since $K_R + E = 0$ there exists a meromorphic $2$-form 
\[ \sigma \in H^0(R, \Omega_R^2(E)) \]
with a simple pole along $E$ and nowhere vanishing outside $E$.
By the construction \cite[Sec.4.1.1]{STV}
the form $\sigma$ yields a surjection
\[ \mathrm{Ob}_{C,f} \to \BC \]
which in turn induces a
nowhere-vanishing cosection of the perfect obstruction theory on the moduli space.
By \cite{KL} we conclude
\[ [ \Mbar_{g,n}(R/E, dF ; ()) ]^{\text{vir}} = 0, \]
which implies the claim.
\end{proof}

\subsection{Proof of Theorem~\ref{theorem_RES2}}
The holomorphic anomaly equation is implied by the following compatibilitites
which cover all all steps in the algorithm used in the proof of Theorem~\ref{theorem_RES1}:
\begin{itemize}
\item The compatibility with boundary restrictions (parallel to \cite[Sec.2.5]{HAE}).
\item The compatibility with the degeneration formula (Proposition~\ref{DEGENERATIONPROP}).
\item The compatibility with the WDVV equation (special case of (i)).
\item The compatibility with the divisor equation (follows by proving a refined weight statement parallel to \cite[Sec.3]{HAE}).
\item The holomorphic anomaly equation holds for $\int \CC_{0,1}() = \Theta_{E_8} \Delta^{-1/2}$. \qed
\end{itemize}

\section{The Schoen Calabi--Yau threefold}
\label{Section_Schoen_variety}

\subsection{Preliminaries}
Let $X = R_1 \times_{\p^1} R_2$ be a Schoen Calabi--Yau
and recall the notation from Section~\ref{intro:schoen_calabiyau}.
In particular we have the commutative diagram of fibrations
\begin{equation}
\label{Diag2}
\begin{tikzcd}
& X \ar[swap]{dl}{\pi_2} \ar{dr}{\pi_1} \ar{dd}{\pi} & \\
R_1 \ar[swap]{dr}{p_1} & & R_2 \ar{dl}{p_2} \\
& \p^1 &
\end{tikzcd}
\end{equation}

Let $\alpha \in H_2(R_1, \BZ)$ be a curve class. For all $(g, \alpha) \notin \{ (0,0), (1,0) \}$ define
\[
\F_{g,\alpha}(\mathbf{z}_2, q_2) = \int \CC^{\pi_2}_{g,\alpha}() 
= 
\sum_{\pi_{2 \ast} \beta = \alpha}
q_2^{W_2 \cdot \beta} \zeta_2^{\beta} \int_{\Mbar_{g}(X,\beta) ]^{\text{vir}}} 1.
\]
For all $(g, k) \notin \{ (0,0), (1,0) \}$ we have
\begin{equation} \label{4tsfgsdf}
\F_{g,k}(\mathbf{z}_1, \mathbf{z}_2, q_1, q_2)
=
\sum_{\substack{ \alpha \in H_2(R_1, \BZ) \\ p_{1 \ast} \alpha = k}} \F_{g,\alpha}(\mathbf{z}_2, q_2) q_1^{W_1 \cdot \alpha} e( \mathbf{z}_1 \cdot \alpha).
\end{equation}

We first prove a weaker version of Theorem~\ref{Thm1}.
\begin{prop} We have
\[
\F_{g,k} \in
\frac{1}{\Delta(q_1)^{k/2}} \QJac_{\frac{k}{2} Q_{E_8}}^{(q_1, \mathbf{z}_1)}
\otimes
\frac{1}{\Delta(q_2)^{k/2}} \QJac_{\frac{k}{2} Q_{E_8}}^{(q_2, \mathbf{z}_2)}.
\]
\end{prop}

\begin{proof}
The Schoen Calabi--Yau can be written as a complete intersection
\[ X \subset \p^1 \times \p^2 \times \p^2 \]
cut out by sections of tri-degree $(1,3,0)$ and $(1,0,3)$. Hence
there exist smooth elliptic fibers $E_i \subset R_i$ of $\pi_i$ for $i=1,2$ and a degeneration
\begin{equation} X \rightsquigarrow (R_1 \times E_2) \cup_{E_1 \times E_2} (E_1 \times R_2) \label{DEGGGG} \end{equation}
which is compatible with the fibration structure of diagram~\eqref{Diag2}.

The degeneration formula applied with respect to this degeneration yields
\begin{equation} \label{dfsdfds}
\F_{g,k} = 
\sum_{\substack{m \geq 0 \\ \eta_1, \ldots, \eta_m, \ell_1, \ldots, \ell_m \\ g = g_1 + g_2 + l(\underline{\eta}) - 1}}
\frac{\prod_i \eta_i}{m!}
\blangle \varnothing \big| \underline{\eta} \brangle^{(R_1 \times E_2) / (E_1 \times E_2), \bullet}_{g_1, k}
\blangle \varnothing \big| \underline{\eta}^{\vee} \brangle^{(E_1 \times R_2) / (E_1 \times E_2), \bullet}_{g_2, k}
\end{equation}
where $\eta_1, \ldots, \eta_m$ run over all positive integers summing up to $k$,
the $\ell_1, \ldots, \ell_m$ run over all diagonal splittings in the weighted partition
\[ \underline{\eta} = (\eta_i, \Delta_{E_1 \times E_2, \ell_i})_{i=1}^m, \quad 
 \underline{\eta}^{\vee} = (\eta_i, \Delta^{\vee}_{E_1 \times E_2, \ell_i})_{i=1}^m,
\]
and the sum is over those disconnected stable maps on each sides which yield a connected
domain after gluing (the bullet $\bullet$ reminds us of the disconnected invariants);
moreover we have used
\begin{multline*}
\blangle \varnothing \big| \underline{\eta} \brangle^{(R_1 \times E_2) / (E_1 \times E_2), \bullet}_{g_1, k} \\
=
\sum_{\pi_{\ast} \beta = k}
\blangle \varnothing \big| \underline{\eta} \brangle^{(R_1 \times E_2) / (E_1 \times E_2), \bullet}_{g_1, \beta}
q_1^{W_1^{(R_1)} \cdot \beta} q_2^{W_2^{(E_2)} \cdot \beta} \exp( \mathbf{z}_1 \cdot \beta )
\end{multline*}
where we use the Gromov--Witten bracket notation on the right side and
\[ W_1^{(R_1)}, W_2^{(E_2)} \in H^2(R_1 \times E_2) \]
are the pullbacks of $W_1 \in H^2(R_1)$ and the point class $[0] \in H^2(E)$
respectively.
The definition of the second factor in \eqref{dfsdfds} is parallel.

We will show
\begin{equation} \label{dfsdfdsrr4444}
\blangle \varnothing \big| \underline{\eta} \brangle^{(R_1 \times E_2) / (E_1 \times E_2), \bullet}_{g_1, k}
\in \frac{1}{\Delta(q_1)^{k/2}} \QJac_{\frac{k}{2} Q_{E_8}}^{(q_1, \mathbf{z}_1)} \otimes \QMod^{(q_2)}.
\end{equation}
By an induction argument it is enough to prove the statement for connected Gromov--Witten invariants.
Let us write
\[ \underline{\eta} = ( \eta_i, c_i \otimes d_i )_{i=1}^{m}, \quad c_i \in H^\ast(E_1), d_i \in H^{\ast}(E_2) \]
Then the relative product formula \cite{LQ} yields
\[
\blangle \varnothing \, \big| \, \underline{\eta} \brangle^{(R_1 \times E_2) / (E_1 \times E_2)}_{g_1, k}
=
\int_{\Mbar_{g,m}}
p_{\ast} \CC^{R_1/E_1}_{g_1,k}\big( \varnothing \, ; (\eta_i, c_i)_{i} \big) \cdot \CC^{E_2}_{g_1}(d_1, \ldots, d_m).
\]
where $p$ is the forgetful map to $\Mbar_{g,m}$.
By \cite{Jtaut,HAE} the class $\CC^{E_2}_{g_1}(d_1, \ldots, d_m)$
is a linear combination of tautological classes with coefficients that are quasi-modular forms.
Using Theorem~\ref{theorem_RES1} and Proposition~\ref{DEGENERATIONPROP} we obtain \eqref{dfsdfdsrr4444}.

By an identical argument for $E_1 \times R_2$ we conclude that
\[
\F_{g,k} \in
\frac{1}{\Delta(q_1)^{k/2}} \QJac_{\frac{k}{2} Q_{E_8}}^{(q_1, \mathbf{z}_1)}
\otimes
\frac{1}{\Delta(q_2)^{k/2}} \QJac_{\frac{k}{2} Q_{E_8}}^{(q_2, \mathbf{z}_2)}. \qedhere
\]
\end{proof}

\subsection{Proof of Theorem~\ref{Thm1}}
We first show that the classes $\CC^{\pi_2}_{g,\alpha}()$ satisfy the holomorphic anomaly equation numerically, i.e. after taking degrees.
Using the degeneration \eqref{DEGGGG}
and the compatibility of the holomorphic anomaly equation with the degeneration formula (Proposition~\ref{Proposition_Compatibility_with_degeneration_formula})
the holomorphic anomaly equation for $\int \CC^{\pi_2}_{g,\alpha}$ follows from the
holomorphic anomaly equations for the elliptic fibrations
\[ \mathrm{pr}_1 : R_1 \times E_2 \to R_1, \quad \quad \mathrm{id}_{E_1} \times p_2 : E_1 \times R_2 \to E_1 \times \p^1. \]
relative to $E_1 \times E_2$. 
To show the holomorphic anomaly equation for $R_1 \times E_2$ (relative to $E_1 \times E_2$)
we again apply the product formula \cite{LQ} and use the holomorphic anomaly equation for the elliptic curve \cite{HAE}.
For $E_1 \times R_2$ we apply the product formula and Theorem~\ref{theorem_RES2}.
Hence $\CC^{\pi_2}_{g,\alpha}()$ satisfies the holomorphic anomaly equation numerically.

From Lemma*~\ref{Lemma_EllHAE} after numerical specialization it follows that 
\[ \F_{g,\alpha} \in \bigcap_{\lambda \in E_8^{(2)}} \mathrm{Ker}( \T_{\lambda} ) \]
or equivalently, that $\F_{g,\alpha}$ satisfies the elliptic transformation law.\footnote{
Since $\F_{g,\alpha}$ is invariant under translation by sections of $\pi_2$
this also follows from Section~\ref{Section_The_elliptic_transformation_law}.
}
By \eqref{4tsfgsdf} and since $\F_{g,k}$ is symmetric under exchanging $(\mathbf{z}_1, q_1)$ and $(\mathbf{z}_2, q_2)$ we obtain
%
\[ \F_{g,k} \in \bigcap_{\lambda_1 \in E_8^{(1)}, \lambda_2 \in E_8^{(2)} }
\mathrm{Ker}\left( \mathsf{T}_{\lambda_1} \otimes \mathsf{T}_{\lambda_2} \right).
\]
Similarly, the series $\F_{g,k}$ is invariant under reflection along the elliptic fibers of $\pi_1$ and $\pi_2$.
Since every reflection along a root can be written as a composition of translation and reflection at the origin,
we conclude that
\[ \F_{g,k} \in 
 \frac{1}{\Delta(q_1)^{k/2}} \widetilde{\Jac}_{E_8,k}^{(q_1, \mathbf{z}_1)}
 \otimes \frac{1}{\Delta(q_2)^{k/2}} \widetilde{\Jac}_{E_8, k}^{(q_2, \mathbf{z}_2)}.
\]
Finally, the weight of the bi-quasi-Jacobi form follows from the holomorphic anomaly equation, see Section~\ref{Subsection_weight_refinement} and \cite[Sec.2.6]{HAE}. \qed

\subsection{Proof of Theorem~\ref{Thm2}}
Assume first $g > 2$ or $k > 0$.
Using \eqref{4tsfgsdf} and Proposition*~\ref{Prop_HAE_for_CY3}\footnote{
In the proof of Theorem~\ref{Thm1} 
we have shown that Conjectures~\ref{Conj_Quasimodularity} and~\ref{Conj_HAE}
hold for the Schoen Calabi--Yau numerically. Hence
we may apply Proposition*~\ref{Prop_HAE_for_CY3} unconditionally.}
we find
\begin{multline} \label{3dfsdfa}
\frac{d}{d C_2(q_2)} \F_{g,\kk} \\
=
\sum_{p_{1 \ast} \alpha = k} q_1^{W_1 \cdot \alpha} \zeta_1^{\alpha}
\Bigg[
\langle K_{R_1} + \alpha, \alpha \rangle \F_{g-1, \alpha} +
\sum_{\substack{g=g_1 + g_2 \\ \alpha = \alpha_1 + \alpha_2}} \langle \alpha_1, \alpha_2 \rangle \F_{g_1, \alpha_1} \F_{g_2, \alpha_2}
\Bigg].
\end{multline}
We analyze the terms on the right side.
If we write $\alpha = k W + d F + \alpha_0$ for some $d \geq 0$ and $\alpha_0 \in E_8^{(1)}$ then we have
\[ \langle \alpha, \alpha \rangle = 2 k d + \langle \alpha_0, \alpha_0 \rangle, \quad \langle K_{R_1}, \alpha \rangle = -k.  \]
Hence the first term in the bracket on the right of \eqref{3dfsdfa} can be written as
\begin{multline*}
\sum_{p_{1 \ast} \alpha = k} q_1^{W_1 \cdot \alpha} \zeta_1^{\alpha}
\langle K_{R_1} + \alpha, \alpha \rangle \F_{g-1, \alpha} \\
=
\left( -k + 2 k D_{q_1} - \sum_{i,j=1}^{8} \left( Q_{E_8}^{-1} \right)_{ij} D_{\mathbf{z}_{1,i}} D_{\mathbf{z}_{1,j}} \right) \F_{g-1, k}.
\end{multline*}
With a similar argument the sum
\[
\sum_{p_{1 \ast} \alpha = k} q_1^{W_1 \cdot \alpha} \zeta_1^{\alpha}
\sum_{\substack{g=g_1 + g_2, \alpha = \alpha_1 + \alpha_2 \\
\forall i \in \{ 1, 2 \} \colon g_i \geq 2 \text{ or } p_{1 \ast} \alpha_i > 0}}
\langle \alpha_1, \alpha_2 \rangle \F_{g_1, \alpha_1} \F_{g_2, \alpha_2}
\]
yields exactly the second term on the right in Theorem~\ref{Thm2}.
Using Lemma~\ref{Lemma_Calculation_of_fiber} below, the remaining terms are
\begin{align*}
& 2 \sum_{p_{1 \ast} \alpha = k} q_1^{W_1 \cdot \alpha} \zeta_1^{\alpha}
\sum_{\substack{ g' \in \{ 0, 1\} \\ \ell \geq 1 }} \langle \alpha - \ell F_1, \ell F_1 \rangle \F_{g-g', \alpha - \ell F_1} \F_{g', \ell F_1} \\
=\ & 2 \sum_{p_{1 \ast} \alpha = k} q_1^{W_1 \cdot \alpha} \zeta_1^{\alpha} \sum_{\ell \geq 1} k \ell \cdot \F_{g-1, \alpha - \ell F_1} \cdot 12 \frac{\sigma(\ell)}{\ell} \\
=\ & \left( 24 k \sum_{\ell \geq 1} \sigma(\ell) q_1^{\ell} \right) \F_{g-1, k}. \qedhere
\end{align*}
Putting all three expressions together yields the desired expression. 

Finally, if $g=2$ and $k=0$ a similar analysis shows
\[ 
\pushQED{\qed} 
\frac{d}{d C_2(q_2)} \F_{2,0} = 0. \qedhere 
\popQED
\]

\begin{lemma} \label{Lemma_Calculation_of_fiber} For all $(\ell_1, \ell_2) \neq 0$ we have
\[
\mathsf{N}^X_{g, \ell_1 F_1 + \ell_2 F_2}
=
\begin{cases}
12 \delta_{\ell_1 0} \frac{\sigma(\ell_2)}{\ell_2} + 12 \delta_{\ell_2 0} \frac{\sigma(\ell_1)}{\ell_1} & \text{ if } g = 1 \\
0 & \text{ if } g \neq 1.
\end{cases}
\]
\end{lemma}
\begin{proof}
Using the degeneration \eqref{DEGGGG} we have
\begin{align*}
\mathsf{N}^X_{g, \ell_1 F_1 + \ell_2 F_2}
& = 
\blangle \varnothing \big| \varnothing \brangle^{(R_1 \times E_2) / (E_1 \times E_2)}_{g, \ell_1 F_1 + \ell_2 F_2}
+ \blangle \varnothing \big| \varnothing \brangle^{(E_1 \times R_2) / (E_1 \times E_2)}_{g, \ell_1 F_1 + \ell_2 F_2}.
\end{align*}
Because the surface $E_1 \times E_2$ carries a holomorphic symplectic form,
all Gromov--Witten invariants of $\p^1 \times E_1 \times E_2$
with non-trivial curve degree over $E_1 \times E_2$ vanish. Hence by a degeneration argument we have
\[ 
\blangle \varnothing \big| \varnothing \brangle^{(R_1 \times E_2) / (E_1 \times E_2)}_{g, \ell_1 F_1 + \ell_2 F_2}
= \blangle \varnothing \brangle^{R_1 \times E_2}_{g, \ell_1 F_1 + \ell_2 F_2}.
\]
The expression for the second term is parallel.
%
Now the result follows by adding in markings, using the divisor equation and applying the product formula.
\end{proof}

\subsection{Proof of Corollary~\ref{Cor_Schoen}}
Since the series $\F_{g,\alpha}$ satisfies the holomorphic anomaly equation,
the disconnected series $F^{\bullet}_{g,\alpha}$ satisfies \eqref{352rw}.
The claim now follows from Lemma~\ref{Lemma_QJAc_ModTrLaw}.

\section{Abelian surfaces} \label{Section_Abelian_surfaces}
\subsection{Overview}
We present (Section~\ref{Subsection_ab_statement}) and prove numerically
(Section~\ref{Subsection_ab_proof}) the holomorphic anomaly equation
for the reduced Gromov--Witten theory of abelian surfaces in primitive classes.
The quasi-modularity of the theory was proven previously in \cite{BOPY}.
The result and strategy of proof is almost identical to the
case of K3 surfaces which appeared in detail in \cite[Sec.0.6]{HAE} and we will be brief.
Since we work with reduced Gromov--Witten theory,
an additional term appears in the holomorphic anomaly equation
for both abelian and K3 surfaces.
This term appeared somewhat mysteriously in \cite{HAE} in the form of a certain operator $\sigma$.
In Section~\ref{Subsection_ab_explain} we explain how it arises
naturally from the theory of quasi-Jacobi forms.

\subsection{Results}\label{Subsection_ab_statement}
Let $E_1, E_2$ be non-singular elliptic curves and consider the abelian surface
\[ \A = E_1 \times E_2 \]
elliptically fibered over $E_1$ via the projection $\pi$ to the first factor,
\[ \pi : \A\to E_1. \]
Let $0_{E_2} \in E_2$ be the zero and fix the section
\[ \iota : E_1 = E_1 \times 0_{E_2} \hookrightarrow \A. \]
A pair of integers $(d_1,d_2)$ determines a class in $H_2(\A,\BZ)$ by
\[ (d_1, d_2) = d_1 \iota_{\ast} [E_1] + d_2 j_{\ast} [E_2] \]
where $j : 0_{E_1} \times E_2 \to \A$ is the inclusion.

Since $\A$ carries a holomorphic symplectic form, the virtual fundamental class of
$\Mbar_{g,n}(\A, \beta)$ vanishes if $\beta \neq 0$. A nontrivial
Gromov--Witten theory of $\A$ is defined by the \emph{reduced} virtual class
$[ \Mbar_{g,n}(\A, \beta) ]^{\text{red}}$, see \cite{BOPY} for details.
For any $\gamma_1, \ldots, \gamma_n \in H^{\ast}(\A)$ define the reduced primitive potential
\begin{multline*} \CA_{g}(\gamma_1, \ldots, \gamma_n)
 = \sum_{d=0}^{\infty} q^d \pi_{\ast}\left( \left[ \Mbar_{g,n}\left(\A, (1,d) \right) \right]^{\text{red}} \prod_{i=1}^{n} \ev_i^{\ast}(\gamma_i) \right) \\
 \in H_{\ast}(\Mbar_{g,n}(E_1, 1))[[q]].
\end{multline*}
By deformation invariance the classes $\CA_g$ determine the Gromov--Witten classes
of any abelian surface in primitive classes.

\begin{conjecture} \label{Conj_Quasimodularity_A}
$\CA_{g,n}(\gamma_1, \ldots, \gamma_n) \in \QMod \otimes H_{\ast}( \Mbar_{g,n}(E_1, 1) )$
\end{conjecture}

We state the reduced holomorphic anomaly equation.
For any $\lambda \in H^{\ast}(\A)$ define the endomorphism $A(\lambda) : H^{\ast}(\A) \to H^{\ast}(\A)$ by
\[ A(\lambda) \gamma = \lambda \cup \pi^{\ast} \pi_{\ast}( \gamma) - \pi^{\ast} \pi_{\ast}( \lambda \cup \gamma ) \quad
\text{ for all } \gamma \in H^{\ast}(\A). \]
Define the operator $T_{\lambda}$ by\footnote{
The notation $T_{\lambda}$ (serif) matches the expected value of the action
of the anomaly operator $\T_{\lambda}$ (sans-serif) given in Lemma*~\ref{Lemma_EllHAE}.
The operator $T_{\lambda}$ is defined independently of the modular properties of $\CA$.}
\[ T_{\lambda} \CA_{g,n}(\gamma_1, \ldots, \gamma_n)
 = \sum_{i=1}^{n} \CA_{g,n}( \gamma_1,\ldots, A(\lambda) \gamma_i, \ldots, \gamma_n).
\]
Let $V \subset H^{2}(\A, \BQ)$ be the orthogonal complement to $[E_1], [E_2]$
and define
\begin{equation} T_{\Delta} = - \sum_{i,j=1}^{4} \left( G^{-1} \right)_{ij} T_{b_i} T_{b_j} \label{TDelta} \end{equation}
where $\{ b_i \}$ is a basis of $V$ and $G = \big( \langle b_i, b_j \rangle \big)_{i,j}$. 

Recall also the virtual class on the moduli space of degree $0$,
\[
[ \Mbar_{g,n}(\A,0) ]^{\text{vir}}
=
\begin{cases}
[\Mbar_{0,n} \times \A] & \text{if } g= 0 \\
0 & \text{if } g \geq 1,
\end{cases}
\]
where we used the identification $\Mbar_{g,n}(\A,0) = \Mbar_{g,n} \times \A$.
We define
\[ \CA_g^{\vir}(\gamma_1, \ldots, \gamma_n)
=
\pi_{\ast}\left( [ \Mbar_{g,n}(\CA, 0) ]^\vir \prod_i \ev_i^{\ast}(\gamma_i) \right). \]

Consider the class in $H_{\ast}(\Mbar_{g,n}(E_1,1))$ defined by
\begin{equation} \label{HAE_for_A}
\begin{aligned}
\HH^{\CA}_{g}(\gamma_1, \ldots, \gamma_n)
& =
\iota_{\ast} \Delta^{!} \CA_{g-1}( \gamma_1, \ldots, \gamma_n, \1, \1 ) \\
& + 2 \sum_{\substack{g= g_1 + g_2 \\ \{1,\ldots, n\} = S_1 \sqcup S_2}}
j_{\ast} \Delta^{!} \left( \CA_{g_1}( \gamma_{S_1}, \1 ) \boxtimes \CA^{\vir}_{g_2}( \gamma_{S_2}, \1 ) \right) \\
& - 2 \sum_{i=1}^{n} \CA_g( \gamma_1, \ldots, \gamma_{i-1}, \pi^{\ast} \pi_{\ast} \gamma_i, \gamma_{i+1}, \ldots, \gamma_n ) \cup \psi_i \\
& + T_{\Delta} \CA_{g}(\gamma_1, \ldots, \gamma_n)
\end{aligned}
\end{equation}

\begin{conjecture} \label{Conj_HAE_A}
$\frac{d}{dC_2} \CA_{g}(\gamma_1, \ldots, \gamma_n) = \HH^{\CA}_{g}(\gamma_1, \ldots, \gamma_n)$.
\end{conjecture}

Let $p : \Mbar_{g,n}(E_1,1) \to \Mbar_{g,n}$
be the forgetful map, and recall the tautological subring
$R^{\ast}(\Mbar_{g,n}) \subset H^{\ast}(\Mbar_{g,n})$.
In the unstable cases we will use the convention of Section~\ref{Subsection_CONVENTION}.
By \cite{BOPY} Conjecture~\ref{Conj_Quasimodularity_A} holds numerically:
\begin{equation} \label{454353}
\int_{\Mbar_{g,n}(E_1,1)} p^{\ast}(\alpha) \cap \CA_g(\gamma_1, \ldots, \gamma_n) \, \in \QMod
\end{equation}
for all tautological classes $\alpha \in R^{\ast}(\Mbar_{g,n})$.
We show the holomorphic anomaly equation holds numerically as well.
\begin{thm} \label{thm_AHAE}
For any tautological class $\alpha \in R^{\ast}(\Mbar_{g,n})$,
\[
\frac{d}{dC_2} \int p^{\ast}(\alpha) \cap \CA_g(\gamma_1, \ldots, \gamma_n)
=
\int p^{\ast}(\alpha) \cap \HH^{\CA}_{g}(\gamma_1, \ldots, \gamma_n).
\]
\end{thm}

\subsection{Discussion of the anomaly equation} \label{Subsection_ab_explain}
The holomorphic anomaly equation for abelian and K3 surfaces (see \cite{HAE}) require two modifications
to Conjecture~\ref{Conj_HAE}.
The first is the modified splitting term (the second term on the right-hand side of \eqref{HAE_for_A}).
It arises naturally from the formula for the restriction of the reduced virtual class $[ \, \cdot \, ]^{\text{red}}$
to boundary components, see e.g. \cite[Sec.7.3]{MPT}.

The second modification in \eqref{HAE_for_A} is the term $T_{\Delta} \CA_{g}(\gamma_1, \ldots, \gamma_n)$ which appears for K3 surfaces
in \cite[Sec.0.6]{HAE} in its explicit form.
To explain its origin we consider the difference in definition of the Gromov--Witten potentials $\CC^{\pi}_{g,\kk}$ and $\CA$.
The class $\CC_{g,\kk}^{\pi}$ is defined by summing over all classes $\beta$ on $X$ which are of degree $\kk$ over the base,
while for $\CA$ we fix the base class $[E_1]$ and sum over the fiber direction $[E_1] + d [E_2]$.
The latter
corresponds to taking the $\zeta^{0}$-coefficient of the quasi-Jacobi form $\CC_{g,\kk}^{\pi}$.
By Proposition~\ref{QJac_Prop2} the $C_2$-derivative of this $\zeta^{0}$-coefficient then naturally acquires an extra term
which exactly matches $T_{\Delta} \CA_g$.

To make the discussion more concrete consider a rational elliptic surface $\pi : R \to \p^1$ and
consider the $\zeta^{0}$-coefficient of the class $\CC_{g,k=1}$,
\[ \CR_{g}(\gamma_1, \ldots, \gamma_n) = \left[ \CC^{\pi}_{g,1}(\gamma_1, \ldots, \gamma_n) \right]_{\zeta^0}. \]
The class $\CR_g$ should roughly correspond to the classes $\CA_g$ for abelian
and $\CK_g$ for K3 surfaces\footnote{The classes $\CK_g$ are the analogues of $\CA_g$ for K3 surfaces, see \cite[Sec.1.6]{HAE} for a definition.}.
Assuming Conjecture~\ref{Conj_Quasimodularity} and using Section~\ref{Subsubsection_Halfunimodular_index}
we find $\CR_g$ is a cycle-valued $\mathrm{SL}_2(\BZ)$-quasi-modular form.
Assuming Conjecture~\ref{Conj_HAE} and using Proposition~\ref{QJac_Prop2} then yields the holomorphic anomaly equation
\begin{align*}
\frac{d}{dC_2} \CR_{g}(\gamma_1, \ldots, \gamma_n)
= \,
& \iota_{\ast} \Delta^{!} \CR_{g-1}( \gamma_1, \ldots, \gamma_n, \1, \1 ) \\
& + 2 \sum_{\substack{g= g_1 + g_2 \\ \{1,\ldots, n\} = S_1 \sqcup S_2}}
j_{\ast} \Delta^{!} \left( \CR_{g_1}( \gamma_{S_1}, \1 ) \boxtimes \CC^{\pi}_{g_2,0}( \gamma_{S_2}, \1 ) \right) \\
& - 2 \sum_{i=1}^{n} \CR_g( \gamma_1, \ldots, \gamma_{i-1}, \pi^{\ast} \pi_{\ast} \gamma_i, \gamma_{i+1}, \ldots, \gamma_n ) \cup \psi_i \\
& + T_{\Delta} \CR_{g}(\gamma_1, \ldots, \gamma_n)
\end{align*}
where the operator $T_{\Delta}$ is defined as in \eqref{TDelta}
but with $V$ replaced by $H^2_{\perp}$.
Hence we recover the same term as for abelian and K3 surfaces.


\subsection{Proof of Theorem~\ref{thm_AHAE}} \label{Subsection_ab_proof}
The quasi-modularity \eqref{454353} was proven in \cite{BOPY}
by an effective calculation scheme using the following ingredients:
(i) an abelian vanishing equation, (ii) tautological relations / restriction to boundary,
(iii) divisor equation, (iv) degeneration to the normal cone of an elliptic fiber.
One checks that each such step is compatible with the holomorphic anomaly equation.
For the K3 surface this was done in detail in \cite{HAE} and the abelian surface case is parallel. \qed

\appendix
\section{Cohomological field theories}
\label{Section_CohFTs}
\subsection{Introduction}
A cohomological field theory (CohFT) $\Omega$ is a collection of classes
\[ \Omega_{g,n}(v_1,\ldots,v_n) \in H^*(\Mbar_{g,n},A) \]
satisfying certain splitting axioms with respect to the boundary divisors of $\Mbar_{g,n}$
(see \cite{Jan} for an introduction).
Here the CohFT has coefficients in some commutative $\BQ$-algebra $A$.
Pushing forward the Gromov-Witten virtual class (after capping with classes pulled back from the target space) is one of the main ways of constructing
cohomological field theories.

There are two important group actions on CohFTs. The first is by the automorphism group $\Aut(A)$ of the coefficient ring $A$. The second is Givental's $R$-matrix action, which involves the boundary geometry of $\Mbar_{g,n}$. Teleman \cite{Tel} proved that for semisimple CohFTs, any two CohFTs with the same values on $\Mbar_{0,3}$ are related by the action of a unique $R$-matrix. This has the following consequence relating the two actions. Suppose that $\Omega$ is a CohFT and $\phi\in\Aut(A)$ is an automorphism fixing $\Omega_{0,3}$. Then there must exist a corresponding $R$-matrix taking $\Omega$ to $\phi(\Omega)$ under Givental's action. For non-semisimple theories, such a correspondence may still exist but is not guaranteed.

Now suppose that $D$ is a derivation of $A$ and we are interested in a formula for $D(\Omega)$. In this case, $\exp(tD)$ is an automorphism of $A[[t]]$, so we may ask whether $\Omega$ and $\exp(tD)(\Omega)$ are related by some $R$-matrix. If they are, then taking the linear part of Givental's $R$-matrix action gives a formula for $D(\Omega)$. In other words, derivations of the coefficient ring correspond sometimes to a linearization of the $R$-matrix action.

In this appendix we will apply this perspective to the holomorphic anomaly equations conjectured in this paper. Things are more difficult than in the discussion above because the $\pi$-relative Gromov-Witten generating series $\CC^{\pi}_{g, \kk}$ discussed in this paper is not quite a CohFT (as it takes values in $H_*(\Mbar_{g,n}(B,\kk))$, not in $H^*(\Mbar_{g,n})$). In Section~\ref{app:weak} we address this issue by defining weak $B$-valued field theories, and then in Section~\ref{app:R-matrix} we define an (infinitesimal) $R$-matrix action on these theories. In Section~\ref{app:cocycle} we describe how our conjectured holomorphic anomaly equations can be expressed via a function from the Jacobi Lie algebra to the space of $R$-matrices satisfying a cocycle condition.

\subsection{Weak $B$-valued field theories}\label{app:weak}
Let $B$ be a non-singular projective variety. For convenience, let $H = H^*(B,\BQ)$. Let $V$ be a finitely generated $H$-module with a perfect\footnote{
By Poincar\'e duality of $B$ 
the pairing $\eta$ is perfect if and only if the $\BQ$-valued pairing $\int_B \eta$ is perfect.}
pairing of $H$-modules $\eta: V\times V\to H$ and a distinguished element $\unit\in V$. Let $A$ be a commutative $\BQ$-algebra. Then a \emph{weak $B$-valued field theory}\footnote{
The word ``weak" in this name refers to the fact that we only use a single boundary divisor in condition (iv) of the definition. The analogous condition in the definition of a cohomological field theory uses all boundary divisors of $\Mbar_{g,n}$.} $\Omega$ on $(V,\eta,\unit)$ with coefficients in $A$ is a collection of maps
\[
\Omega_{g,n}^\kk: V^{\otimes n} \to H_*(\Mbar_{g,n}(B,\kk)) \otimes A
\]
(all tensor products taken over $\BQ$ unless otherwise stated) defined for all $g,n\ge 0$ and $\kk\in H_2(B,\BZ)$ with $2g-2+n > 0$ or $\kk > 0$, satisfying the following four conditions:
\begin{enumerate}
\item[(i)] Each map $\Omega_{g,n}^\kk$ is $H^n$-equivariant, where the $i$-th copy of $H$ acts on the $i$-th factor of $V^{\otimes n}$ and by pulling back classes to $\Mbar_{g,n}(B,\kk)$ using the evaluation map at the $i$-th marked point.
\item[(ii)] Each map $\Omega_{g,n}^\kk$ is $S_n$-equivariant, where $S_n$ acts by permuting the factors of $V^{\otimes n}$ and permuting the labels of marked points in $\Mbar_{g,n}(B,\kk)$.
\item[(iii)] For any classes $v,w\in V$,
\[
\Omega_{0,3}^0(\unit,v,w) = \eta(v,w)
\]
under the isomorphism $H_*(\Mbar_{0,3}(B,0))\otimes A \cong H \otimes A$.
\end{enumerate}

For the fourth condition, we will need two further definitions. First, define the quantum product $*$ on $V\otimes A$ by the property
\[
\Omega_{0,3}^0(u,v,w) = \eta(u*v,w).
\]
Second, suppose that $\iota:\Mbar_{g,n+1}\to\Mbar_{g,n+2}$ is defined by replacing the marked point $p_{n+1}$ by a rational bubble containing two marked points $p_{n+1},p_{n+2}$. Let $F$ be the fiber product of this map and the forgetful map $\Mbar_{g,n+2}(B,\kk)\to\Mbar_{g,n+2}$. One connected component of $F$ is naturally isomorphic to $\Mbar_{g,n+1}(B,\kk)$. Given any class $\alpha\in H_*(\Mbar_{g,n+2}(B,\kk)$, let $\iota^\sharp\alpha\in H_*(\Mbar_{g,n+1}(B,\kk))$ be the restriction of $\iota^!\alpha$ to this component. Then our fourth condition is:
\begin{enumerate}
\item[(iv)] For any $g,n,\kk$, and $v_1,\ldots,v_{n+2}$,
\[
\iota^\sharp\Omega_{g,n+2}^\kk(v_1,\ldots, v_{n+2}) = \Omega_{g,n+1}^\kk(v_1,\ldots, v_n,v_{n+1}*v_{n+2}).
\]
\end{enumerate}

It is straightforward to check that the $\pi$-relative Gromov-Witten generating series $\CC^{\pi}_{g, \kk}$ discussed in this paper forms a weak $B$-valued field theory on $(H^*(X,\BQ),\eta,1)$ with coefficients in $\BQ[[q^{\frac{1}{2}},\zeta]]$, where the pairing is given by $\eta(\alpha,\beta) := \pi_*(\alpha\beta)$. If we assume Conjecture~\ref{Conj_Quasimodularity} then we may take the coefficient ring $A$ to be the algebra $\QJac[\Delta^{-1/2}]$.

\subsection{Matrix actions}\label{app:R-matrix}

In this section, we define a matrix action on weak $B$-valued field theories that should be viewed as an infinitesimal analogue of Givental's $R$-matrix action on cohomological field theories. Fix the data $(V,\eta,\unit)$ and the coefficient ring $A$ as before. Let $\CR(V,\eta)$ be the (associative) algebra of formal Laurent series
\[
M = \, \ldots + M_{-1}z^{-1} + M_0 + M_1 z + \ldots,
\]
where $M_i$ is an element of $V\otimes_H V$ for $i \ge 0$ and an element of $\End(V) = \Hom_H(V,V)$ for $i < 0$ (and vanishes for all $i$ sufficiently negative). The multiplication on $\CR(V,\eta)$ is defined by contraction by the pairing $\eta:V\otimes_H V\to H$ along with the homomorphism
\[
V\otimes_H V \to \End(V)
\]
defined by $(a\otimes b)(v) = \eta(b,v)a$.

Let $M$ be an element of $\CR(V,\eta)\otimes A$ satisfying the following two conditions:
\begin{enumerate}
\item[(a)] Let $M_+\in V\otimes_H V[[z]]\otimes A$ be the part of $M$ with nonnegative powers of $z$. Then we require that
\[
M_+(z) + M_+^t(-z) = 0,
\]
where $M_+^t$ is defined by interchanging the two copies of $V$ in $V\otimes_H V$.

\item[(b)] The principal part of $M$ is of the form
\[
M - M_+ = m_v z^{-1},
\]
where 
$v\in V\otimes A$ and 
$m_v\in \End(V)\otimes A$ is the operator of quantum multiplication by $v$.
\end{enumerate}

Given a weak $B$-valued theory $\Omega$ on the above data, we define new maps
\[
(r_M\Omega)^{\kk}_{g,n}: V^{\otimes n}\to H_*(\Mbar_{g,n}(B,\kk)) \otimes A
\]
by
\begin{align*}
& (r_M\Omega)_{g,n}^\kk(v_1,\ldots,v_n) \\
:= \  & -\frac{1}{2}\iota_{\ast}\Delta^!
\Omega_{g-1,n+2}^\kk(v_1,\ldots,v_n,\CE) \\
&
-\frac{1}{2}\sum_{\substack{ g= g_1 + g_2 \\
\{1,\ldots, n\} = S_1 \sqcup S_2 \\
\mathsf{k} = \mathsf{k}_1 + \mathsf{k}_2}}
j_\ast\Delta^!\left(
\Omega_{g_1,|S_1|+1}^{\kk_1}(v_{S_1},\CE^{(1)}) \boxtimes
\Omega_{g_2,|S_2|+1}^{\kk_2}(v_{S_2},\CE^{(2)}) \right) \\
&
+ \sum_{i=1}^n
\Omega_{g,n}^\kk(v_1,\ldots,v_{i-1},M_+ v_i,v_{i+1},\ldots,v_n) \\
&
- p_*\Omega_{g,n+1}^\kk(v_1,\ldots,v_n,zM\unit),
\end{align*}
where $\CE$ is any lift of
\[
\frac{M_+(z) + M_+^t(z')}{z+z'} \in (V\otimes_H V)[[z,z']]\otimes A
\]
to $(V\otimes_\BQ V)[[z,z']]\otimes A$, we are using notation as in Conjecture~\ref{Conj_HAE}, and all $z$ variables should be replaced by capping with the corresponding $\psi$ classes.

We make some comments on $r_M\Omega$:

1) If $\Omega$ is a weak $B$-valued field theory with coefficients in $A$, then $\Omega + t\cdot r_M\Omega$ is a weak $B$-valued field theory with coefficients in $A[t]/t^2$.

2) Our main holomorphic anomaly equation, Conjecture~\ref{Conj_HAE}, can be restated as saying that
\[
\T_q\CC^\pi = r_{-2(1\otimes 1)z}\CC^\pi,
\]
where $\T_q$ is the derivation defined in Section~\ref{Subsubsection_quasiJacobi_Differentialoperators} on the coefficient ring $A = \QJac[\Delta^{-1/2}]$.

3) If $M = m_v z^{-1}$ for some $v\in V\otimes A$, then $M_+ = 0$ and the definition above simplifies to
\[
(r_{m_v z^{-1}}\Omega)_{g,n}^\kk(v_1,\ldots,v_n) = - p_*\Omega_{g,n+1}^\kk(v_1,\ldots,v_n,v).
\]
Then the divisor equation says that
\[
D_q\CC^\pi = r_{-m_W z^{-1}}\CC^\pi, \quad D_\lambda\CC^\pi = r_{-m_\lambda z^{-1}}\CC^\pi.
\]

\subsection{The derivation-matrix correspondence}\label{app:cocycle}

The derivations $\T_q,D_q,D_\lambda$ on $\QJac$ generate the Jacobi Lie algebra.
We have seen above that 
the action of each of these derivations on $\CC^\pi$ is given by some matrix action $r_M$.
The following general result extends this to the entire Jacobi Lie algebra.

\begin{prop}\label{app:commutator}
Let $\Omega$ be a weak $B$-valued theory on $(V,\eta,\unit)$ with coefficients in $A$. Suppose that $D_1,D_2$ are $\BQ$-linear derivations of $A$ and $M_1,M_2\in \CR(V,\eta)\otimes A$ satisfy the conditions (a), (b) used to define $r_{M_1}\Omega,r_{M_2}\Omega$ above. If
\[
D_i\Omega = r_{M_i}\Omega
\]
for $i = 1,2$, then
\[
[D_1,D_2]\Omega = r_{[M_1,M_2] + D_1(M_2)-D_2(M_1)}\Omega.
\]
\end{prop}
\begin{proof}[Sketch of proof]
We can compute $D_1 D_2\Omega = D_1 r_{M_2}\Omega$ by applying the derivation $D_1$ to the definition of $r_{M_2}\Omega$, then replacing $D_1\Omega$ in the result with $r_{M_1}\Omega$, and finally expanding $r_{M_1}\Omega$ using its definition. Repeating this procedure for $D_2 D_1\Omega$ and taking the difference, most terms cancel. The non-canceling terms come from several different sources (applying $D_i$ to the coefficients of $M_j$; $M_1$ and $M_2$ not necessarily commuting; $p^*\psi_i \ne \psi_i$; $\Mbar_{0,2}(B,0)$ being unstable) and sum to the claimed matrix action.
\end{proof}

Assuming Conjecture~\ref{Conj_HAE} and applying this result to $\CC^{\pi}$, we have the following corollary:
\begin{corstar}
Let $J$ be the Jacobi Lie algebra of derivations of $\QJac$ generated by $\T_q,D_q,D_\lambda$. Then there exists a function
\[
f:J\to \CR(V,\eta)\otimes\QJac[\Delta^{-1/2}]
\]
such that
\[
D\CC^\pi = r_{f(D)}\CC^\pi
\]
for all $D\in J$.
\end{corstar}

It is straightforward to compute the function $f$ above from the initial values, the commutator formulas in \eqref{COMMUTATIONRELATIONS}, and the formula in Proposition~\ref{app:commutator}:
\begin{equation}
\begin{alignedat}{2}
f(\T_q) &= -2(1\otimes 1)z^1 &
f(\T_\lambda) &= (\lambda\otimes 1 - 1 \otimes \lambda) \\
f(D_q) &= -m_W z^{-1} &
f(D_\lambda) &= -m_\lambda z^{-1} \\
f\left(-\frac{1}{2}[\T_q,D_q]\right) &= (W\otimes 1 - 1\otimes W) &
\quad f([\T_\lambda,D_\mu]) &= m_{\pi^*\pi_*(\lambda\mu)}z^{-1}.
\end{alignedat}
\end{equation}

Lemma*s \ref{Lemma_Weight_statement} and \ref{Lemma_EllHAE} can be recovered from the values of $f(-\frac{1}{2}[\T_q,D_q])$ and $f(\T_\lambda)$. The fact that the function $f$ satisfies the Lie algebra cocycle condition
\[
f([A,B]) = [f(A),f(B)] + A(f(B)) - B(f(A))
\]
can be viewed as a check on Conjecture~\ref{Conj_HAE}.

\section{K3 fibrations}
\label{Section_K3fibrations}
\subsection{Definition}
The second cohomology of a non-singular projective K3 surface $S$ is
a rank $22$ lattice with intersection form
\[ H^2(S,\BZ) \cong U \oplus U \oplus U \oplus E_8(-1) \oplus E_8(-1) \]
where $U = \binom{0\ 1}{1\ 0}$ is the hyperbolic lattice.
Consider a primitive embedding
\[ \Lambda \subset H^2(S,\BZ) \]
of signature $(1,r-1)$ and let $v_1, \ldots, v_r \in \Lambda$ be an integral basis.

Let $X$ be a non-singular projective variety with line bundles 
\[ L_1, \ldots, L_r \to X \]
A $\Lambda$-polarized K3 fibration is a flat morphism
\[ \pi : X \to B \]
with connected fibers satisfying the following properties\footnote{We refer to \cite{KMPS} for the definition of a $\Lambda$-polarized K3 surface.}:
\begin{enumerate}
\item[(i)] The smooth fibers
$X_{\xi}, \xi \in B$
of $\pi$ are $\Lambda$-polarized
K3 surfaces via \[ v_i \mapsto L_{i}|_{X_{\xi}}. \]
\item[(ii)] There exists a $\lambda \in \Lambda$ which restricts to a quasi-polarization on all smooth fibers of $\pi$ simultaneously.
\end{enumerate}

Given a curve class $\kk \in H_2(B,\BZ)$ and classes $\gamma_1, \ldots, \gamma_n \in H^{\ast}(X)$
we define the $\pi$-relative Gromov--Witten potential
\[
\CC_{g,\kk}^{\pi}(\gamma_1, \ldots, \gamma_n)
=
\sum_{\pi_{\ast} \beta = \kk} q_1^{L_1 \cdot \beta} \cdots q_r^{L_r \cdot \beta}
\pi_{\ast}\left( \left[ \Mbar_{g,n}(X,\beta) \right]^{\text{vir}} \prod_{i=1}^{n} \ev_i^{\ast}(\gamma_i) \right).
\]
where 
$\pi : \Mbar_{g,n}(X,\beta) \to \Mbar_{g,n}(B, \kk)$
is the morphism induced by $\pi$.
%

\vspace{5pt}
\noindent \textbf{Problem.} Find a ring of
quasi-modular objects $\CR \subset \BQ[[ q_1^{\pm 1}, \ldots, q_r^{\pm 1} ]]$
(depending only on $\Lambda$) such that for all $g, \kk$ and $\gamma_1, \ldots, \gamma_n$ we have
\[ \CC_{g,\kk}^{\pi}(\gamma_1, \ldots, \gamma_n) \in H_{\ast}(\Mbar_{g,n}(B,\kk)) \otimes \CR. \]
\vspace{0pt}

By quasi-modular objects we mean here functions of $q_1, \ldots, q_r$ which have modular properties after adding
a dependence on non-holomorphic parameters. We moreover ask the derivative along the
non-holomorphic parameters to induce a derivation on $\CR$.
We expect the classes $\CC_{g,\kk}$ to be govenered by a holomorphic anomaly equation
taking a shape similar to Conjecture~\ref{Conj_HAE}.
We discuss a basic example in the next Section.

\subsection{An example}
The STU model is a particular non-singular projective
Calabi--Yau threefold $X$ which admits a K3 fibration
\[ \pi : X \to \p^1 \]
polarized by the hyperbolic lattice $U$ via line bundles $L_1, L_2 \to X$.
Every smooth fiber $X_{\xi}$ of $\pi$ ($\xi \in \p^1)$ is an elliptic K3 surface with section.
The line bundles $L_i$ restrict to
\[ L_1|_{X_{\xi}} = F, \quad L_2|_{X_{\xi}} = S + F \]
where $S, F \in H^2(X_{\xi}, \BZ)$ are the section and fiber class respectively.

By \cite[Prop.5]{KMPS} we have the following basic evaluation of the $\pi$-relative
potential of $X$:
\begin{equation} \label{54534}
\int \CC_{0,0}(L_2, L_2, L_2)
\, =\, 
2 \frac{E_4(q_1) E_6(q_1)}{\Delta(q_1)} \cdot \frac{E_4(q_2)}{j(q_1) - j(q_2)}
\end{equation}
where $E_k = 1 + O(q)$ are the Eisenstein series and $j(q) = q^{-1} + O(1)$ is the $j$-function, and the expansion on the right-hand side is taken in the region $|q_1|<|q_2|$.
It is hence plausible for $\CR$ to be the ring
(of Laurent expansions in the region $|q_1| < |q_2|$)
of meromorphic functions of $q_1, q_2$
which are quasi-modular in each variable
and have poles only at $q_1 = q_2$ and $q_i = 0$ for $i \in \{1, 2 \}$.
The modularity in each variable on the right-hand side of
\eqref{54534} is in agreement with the expected holomorphic anomaly equation.

\bigskip

\noindent 
MIT, Department of Mathematics \\
E-mail address: georgo@mit.edu \\[5pt]

\noindent 
MIT, Department of Mathematics \\
E-mail address: apixton@mit.edu \\

\end{document}